\documentclass[a4paper]{amsart}
\usepackage[english]{babel}
\usepackage{amsmath}
\usepackage{amssymb}
\usepackage{xpatch}
\usepackage{mathrsfs}
\usepackage{amsthm,xpatch}
\usepackage{stmaryrd}
\usepackage{graphicx}
\usepackage{float}
\usepackage{dsfont}
\usepackage{amsmath}
\usepackage{dsfont}
\usepackage{xcolor}
\usepackage{caption}
\usepackage{subcaption}
\usepackage{bbm, dsfont}
\usepackage{algorithm}
\usepackage{algorithmic}
\usepackage{anysize}
\usepackage{bbold}

\usepackage{euscript}
\usepackage{amssymb}
\usepackage{mathrsfs}
\usepackage{mathrsfs}
\usepackage{float}
\usepackage{amsmath}
\usepackage{dsfont}
\usepackage{xcolor}
\usepackage{caption}
\usepackage{subcaption}
\usepackage{bbm, dsfont}
\usepackage{anysize}
\usepackage{bbold}
\usepackage{fancyhdr,ragged2e}
\usepackage{setspace}

\usepackage{amsfonts,eurosym,amssymb,amsmath,stmaryrd,mathtools,mathrsfs}
\usepackage{algorithm}
\usepackage{algorithmic}
\usepackage{mathtools}
\usepackage{amsmath}
\usepackage{amssymb}
\usepackage{mathrsfs}
\usepackage{amsthm}
\usepackage{graphicx}
\usepackage{color}
\usepackage{colortbl}
\usepackage{verbatim}
\newcommand{\N}{\mathbb{N}}

\newcommand{\R}{\mathbb{R}}

\renewcommand{\P}{\mathbb{P}}
\newcommand{\E}{\mathcal{E}}

\newcommand\procn[1]{\left(#1,\, j\in \llbracket 0,n \rrbracket\right)}
\newcommand\suiten[1]{\left(#1,\, j\in \llbracket 0,n \rrbracket\right)}

\newcommand\pr[1]{{\mathbb P}\left[#1\right]}
\newcommand\prm[1]{{\mathbb P}_\mu\left[#1\right]}

\newcommand\proc[1]{\left(#1,\, t\ge 0\right)}

\newcommand\procun[1]{\left(#1,\, t\in [0,1]\right)}

\newcommand\suite[1]{\left(#1,n\in\N_+\right)}

\newcommand\suitei[1]{\left\{#1\right\}_{i \in \N}}
\def\esp#1{{\mathbb E}\left[#1\right]}

\def\eps{\epsilon}

\newcommand{\M}{\maM}

\def\maN{{\mathcal N}}
\def\bphi{\mathbf{\Phi}}

\newcommand\cro[1]{\left\langle #1 \right\rangle}
\newcommand\llangle{\left\langle\!\left\langle}
\newcommand\rrangle{\right\rangle\!\right\rangle}
\def\d{\operatorname{d}\!}
\def\esp#1{{\mathbb E}\left[#1\right]}

\usepackage{graphicx}
\graphicspath{{./images/}}
\usepackage{pgf,tikz}
\usetikzlibrary{graphs}
\usetikzlibrary{graphs.standard}
\usetikzlibrary{positioning}
\usetikzlibrary{arrows.meta}
\usepackage{ae,aecompl}
\usepackage{eso-pic}	
\usepackage{array}		

\newtheorem{theorem}{Theorem}[section]
\newtheorem{definition}{Definition}[section]
\newtheorem{corollary}{Corollary}[section]
\newtheorem{lemma}{Lemma}[section]
\newtheorem{Ex}{Example}[section]
\newtheorem{proposition}{Proposition}[section]
\newtheorem{hypo}{Assumption}[section]
\newtheorem{remark}{Remark}[section]

\usepackage{caption}
\usepackage{subcaption}
\usepackage{bbm, dsfont}
\usepackage{stmaryrd}
\usepackage{algorithm}
\usepackage{algorithmic} 
\usepackage{mathrsfs}

\newcommand{\be}{ \begin{equation}}
\newcommand{\ee}{\end{equation}}

\newcommand\mr[1]{\mathring{\mathcal{#1}}}
\newcommand\wh[1]{\widehat{#1}}

\makeatletter
\newcommand{\proofpart}[2]{%
  \par
  \addvspace{\medskipamount}%
  \noindent\emph{Part #1: #2}\par\nobreak
  \addvspace{\smallskipamount}%
  \@afterheading
}
\makeatother

\def\E{{\mathbb E}}

\def\P{{\mathbb P}}

\def\R{{\mathbb R}}

\def\N{{\mathbb N}}

\def\maC{{\mathscr{C}}}

\def\T{{\mathcal{T}}}

\def\bdmu{\bar\mu^\dag}

\def\maT{{\mathscr T}}
\def\maK{{\mathcal K}}

\def\maI{{\mathcal I}}
\def\maU{{\mathcal U}}
\def\maA{{\mathcal{A}}}
\def\maV{{\mathcal{V}}}
\def\maE{{\mathcal{E}}}
\def\maG{{\mathcal{G}}}
\def\maB{{\mathcal{B}}}

\def\maM{{\mathcal{M}}}
\def\x{\mathbb x}

\def\T2a{{\tau_{2\alpha^{+}^{+}}}}
\def\t2a{{t_{2\alpha^{+}^{+}}}}

\def\M{{\mathscr M}}

\def\({{\Bigl(}}
\def\){{\Bigr)}}

\def\ind{{\mathchoice {\rm 1\mskip-4mu l} {\rm 1\mskip-4mu l}
{\rm 1\mskip-4.5mu l} {\rm 1\mskip-5mu l}}}
\def\square{\ifmmode\sqr\else{$\sqr$}\fi}
\def\sqr{\vcenter{
         \hrule height.1mm
         \hbox{\vrule width.1mm height2.2mm\kern2.18mm\vrule width.1mm}
         \hrule height.1mm}}                  

\begin{document}
\title[Local matching on the Configuration model]{Large graph limits of local matching algorithms on Configuration model graphs}
\author[M. H. A. Diallo Aoudi, P. Moyal and V. Robin]{Mohamed Habib Aliou Diallo Aoudi, Pascal Moyal, and Vincent Robin}
\begin{abstract}
{In this work, we propose a large-graph limit estimate of the matching coverage for several matching algorithms, on general random graphs generated by the configuration model. For a wide class of {\em local} matching algorithms, namely, algorithms that only use information on the immediate neighborhood of the explored nodes, we propose a joint construction of the graph by the configuration model, and of the resulting matching on the latter graph. This leads to a generalization in infinite dimension of the differential equation method of Wormald: We keep track of the matching algorithm over time by a measure-valued CTMC, for which we prove the convergence, to the large-graph limit, to a deterministic hydrodynamic limit, identified as the unique solution of a system of ODE's in the space of integer measures.  
Then, the asymptotic proportion of nodes covered by the matching appears as a simple function of that solution. We then make this solution explicit for three particular local algorithms: the classical {\sc greedy} algorithm, and then the {\sc uni-min} and {\sc uni-max} algorithms, two variants of the greedy algorithm that select, as neighbor of any explored node, its neighbor having the least (respectively largest) residual degree.}
%
\end{abstract}

\maketitle

\section{Introduction}
\label{sec:intro}
{A {\em matching} on a graph is a subgraph in which all nodes are of degree one, that is, have exactly one neighbor. A matching is {\em maximal}, if no more node can be added to it without breaking this property. A maximal matching is {\em perfect}, if it covers all the nodes of the graph. 
Providing conditions for the existence of a perfect matching on a graph is a classical problem in graph theory. The problems of implementing algorithms that are able to achieve this perfect matching {if it exists}, and if not, to maximize the so-called {\em matching coverage}, namely, the proportion of nodes appearing in 
the resulting matching, have received an increasing attention in various  research communities, from discrete mathematics to theoretical computer science, and discrete probability. Such problems have natural applications 
in a wide variety of fields, from on-line advertisement to peer-to-peer interfaces, from job and housing allocations to organ transplants and blood banks, online dating, and so on.}

{Necessary and sufficient conditions for the existence of a perfect matching were given in the classical reference \cite{Hall} for bipartite graphs, and then 
\cite{Tutte} for general (i.e., non-necessarily bipartite) graphs. Provided that these conditions hold, a flurry of algorithms were proposed to {achieve} this 
perfect matchings, starting from the now classical {\sc blossom} algorithm, see \cite{edmonds_1965}. The naive matching algorithm consists of the repetition of the following iteration:  1) An edge is added to the matching, and 2) Its neighboring edges are blocked. Maximum matching algorithms would then proceed to progessively {augment} the resulting maximal matching. In practice it is profitable, but algorithmically expensive, to allow the possibility of {\em backtracking} during the construction, that is, to delete edges from the matching, if this allows to add more edges to it later on, by following another path. This is the case 
in particular in {\sc blossom}-type algorithms. However, it may be imposed by practical constraints, and in particular by complexity costs, as the size of the graph gets large, to {\em forbid} backtracking along the procedure, specifically: as soon as an edge is included in the matching, it will remain so until the end of the procedure.}

{This is naturally the case in the so-called {\em online-matching} problem: Consider a bipartite graph $\maG(\maU\cup\maV,\maE)$, and suppose that the nodes of the ``$\maV$-side'' are known, but nodes of the ``$\maU$-side'' and their neighbors on the ``$\maV$-side'', are unveiled sequentially. Then, the matching algorithm just consists in performing sequentially, upon the unveiling of each new node $u\in\maU$ and its neighbors, the match of $u$ with one of its unmatched neighbors, if any, following a criterion that is fixed beforehand. The algorithm is said to be {\sc greedy}, if the match is chosen uniformly at random among the unmatched neighbors of $u$. See e.g. 
\cite{feldman2009online,Manshadi12,TCS-057,MJ13}. This {\em online} matching procedure, consisting of matching incoming item on the fly, is particularly adapted to the contexts of online advertisement (in which $\maV$ is the set of adds and $\maU$ is the set of users, that are interested by some adds but not all, and arrive sequentially) or organ transplants in real time, for instance. For a given resulting bipartite graph, the {\em competitive ratio} of a given algorithm, is then the proportion of nodes covered 
by the online matching, over the size of the largest maximal matching that could have been achieved on that graph ({\em offline} matching). In \cite{Karp}, Karp et al. used the so-called \emph{adversarial} order of arrivals as a worst case scenario for the matching completion, leading to the well-known $1-1/e$ bound for the competitive ratio.}

{In this work, in the same spirit we address a class of matching algorithms that prohibit any backtracking, and investigate the typical 
proportion of nodes covered by the matching, as the size of the graph grows large. We are interested in simple matching algorithms, that only take information on the immediate neighborhood of the explored nodes. They can be roughly defined as follows: at each step, 1) We choose a given unmatched node, following a criterion that uses, at most, the sole knowledge of the {\em degrees} of the unmatched nodes; 2) We choose the match of that node among its neighbors, if any, again on the basis of the sole knowledge of the degrees of these nodes, and 3) We erase these two matched nodes, and reiterate the procedure on the graph induced by the remaining nodes. We say that such matching algorithms are {\em local}, in the sense that the choices (of the explored node and then, of its match) only use information on the neighborhood of nodes at distance one. It is immediate that the aforementioned {\sc greedy} algorithm is local, in that sense.
The motivation behind this assumption is mostly practical: in today's huge networks in Telecom, computer architectures, biological networks and social media, 
a precise view of the geometry of the graph at hand is clearly out-of-reach, only the local characteristics of the graph are known. 
In such contexts, it is a classical approach to suppose that the graph at hand is {\em random}, characterized by its local characteristics (e.g., the degrees of its nodes), and generated by a random procedure. 
In this work, we consider the classical so-called {\em Configuration Model} (CM), see e.g. \cite{BOLLOBAS1980311,10.5555/3086985}. This random graph model 
is known to be particularly simple and versatile, and relevant to represent a wide variety of networks, from social networks to epidemiological and biological networks, 
for instance. By its very construction, based on the procedure of so-called {\em uniform pairing} of half-edges (see below), it is the suitable model to produce a {\em typical} realization of a random graph having a prescribed degree distribution. Specifically, it can be shown that, conditionality of producing a simple graph, its distribution 
is uniform among all simple graphs having this degree distribution (see Proposition 7.4 of \cite{10.5555/3086985}).}

{The matching problem on random graphs has an intense recent history, and has been investigate along various aspects: among others, the number of matchings of a given size in Erdös-Rényi and regular random graphs was characterized in \cite{ZM06} by the so-called cavity method, \cite{gamarnik2010randomized} provided performance bounds for the `randomized' greedy algorithm on regular random graphs. \cite{BLS13} characterizes the asymptotic optimal matching size in the CM. Recently, \cite{SNLMP23} analyzed the performance of online and offline matchings on geometric random graphs. \cite{SJM23} considers a stochastic block model (SBM), and provides a condition for achieving a perfect matching infinitely often, as the size of the graph gets large, by using a comparison to the stability problem of a related queueing system, the so-called {\em Stochastic matching model} 
introduced in \cite{MaiMoy16}.}

{The closest contribution to the present work is \cite{NSP21}, in which the performance of the {\sc greedy} matching algorithm is analyzed, on a 
bipartite configuration model graph. The asymptotic expected competitive ratio is identified as the solution of a functional equation involving the generating function of the degree distribution. 
In the present work, we complete and extend this result, by deriving the asymptotic matching coverage, to the large graph limit, for a general 
(instead of bipartite) graph, generated by the CM, and the simultaneous construction we propose for the graph together with a matching on it, 
slightly differs from that of \cite{NSP21}. Second, we propose a more exhaustive description of the construction, through the measure-valued process keeping track of the number of 
open half-edges of the nodes along the construction. Third, we extend our result to a wider class of {local} matching algorithms, in the sense introduced above, including the so-called {\sc uni-min} matching algorithm, similar in favor to the {\em Degree-greedy} algorithm introduced in \cite{bermolen2019degree} for 
the construction of Maximal Independent Sets.}

{A remarkable feature of the CM is to allow for a {\em joint} (synchronized) construction of the graph by the uniform pairing procedure, {\em and} of the maximal matching on the latter graph. Then the construction leads to a simple point measure-valued process, as described above. 
We can then come back to the original problem, by showing that the asymptotic matching coverage coincides, in some sense, to the one obtained if we implement the same algorithm on a graph that was previously constructed by the CM, see Theorem \ref{thm:coupling} below. The asymptotic matching coverage can then 
be obtained as a simple function of the limiting deterministic measure-valued process, obtained as the solution of a system of ODE. We thereby apply an infinite-dimensional version of the so-called {\em differential equation method} of Wormald, see \cite{10.2307/2245111} and \cite{W19}.  
The properties of measure-valued processes have been well understood since the pioneering monograph \cite{10.1007/BFb0084190}. Further, tools have been developed, especially in the context of queueing systems and networks, to obtain their fluid and diffusion limits, which are essential results to understand their asymptotic and/or mean behavior, in particular as the size of the queue gets large and/or in time-space scaling of the queueing system at hand, see e.g., along various kinds of queueing models, \cite{doytchinov2001real,grom,DecMoyEDF,decreusefond2008functional,kaspi2011law,ABK15}.}

{In the present work, we show the convergence of the sequence of processes to a deterministic and continuous measure-valued function, and thereby establish formally, and in the context of {\em general} (instead of bipartite) graphs, the fluid approximation heuristics proposed in \cite{aoudi:hal-03294781} for the {\sc greedy} and the {\sc uni-min} matching algorithms on the bipartite Configuration Model. In the latter reference, the performance of {\sc greedy} and {\sc uni-min} were then compared by solving numerically the limiting ODE in both cases. 
In the context of large-graph limits of Markovian processes on large random graphs, a similar approach has been used to study the large-graph behavior of a SIR epidemics (see \cite{decreusefond2012}), Greedy or degree-greedy construction of Maximal Independent sets on the CM, see \cite{Jam,bermolen2019degree}, and applications to CSMA ({\em Carrier Sense Multiple Access})-type algorithms on radio-mobile networks in \cite{bermolen2016estimating}.}

{This paper is organized as follows. After some preliminary in Section \ref{sec:prelim}, we introduce our main results regarding the large graph approximation of the 
matching coverage for the {\sc greedy} and the {\sc uni-min} algorithms in Section \ref{sec:mainres}, including examples of degree distributions in Section \ref{subsec:mainexamples}. Then, we describe precisely the more general class of {\em local} matching algorithms, on a previously constructed graph in Section \ref{subsec:localfinite}, and the corresponding joint construction of the graph by the CM and of the matching on that graph, in Section \ref{subsec:localCM}. In Section \ref{sec:Markov}, we introduce the measure-valued Markov chain of our joint construction. In Section \ref{sec:constructionhat} we introduce an auxiliary construction whose large-graph limits will turn out to be easier to establish. Our main general results, including the hydrodynamic limits of the measure-valued Continuous Time Markov Chain (CTMC), and the large graph limit of the matching coverage for local matching algorithms, is presented in Section \ref{sec:Main}. The proof of our main convergence result, using the latter auxiliary construction, 
is developed in Section \ref{sec:proof}. The specialization of these results to the special cases of the {\sc greedy}, the {\sc uni-min} and the {\sc uni-max} algorithms, are presented in Sections \ref{sec:greedy}, \ref{sec:unimin} and \ref{sec:unimax}, respectively.}

\section{Preliminary}
\label{sec:prelim}

\subsection{General notation}
Let $\R$, $\R_+$, $\N$ and $\N_+$ denote the sets of real numbers, non-negative numbers, non-negative and positive integers, respectively. 
For any $a\le b\in\N$, we denote by $\llbracket a,b \rrbracket$, the integer interval 
$$\llbracket a,b \rrbracket=\{a,a+1,\cdots,b-1,b\}.$$ 
In what follows, any finite set $\maA$ of cardinality $q\in\N_+$ is naturally identified with the integer interval $\llbracket 1,q \rrbracket$. 
For any real number $x\in\R$, we let $\lfloor x\rfloor$ denote the integer part of $x$. 

We denote by $\maC_b:=\maC_b(\N)$ (respectively, $\maC_K:=\maC_K(\N)$), the set of Borel bounded (resp., compact supported) functions 
$\N$ to $\R$. Let us denote by $\mathbb 1:x \longmapsto 1$, the mapping on $\R$ constantly equal to $1$. Let $\chi :x \longmapsto x$ be the {identity function} on $\R$, and for any $p\ge 1$, $\chi^p:x \longmapsto x^p$ be the $p$-power function on $\R$. 
For any $f:\N\to \R$, we define the {\em discrete gradient} of $f$, as 
 $$\nabla f:\, x \longmapsto f(x) - f(x-1),\, x\in\N.$$
  For two mappings $f$ and $g:\,\N\to \R$ and $n\in\N$, we write $f(n) = o(g(n))$ if $$\lim_{n\rightarrow\infty} \frac{f(n)}{g(n)} = 0.$$
   Likewise, we write $f(n) = O(g(n))$ if $$\lim_{n\rightarrow\infty} \left|\frac{f(n)}{g(n)}\right| < C,$$ for some positive constant $C$.
   
The null measure is denoted by $\mathbf 0$. For any measurable set $E$, we let $\M_F(E)$ denote the 
space of finite positive measures on $E$. We let $\M_F:=\M_F(\N)$ be the space of finite positive measures on $\N$, and $\M_p\subset \M_F$  
be the subset of counting measures on $\N$. For any $\mu\in\M_F$ and $i\in\N$, with some abuse we 
set $\mu(i)=\mu(\{i\}).$ 
For any $n\in\N$, we define $\M^n$ as the subset of $\M_p$ whose elements $\mu$ 
have a total mass less or equal to $n$, that is, $$\mu(\N)=\sum_{i=0}^\infty \mu(i) = m\le n.$$ 
Note, that any such measure $\mu\in\M^n$ can thus be written as 
\[\mu=\sum_{i=1}^m \delta_{a_i},\]
where $a_i \in \N$ for all $i\in\llbracket 1,m \rrbracket$, and $m\le n$. 
For any $n\in\N_+$, we then define the set 
$$\bar{\M}^n:= \frac{1}{n}\M^n = \left\{\frac{1}{n}\mu \, :\, \mu \in \M^n\right\}.$$ 
Then, $\bar{\M}^n$ is a subset of $\bar\M$, the space of finite measures on $\N$ having total mass bounded by $1$. 
For a function $f:\N\to \R$ and a measure $\mu\in\M_F$, we write 
\[\cro{\mu,f}=\int_\R f d\mu:=\sum_{i\in\N}f(i)\mu(i).\]
For a measure $\bar\mu\in\bar\M$, we let respectively $F_{\bar\mu}$ and $\bar F_{\bar\mu}$ denote the c.d.f. and the tail of distribution 
of the size-biased probability measure associated to 
$\bar\mu$, that is, for all $y\in\N$, 
\begin{align}
F_{\bar\mu}(y)&=\sum_{j=0}^y {j\bar\mu(j)\over \cro{\bar\mu,\chi}};\label{eq:defF}\\
\bar F_{\bar\mu}(y)&=1-F_{\bar\mu}(y)=\sum_{j=y+1}^\infty {j\bar\mu(j)\over \cro{\bar\mu,\chi}}\cdot\label{eq:defbarF} 
\end{align}
In particular, for any $\mu\in\M_F$ and $p\ge 1$, $\cro{\mu,\mathbb 1}$ and $\cro{\mu,\chi^p}$ respectively represent the total mass and the $p$-th moment of the measure $\mu$. 
The above measure spaces are endowed with their weak (resp. vague) topology. In particular, the weak convergence of measures is denoted by 
\[\mu^n \xRightarrow{n\to\infty} \mu \iff \cro{\mu^n,f} \xrightarrow{n\to \infty} \cro{\mu,f}, \mbox{ for all }f\in\maC_b.\]
Both $\maC_b$ and $\maC_K$ are endowed with the topology induced by the sup norm 
$$\lVert f \rVert = \sup_{n\in\N} |f(n)|,\quad f\in\maC_b \,(\mbox{resp.},\,\maC_K).$$ 

\medskip

For a non-oriented graph $\mathcal G(\mathcal V,\mathcal E)$, $\mathcal V$ denotes the set of nodes and $\mathcal E$, the set of edges, that is, a set of subsets of $\mathcal V$ of cardinality $2$. For any node $v$, we let $\maE(v)$ be the set of neighbors of $v$, i.e., of nodes of $\maV$ that share an edge with $v$. 
The degree of $v$ is then the cardinality of $\maE(v)$. 

\medskip

Throughout, unless the contrary is explicitly mentioned, all random variables (r.v.'s, for short) are defined on a common probability space $(\Omega, \mathscr F, \mathbb P)$. For a Polish space $E$ and $T>0$ we denote by $\mathbb{D}([0,T],E)$, the space of RCLL (right-continuous with left-hand limits) $E$-valued processes, endowed with the Skorokhod topology  (see e.g. \cite{Mlard1993SurLC,10.2307/1427238}). We let $\mathbb{C}([0,T],E)$ denote the space of continuous $E$-valued processes. We use indifferently the notation ``$\Longrightarrow$'' for weak convergence of $E$-valued r.v.'s, and for convergence of measures in the space $(\M_F(E),w)$ endowed with the weak topology. 

\section{Main results}
\label{sec:mainres}
In this section, we introduce and summarize our main results. We start by introducing two of the main matching algorithms that will be addressed in this work. 
A wider class of algorithms will be defined, and further formalized, in Section \ref{sec:localalgo}. 

\subsection{Two matching algorithms on a finite graph}
\label{subsec:mainresalgo}
We start by considering a given finite graph $\mr G(\mr V,\mr E)$ of size $|\mr V|=n$. Hereafter, a circle will be appended to all 
characteristics of algorithms on the fixed graph $\mr G$. We consider two matching algorithms on $\mr G$: 
First, in Algorithm \ref{algo:greedy} we consider the classical {\em greedy matching}, which consists in choosing, at each step, a node uniformly at random 
among the unmatched nodes having a non-empty neighborhood, if any, and then match it with another unmatched node, chosen uniformly at random among its neighbors. In Algorithm \ref{algo:unimin}, we formalize what we call the \textsc{uni-min} matching algorithm. It is defined exactly as the \textsc{greedy} matching algorithm, except that, at each step, the match of the first chosen node is drawn uniformly at random among its neighbors having {\em minimal degree} in the graph induced by the unmatched nodes. This good-sense modification tends to favor, among the neighbors of the first node, those who will be {harder} to match afterwards, due to a small degree. 

\begin{algorithm}
\caption{\textsc{greedy} matching algorithm on $\mr{G}(\mr V,\mr E)$}
\label{algo:greedy}
\begin{algorithmic}
\REQUIRE Non-empty graph $\mr G(\mr V,\mr E)$
\STATE $\mr{\mbox{Matching}} \leftarrow \emptyset$;
\FOR{$1\le i \le  n$}
   \IF{$\mr V$ is empty, or all nodes in $\mr V$ have an empty neighborhood}
        \STATE Do not do anything 
   \ELSE
	\STATE Pick $v$ uniformly at random among the nodes of $\mr V$ having a non-empty neighborhood
	\STATE Pick $v'$uniformly at random in its neighborhood $\mr E(v)$;
	\STATE $\mr V \leftarrow \mr V \setminus \{v,v'\}$;
	\STATE $\mr{\mbox{Matching}} \leftarrow \mr{\mbox{Matching}} \cup \{\{v,v'\}\}$;
   \ENDIF
   \STATE $i \leftarrow i+1$;
\ENDFOR
\end{algorithmic}
\end{algorithm}
\normalsize

\begin{algorithm}
\caption{\textsc{uni-min} matching algorithm on $\mr G(\mr V,\mr E)$}
\label{algo:unimin}
\begin{algorithmic}
\REQUIRE Non-empty graph $\mr G(\mr V,\mr E)$
\STATE $\mr{\mbox{Matching}} \leftarrow \emptyset$;
\FOR{$1\le i \le  n$}
   \IF{$\mr V$ is empty, or all nodes in $\mr V$ have an empty neighborhood}
        \STATE Do not do anything 
   \ELSE
	\STATE Pick $v$ uniformly at random among the nodes of $\mr V$ having a non-empty neighborhood
	\STATE Pick $v'$uniformly at random in in the set of nodes of its neighborhood $\mr E(v)$ having minimal degree in $\mr V$;
	\STATE $\mr V \leftarrow \mr V \setminus \{v,v'\}$;
	\STATE $\mr{\mbox{Matching}} \leftarrow \mr{\mbox{Matching}} \cup \{\{v,v'\}\}$;
   \ENDIF
   \STATE $i \leftarrow i+1$;
\ENDFOR
\end{algorithmic}
\end{algorithm}

Both algorithms \ref{algo:greedy} and \ref{algo:unimin} produce a maximal matching on $\mr G$. 
For $\bphi\in\{\textsc{greedy},\textsc{uni-min}\}$, the {\em matching coverage}
$\mathring{\mathbf M}^n_{\bphi}(\mr{G})$ is then the proportion of nodes of $\mr V$ that ended up in the matching at the termination time $n$, 
namely, 
\begin{equation}
\label{eq:defmatchingcov}
\mathring{\mathbf M}^n_{\bphi}(\mr{G}) = {2\left|\,\mr{\mbox{Matching}}\,\right| \over n }\cdot 
\end{equation} 
For both algorithms, the matching can be perfect, that is, $\mathring{\mathbf M}^n_{\bphi}(\mr{G})=1$, only if $n$ is even and $\maG$ satisfies Tutte's condition, see \cite{Tutte}. On the other hand, it is well known that both algorithms produce a matching whose cardinality is at least half of the largest possible matching on 
$\mr G$, see e.g. Section 10.3 of \cite{Lovasz}. Both algorithms can be seen as {\em online} algorithms, in the sense that the edges that are added to the matching cannot be erased, that is, there is no backtracking in the construction of the matching. 

\subsection{The two corresponding algorithms on the configuration model}
\label{subsec:mainresalgoCM}
Let us now consider the two corresponding matching algorithms on a {\em random} (multi-)graph, constructed by the configuration model, as introduced in \cite{BOLLOBAS1980311}. As is well known, this random graph model consists in the following three main steps: 
\begin{itemize}
\item[(i)] Fix a probability distribution $\xi$ on $\N$, and draw the degree vector $\mathbb d=(d(1),\ldots,d(n))$ of the $n$ nodes, as a $n$-sample of $\xi$. 
\item[(ii)] Append to each node $i$, as many half-edges as its degree $d(i)$; 
\item[(iii)] Draw uniformly at random, a pairing of all the $\sum_{i=1}^n d(i)$ half-edges thereby obtained, to obtain edges. 
\end{itemize}
This results in a multi-graph having the prescribed degree distribution. However, it is well known that the number of self-loops and mutliple edges become small as $n$ goes large. Moreover, conditional on the resulting multi-graph being a graph (an event whose probability tends to a positive value), the latter is drawn uniformly at random among the graphs having degree distribution $\mathbb d$, and thereby provides a `{\em typical}' realization of a graph having this degree distribution. For details regarding this classical model, see e.g. Chapter 7 of \cite{RDH}. 

In line with \cite{decreusefond2012,Jam,bermolen2016estimating,bermolen2019degree,NSP21,DFSCM}, we adopt the so-called 
`{\em Constructing while exploring}' approach, to provide a declination of Algorithms \ref{algo:greedy} and \ref{algo:unimin} in which the matching is constructed together with the multi-graph. This results in Algorithms \ref{algo:greedyCM} and \ref{algo:uniminCM}. 

\begin{algorithm}
\caption{\textsc{greedy} matching algorithm on CM($\mathbb d$)}
\label{algo:greedyCM}
\begin{algorithmic}
\REQUIRE Degree vector $\mathbb d=(d(1),\ldots,d(n))$; $n$ nodes with $d(1)$, $d(2)$, ... ,$d(n)$ incident half-edges. 
\STATE $\mbox{Matching} \leftarrow \emptyset$;
\FOR{$1\le i \le  n$}
   \IF{No vertices have a positive number of incident half-edges}
        \STATE Do not do anything 
   \ELSE
	\STATE Pick $v$ uniformly at random in the set of nodes having incident half-edges;
	\STATE Pair all half-edges of $v$ uniformly at random with available half-edges, one by one. 
	\STATE Update the number of available half-edges of all nodes. 
	\STATE $\maV \leftarrow \maV \setminus \left\{v\right\}$;
	\IF{$v$ shares an edge with no other node than itself}
		\STATE  Do not do anything
	\ELSE
		\STATE Pick $v'$uniformly at random in the set of nodes that share an edge with $v$;
		\STATE $\maV \leftarrow \maV \setminus \left\{v'\right\}$;
		\STATE $\mbox{Matching} \leftarrow \mbox{Matching} \cup \{\{v,v'\}\}$;
	\ENDIF
   \ENDIF
   \STATE $i \leftarrow i+1$;
\ENDFOR
\end{algorithmic}
\end{algorithm}

\begin{algorithm}
\caption{\textsc{uni-min} matching algorithm on CM($\mathbb d$)}
\label{algo:uniminCM}
\begin{algorithmic}
\REQUIRE Degree vector $\mathbb d=(d(1),\ldots,d(n))$; $n$ nodes with $d(1)$, $d(2)$, ... ,$d(n)$ incident half-edges. 
\STATE $\mbox{Matching} \leftarrow \emptyset$;
\FOR{$1\le i \le  n$}
   \IF{No vertices have a positive number of incident half-edges}
        \STATE Do not do anything 
   \ELSE
	\STATE Pick $v$ uniformly at random in the set of nodes having incident half-edges;
	\STATE Pair all half-edges of $v$ uniformly at random with available half-edges, one by one. 
	\STATE Update the number of available half-edges of all nodes. 
	\STATE $\maV \leftarrow \maV \setminus \left\{v\right\}$;
	\IF{$v$ shares an edge with no other node than itself}
		\STATE  Do not do anything
	\ELSE
		\STATE Pick $v'$uniformly at random in the set of nodes having the smallest number of incident half-edges, among those which 
		share an edge with $v$;
		\STATE $\maV \leftarrow \maV \setminus \left\{v'\right\}$;
		\STATE $\mbox{Matching} \leftarrow \mbox{Matching} \cup \{\{v,v'\}\}$;
	\ENDIF
   \ENDIF
   \STATE $i \leftarrow i+1$;
\ENDFOR
\end{algorithmic}
\end{algorithm}

By the definition of the uniform pairing procedure, both Algorithm \ref{algo:greedyCM} and Algorithm \ref{algo:uniminCM} produce a multi-graph 
$\maG$ of degree vector $\mathbb d$, together with a maximal matching on $\maG$.  
Similarly to \eqref{eq:defmatchingcov}, for $\bphi\in\{\textsc{greedy},\textsc{uni-min}\}$, the {matching coverage} 
is then defined as the following r.v., 
\begin{equation}
\label{eq:defmatchingcovCM}
{\mathbf M}^n_{\bphi}(\xi) = {2\left|\,{\mbox{Matching}}\,\right| \over n }\cdot
\end{equation}

Clearly, Algorithm \ref{algo:greedyCM} (respectively, Algorithm \ref{algo:uniminCM}) transposes the matching procedure of Algorithm \ref{algo:greedy} (resp., of Algorithm \ref{algo:unimin}) in a context where, at each step, the neighborhood of node $v$ is constructed on the fly, instead of being previously known. 
This fact will be formalized later on, by establishing the equivalence in distribution of two processes representing the two corresponding constructions in both cases; See Theorem \ref{thm:coupling} below. 

\subsection{Main results}
\label{subsec:mainres}
We have the following convergence results for the matching coverages of Algorithm \ref{algo:greedyCM} and \ref{algo:uniminCM}. 

\begin{theorem}
\label{thm:matchcovgreedy}
If $\xi$ has a support that is not restricted to $\{0\}$, and a finite $3,5+\varepsilon$-th moment for any $\varepsilon>0$, 
then we have the convergence in probability 
\[{\mathbf {M}}^n_{\textsc{greedy}}(\xi) \xrightarrow{(n,\P)} 1 - \bar{\mu}^{\textsc{greedy}}_1 (\{0\}),\]
where the measure-valued function $\procun{\bar{\mu}_t^{\textsc{greedy}}}$ is the unique solution in $\mathbb{D}([0,1],\bar\M)$ 
of the system of ODE's 
\begin{equation}\label{eq:ODEgreedy}
\begin{cases}
\eta_0 &=\xi;\\
{\d\cro{\eta_t,f} \over \d t}&=-\Biggl\{{\cro{\eta_t,f\mathbb 1_{\N_+}}\over \cro{\eta_t,\mathbb 1_{\N_+}}} +{\cro{\eta_t,\chi f}\over\cro{\eta_t,\chi}}
+ {\cro{\eta_t,\chi\nabla f} \over \cro{\eta_t,\chi}}\left\{{\cro{\eta_t,(\chi- \mathbb 1_{\N_+})}\over \cro{\eta_t,\mathbb 1_{\N_+}}}
+ {\cro{\eta_t,\chi^2 - \chi} \over \cro{\eta_t,\chi}}\right\} \Biggl\}\mathbb 1_{\cro{\eta_t,\mathbb 1_{\N_+}}>0},\\
&\hspace{20pt}\quad t\in[0,1],\,f \in \maC_b(\R).
\end{cases}\end{equation}
\end{theorem} 

\begin{theorem}
\label{thm:matchcovunimin}
If $\xi$ has a support that is not restricted to $\{0\}$, and included in $\llbracket 0,N \rrbracket$ for some integer $N$, then 
we have the convergence in probability 
\[{\mathbf {M}}^n_{\textsc{uni-min}}(\xi) \xrightarrow{(n,\P)} 1 - \bar{\mu}^{\textsc{uni-min}}_1 (\{0\}),\]
where, recalling \eqref{eq:defbarF}, the measure-valued function $\procun{\bar{\mu}_t^{\textsc{uni-min}}}$ is the unique solution in $\mathbb{D}([0,1],\bar\M)$ 
of the system of ODE's 
\begin{equation}\label{eq:ODEunimin}
\begin{cases}
\eta_0 &=\xi;\\
{\d\cro{\eta_t,f} \over \d t}&=-\Biggl\{{\cro{\eta_t,f\mathbb 1_{\N_+}}\over \cro{\eta_t,\mathbb 1_{\N_+}}} + {\sum_{k\in\N_+}\eta_t(k)\sum_{k'\in\N_+}f(k')
\left\{\bar F_{\eta_t}(k'-1)^k-\bar F_{\eta_t}(k')^k\right\}\over\cro{\eta_t,\mathbb 1_{\N_+}}}\\
&\hspace{12pt}\quad \left.+ {\cro{\eta_t,\chi\nabla f} \over \cro{\eta_t,\chi}}\left\{{\cro{\eta_t,(\chi- \mathbb 1_{\N_+})}\over \cro{\eta_t,\mathbb 1_{\N_+}}}
+{\sum_{k\in\N_+}\eta_t(k)\sum_{k'\in\N_+}\left(k'-1\right)\left\{\bar F_{\eta_t}(k'-1)^k-\bar F_{\eta_t}(k')^k\right\}\over\cro{\eta_t,\mathbb 1_{\N_+}}}\right\} \right\}\mathbb 1_{\cro{\eta_t,\mathbb 1_{\N_+}}>0},\\
&\hspace{20pt}t\in[0,1],\,f \in \maC_b(\R).
\end{cases}\end{equation}
\end{theorem}
\noindent Theorems \ref{thm:matchcovgreedy}  and Theorem \ref{thm:matchcovunimin} are respectively proven in 
Section \ref{sec:greedy} and Section \ref{sec:unimin}, as corollaries of a more general result, Theorem \ref{thm:main}. 

These two results 
thus establish large-graph approximations for the matching coverage of two simple, online matching algorithms on the Configuration model, given respectively 
by Algorithm \ref{algo:greedyCM} and Algorithm \ref{algo:uniminCM}. In view of the above remarks, they can be seen as proxies for the large-graph approximations of the respective matching coverages of Algorithm \ref{algo:greedy} and Algorithm \ref{algo:unimin}, for a graph that is uniformly drawn among the graphs having degree distribution $\xi$. 

\subsection{Examples}
\label{subsec:mainexamples}
We now provide a few examples, to illustrate how these results specialize for given degree distributions. 

\subsubsection{$1$-regular graphs}
We first consider the trivial case where $\xi=\delta_1$, namely, the graph is $1$-regular and so it is itself a perfect matching. Then, 
we readily obtain that the unique solution to \eqref{eq:ODEgreedy} and the unique solution to \eqref{eq:ODEunimin}
coincide, and 
read 
\begin{equation*}
\bar{\mu}_t^{\textsc{greedy}}=\bar{\mu}_t^{\textsc{uni-min}}=(1-2t)\delta_1\mathbb 1_{\left[0,{1\over 2}\right]}(t),\quad t\in[0,1]. 
\end{equation*}
In particular, the limiting matching coverages are given by 
\[1 - \bar{\mu}^{\textsc{greedy}}_1 (0) = 1 - \bar{\mu}^{\textsc{uni-min}}_1 (0)=1.\]
In other words, 2 nodes are matched per unit of time, and all nodes have been added to the matching by time ${1\over 2}$, and so the matching is perfect, as it must.

\subsubsection{Finite support distributions}
Let us now address the general case where the limiting degree distribution has support in $\llbracket 0,n \rrbracket$, for $n\in\N_+$. 

First consider the matching criterion {\sc greedy}. Then, the unique solution to the system \eqref{eq:ODEgreedy} is fully characterized by 
the $(n+1)$-dimensional process $\procun{\left(\eta_t(0),\eta_t(1),\ldots,\eta_t(n)\right)}$, and simple computations shows that the latter process 
solves, on the interval $[0,1]$, the following $(n+1)$-dimensional ODE of unknown $\procun{\left(x_0(t),x_1(t), \ldots, x_n(t)\right)}$: 
\begin{equation}\label{eq:ODEgreedyd=n}
\begin{cases}
\left(x_0(0),x_1(0),\ldots,x_n(0)\right) =\xi\,;\\
\vspace{8pt}
{\d x_0(t) \over \d t}
={x_1(t)\over \sum_{i=1}^n ix_i(t)}A\left(x_1(t),\ldots,x_n(t)\right)
\mathbb 1_{\sum_{i=1}^n x_i(t)>0}\,;\\
\vspace{5pt}
{\d x_j(t) \over \d t}
=-\Biggl\{{x_j(t)\over \sum_{i=1}^n x_i(t)}
+{jx_j(t)\over\sum_{i=1}^nix_i(t)}+\left({jx_j(t)-(j+1)x_{j+1}(t)\over\sum_{i=1}^nix_i(t)}\right)A\left(x_1(t),\ldots,x_n(t)\right)\Biggl\}
\mathbb 1_{\sum_{i=1}^n x_i(t)>0},\\
\,\hspace{354pt} j\in\llbracket 1,n-1 \rrbracket\,;\\
\vspace{5pt}
{\d x_n(t) \over \d t}
=-\Biggl\{{x_n(t)\over \sum_{i=1}^n x_i(t)}
+{nx_n(t)\over\sum_{i=1}^nix_i(t)}+\left({nx_n(t)\over\sum_{i=1}^nix_i(t)}\right)A\left(x_1(t),\ldots,x_n(t)\right)\Biggl\}
\mathbb 1_{\sum_{i=1}^n x_i(t)>0}, 
\end{cases}\end{equation}
for $$A(x_1,\ldots,x_n)={\sum_{i=2}^n(i-1)x_i\over \sum_{i=1}^n x_i} + {\sum_{i=2}^ni(i-1)x_i \over \sum_{i=1}^nix_i}\cdot$$

Similarly, in the case of {\sc uni-min}, the unique solution to \eqref{eq:ODEunimin} is fully characterized by 
the process $\procun{\left(\eta_t(0),\eta_t(1),\ldots,\eta_t(n)\right)}$, which is the only solution of the $(n+1)$-dimensional ODE 
\begin{equation}\label{eq:ODEunimind=n}
\begin{cases}
\left(x_0(0),x_1(0),\ldots,x_n(0)\right) =\xi\,;\\
\vspace{8pt}
{\d x_0(t) \over \d t}
={x_1(t)\over \sum_{i=1}^n i x_i(t)}B\left(x_1(t),\ldots,x_n(t)\right)
\mathbb 1_{\sum_{i=1}^n x_i(t)>0}\,;\\
\vspace{5pt}
{\d x_j(t) \over \d t}
=-\Biggl\{{x_j(t)\over \sum_{i=1}^n x_i(t)}
+{\sum_{k=1}^n x_k(t){\left(\sum_{\ell=j}^n \ell x_\ell(t)\right)^k- \left(\sum_{\ell=j+1}^n \ell x_\ell(t)\right)^k\over \left(\sum_{\ell=1}^n \ell x_\ell(t)\right)^k}\over\sum_{i=1}^n x_i(t)}\\
\hspace{97pt} +\left({jx_j(t)-(j+1)x_{j+1}(t)\over\sum_{i=1}^nix_i(t)}\right)B\left(x_1(t),\ldots,x_n(t)\right)\Biggl\}
\mathbb 1_{\sum_{i=1}^n x_i(t)>0},\quad j\in\llbracket 1,n-1 \rrbracket\,;\\
\vspace{5pt}
{\d x_n(t) \over \d t}
=-\Biggl\{{x_n(t)\over \sum_{i=1}^n x_i(t)}
+{\sum_{k=1}^n x_k(t){\left(nx_n(t)\right)^k\over \left(\sum_{\ell=1}^n \ell x_\ell(t)\right)^k}\over\sum_{i=1}^nx_i(t)}+\left({nx_n(t)\over\sum_{i=1}^nix_i(t)}\right)B\left(x_1(t),\ldots,x_n(t)\right)\Biggl\}
\mathbb 1_{\sum_{i=1}^n x_i(t)>0},
\end{cases}\end{equation}
for 
$$B(x_1,\ldots,x_n)={\sum_{i=1}^n(i-1)x_i\over \sum_{i=1}^n x_i} + {\sum_{k=1}^n x_k \sum_{k'=2}^n(k'-1){\left(\sum_{\ell=k'}^n \ell x_\ell\right)^k- \left(\sum_{\ell=k'+1}^n \ell x_\ell\right)^k\over \left(\sum_{\ell=1}^n \ell x_\ell\right)^k} \over \sum_{i=1}^nx_i}\cdot$$

For any fixed dimension $n$ we can retrieve a numerical approximation of the matching coverage in both cases, by solving the differential systems \eqref{eq:ODEgreedyd=n} and \eqref{eq:ODEunimind=n} numerically, and then retreiveing the matching coverage, as $1-\eta_1(0)$. 
On Figure 1, we present our numerical estimates (Runge-Kutta K4 numerical resolution of mesh $10^{-5}$) for 
the respective matching coverages of {\sc greedy} and {\sc uni-min}, for various regular graphs (degrees from $2$ to $15$ - left sub-figure) and 
uniform degree distributions on $\llbracket 1,n \rrbracket$ (maximal degree from $2$ to $20$ - right sub-figure), which show that 
 {\sc uni-min} outperforms {\sc greedy} by an important margin in all cases.

\begin{figure}[h!]			
			\centering
				\begin{subfigure}{0.4\textwidth}
				\includegraphics[width=.99\textwidth]{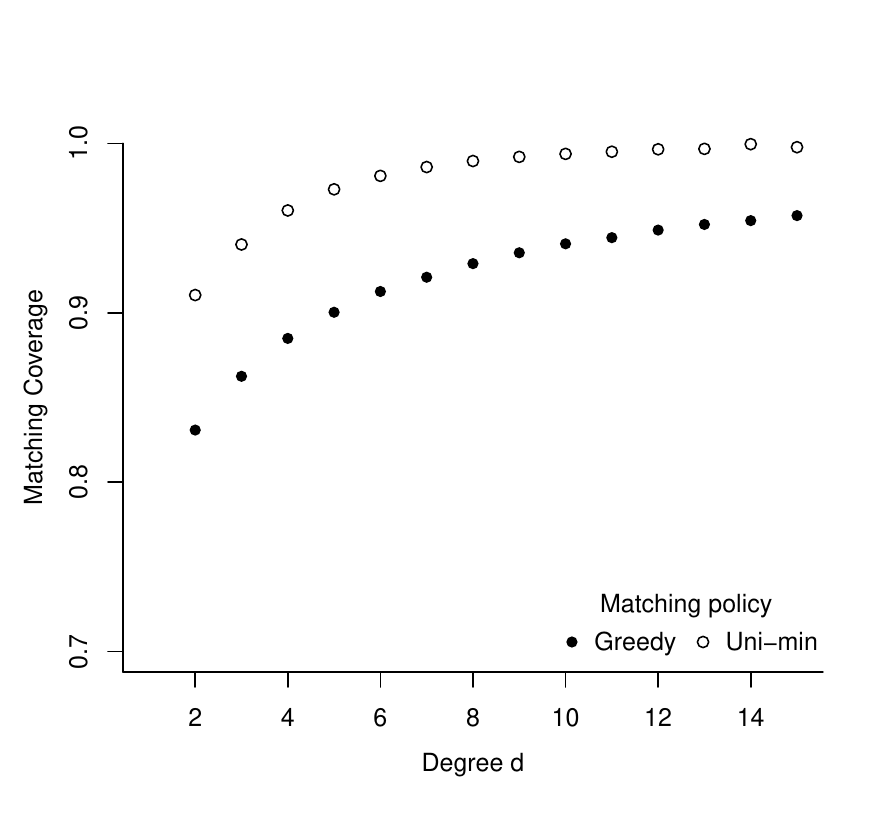}
				\end{subfigure}
			\begin{subfigure}{0.4\textwidth}
				\includegraphics[width=.99\textwidth]{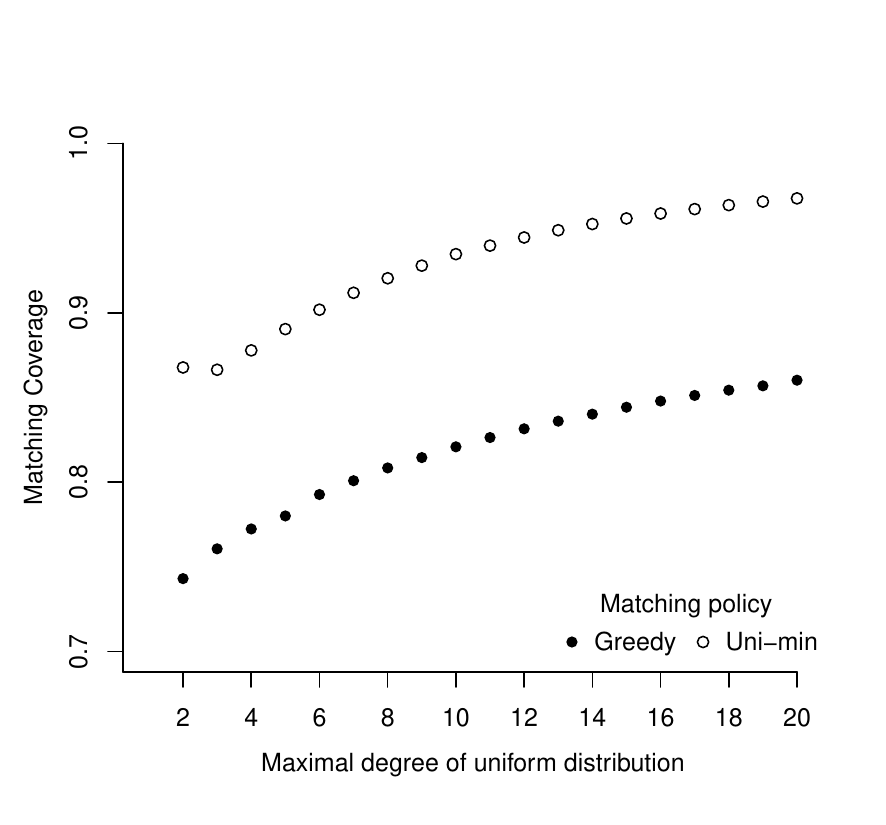}
			\end{subfigure}
	\label{fig:GreedyVSUnimin}
	\caption{Matching coverages for {\sc greedy} (bullets) and {\sc uni-min} (circles), for various $n$-regular graphs (left sub-figure) and uniform degree distributions over $\llbracket 1,n \rrbracket$ (right sub-figure).}
\end{figure}

\subsubsection{Greedy for Poisson degree distributions}
\label{subsubsec:greedyPoisson}
Let us now consider that the initial measure $\xi$ has support on $\N$, and set the matching criterion as {\sc greedy}. 
Then, the system \eqref{eq:ODEgreedyd=n} can be extended to the whole support $\N$, as 

\begin{equation}\label{eq:ODEgreedyN}
\begin{cases}
\left(x_0(0),x_1(0),\ldots,x_n(0)\right) =\xi\,;\\
\vspace{8pt}
{\d x_0(t) \over \d t}
={x_1(t)\over \sum_{i=1}^{+\infty} ix_i(t)}A\left(x_1(t),x_2(t),\ldots\right)
\mathbb 1_{\sum_{i=1}^{+\infty} x_i(t)>0}\,;\\
\vspace{5pt}
{\d x_j(t) \over \d t}
=-\Biggl\{{x_j(t)\over \sum_{i=1}^{+\infty} x_i(t)}
+{jx_j(t)\over\sum_{i=1}^{+\infty}ix_i(t)}+\left({jx_j(t)-(j+1)x_{j+1}(t)\over\sum_{i=1}^{+\infty}ix_i(t)}\right)A\left(x_1(t),x_2(t),\ldots\right)\Biggl\}
\mathbb 1_{\sum_{i=1}^n x_i(t)>0},\\
\,\hspace{354pt} j\ge 1,
\end{cases}\end{equation}
for $$A(x_1,x_2,\ldots)={\sum_{i=2}^{+\infty}(i-1)x_i\over \sum_{i=1}^{+\infty} x_i} + {\sum_{i=2}^{+\infty} i(i-1)x_i \over \sum_{i=1}^{+\infty}ix_i}\cdot$$
In the Poisson case, we have the following result. 

\begin{proposition}
Suppose that $\xi$ is a Poisson measure $\mathcal P(\rho)$, with parameter $\rho>0$. Then the unique solution to \eqref{eq:ODEgreedyN} is given by the 
family of weighted Poisson measures 
\begin{equation}
\label{eq:Poisssyst1}
\begin{cases}
x_j(t) &=\displaystyle v(t) {(\rho v(t))^j \over j! }e^{-\rho v(t)},\hspace{111pt}\quad j\in\N_+,\quad t\in [0,1];\\
x_0(t)&=\displaystyle e^{-\rho}+\int_0^t e^{-\rho v(s)}\left({\rho v(s) \over 1-e^{-\rho v(s)}}-1+\rho v(s)\right)\mathbb 1_{v(t)>0}\d t, \quad t\in [0,1],
\end{cases}\end{equation}
where $\procun{v(t)}$ is the unique solution of the ODE 
\begin{equation}
\label{eq:PoissODE}
\begin{cases}
v(0) &=1;\\
v^\prime(t)&=-\left(1+{1\over 1-e^{-\rho v(t)}}\right)\mathbb 1_{v(t)>0},\quad t \in [0,1]. 
\end{cases}\end{equation}
Consequently, the asymptotic matching coverage reads as follows, 
\begin{equation}
\label{eq:truc}
1 - \bar{\mu}^{\textsc{greedy}}_1 (\{0\})=1-e^{-\rho}-\int_0^1 e^{-\rho v(s)}\left({\rho v(s) \over 1-e^{-\rho v(s)}}-1+\rho v(s)\right)\mathbb 1_{v(t)>0} \d s.
\end{equation}
\end{proposition}

\begin{proof}
We write the unique solution to \eqref{eq:ODEgreedyN} under the form 
\begin{equation}
\label{eq:Poisssyst2}
\begin{cases}
x_j(t) &=\displaystyle v(t) {(\rho v(t))^j \over j! }e^{-\rho v(t)},\hspace{132pt}\quad j\in\N_+,\quad t\in [0,1];\\
x_0(t) &=x_0(0)+\displaystyle\int_0^t {x_1(t)\over \sum_{i=1}^{+\infty} ix_i(t)}A\left(x_1(t),x_2(t),\ldots\right)
\mathbb 1_{\sum_{i=1}^{+\infty} x_i(t)>0}\d t,\quad t\in[0,1]. 
\end{cases}\end{equation}
Let us first observe the following identities: for all $t\in [0,1]$, 
\begin{align}
\label{eq:Poissrhs1}
\sum_{i=1}^\infty x_i(t) &= v(t)\left(1-e^{-\rho v(t)}\right);\\
\label{eq:Poissrhs2}
\sum_{i=1}^\infty ix_i(t) &= \rho v(t)^2;\\
\label{eq:Poissrhs3}
\sum_{i=1}^\infty (i-1)x_i(t)&=\sum_{i=1}^\infty ix_i(t)-\sum_{i=1}^\infty x_i(t) = v(t)\left(\rho v(t) -1+e^{-\rho v(t)}\right);\\
\label{eq:Poissrhs4}
\sum_{i=1}^\infty i(i-1)x_i(t) &= v(t)e^{-\rho v(t)}(\rho v(t))^2\sum_{i=2}^\infty {(\rho v(t))^{j-2} \over (j-2)!}=\rho^2 v(t)^3,\\
\label{eq:Poissrhs5}
A\left(x_1(t),x_2(t),\ldots\right)&={\sum_{i=1}^{+\infty}(i-1)x_i(t)\over \sum_{i=1}^{+\infty} x_i(t)} + {\sum_{i=1}^{+\infty} i(i-1)x_i(t) \over \sum_{i=1}^{+\infty}ix_i(t)}
= {\rho v(t) \over 1-e^{-\rho v(t)}}-1+\rho v(t). 
\end{align}
Then, in \eqref{eq:ODEgreedyN}, for any $j\ge 1$, the left-hand side of the ODE satisfied by $\procun{x_j(t)}$ reads for all $t$ as
\begin{align}
{\d x_j(t) \over \d t} &= v^\prime(t)\left({(\rho v(t))^j \over j! }e^{-\rho v(t)}\right)+v(t)\left(\rho v^\prime(t){(\rho v(t))^{j-1} \over (j-1)!}e^{-\rho v(t)}-\rho v^\prime(t){(\rho v(t))^{j} \over j!}e^{-\rho v(t)} \right)\notag\\
&=e^{-\rho v(t)}\left\{{(\rho v(t))^j \over j! }(1-\rho v(t))v^\prime(t)+{(\rho v(t))^{j-1} \over (j-1)!}\rho v(t)v^\prime(t)\right\}.\label{eq:Poisslhs}
\end{align}
On the other hand, in view of (\ref{eq:Poissrhs1}-\ref{eq:Poissrhs5}), the right-hand side of that ODE can be written for all $t$ as 
\begin{multline}
-\Biggl\{{(\rho v(t))^j \over j! }{e^{-\rho v(t)}\over 1-e^{-\rho v(t)}}
+{(\rho v(t))^j \over (j-1)! }{v(t)e^{-\rho v(t)}\over\rho v(t)^2}\\
\shoveright{+\left({{(\rho v(t))^j \over (j-1)! }v(t)e^{-\rho v(t)} - {(\rho v(t))^{j+1} \over j! }v(t)e^{-\rho v(t)}\over\rho v(t)^2}\right)
\left( {\rho v(t) \over 1-e^{-\rho v(t)}}-1+\rho v(t)\right)\Biggl\}
\mathbb 1_{v(t)>0}}\\
\shoveleft{=e^{-\rho v(t)}\Biggl\{-{(\rho v(t))^j \over j! }{1\over 1-e^{-\rho v(t)}}
-{(\rho v(t))^{j-1} \over (j-1)! }v(t)}\\
\shoveright{+\left({(\rho v(t))^{j} \over j! }-{(\rho v(t))^{j-1} \over (j-1)! }\right)
\left( {\rho v(t) \over 1-e^{-\rho v(t)}}-1+\rho v(t)\right)\Biggl\}
\mathbb 1_{v(t)>0}}\\
\shoveleft{=e^{-\rho v(t)}\Biggl\{{(\rho v(t))^j \over j! }\left(1-\rho v(t)\right)\left(-1-{1\over 1-e^{-\rho v(t)}}\right)}\\
+{(\rho v(t))^{j-1} \over (j-1)! }\rho v(t)\left(-1-{1\over 1-e^{-\rho v(t)}}\right)\Biggl\}
\mathbb 1_{v(t)>0}.\label{eq:Poissrhs}
\end{multline}
As for all $t$, \eqref{eq:Poisslhs} and \eqref{eq:Poissrhs} coincide for all $j\ge 1$, by identification we obtain that 
\[v^\prime(t)=-\left(1+{1\over 1-e^{-\rho v(t)}}\right)\mathbb 1_{v(t)>0},\quad t \in [0,1],\]
as desired. 
It then follows from the initial conditions, that $v(0)=1$, and thus the second relation of \eqref{eq:Poisssyst1} follows from that of \eqref{eq:Poisssyst2}, together with \eqref{eq:Poissrhs2} and \eqref{eq:Poissrhs5}. The uniqueness of the solution to \eqref{eq:PoissODE} is a straightfoward consequence of Cauchy-Lipschitz Theorem. 
Finally, \eqref{eq:truc} follows by the very definition of the matching coverage. 
\end{proof}

The matching coverage for the {\sc greedy} criterion can then be numerically estimated, by solving numerically the 1-dimensional ODE \eqref{eq:PoissODE} (using again a Runge-Kutta K4 numerical scheme with mesh $10^{-5}$), and then integrating numerically the corresponding solution, 
to compute \eqref{eq:truc}. As is illustrated on Figure 2, we retrieve exactly, for any mean degree $\rho$, the matching coverage 
\begin{equation}\label{eq:matchcovJaillet}1-{\log(2-e^{-\rho})\over \rho},\end{equation}
although we do not have an analytical proof of this equality at this point. Interestingly enough, \eqref{eq:matchcovJaillet} is precisely the matching coverage of an online matching algorithm on a bipartite Erdös-R\'enyi graph of asymptotic degree Poisson $\mathcal P(\rho)$ (see \cite{MJ13}), and also coincides with the asymptotic matching coverage of an online matching algorithm on a bipartite Configuration model graph of degree distribution $\mathcal P(\rho)$, see Section 2.3.2 of \cite{NSP21}. 
Observe however, that such Erdös-R\'enyi graphs do {\em not} have the law of a Configuration Model (even if it is has the same limiting degree distribution) 
and second, that the online matching algorithms on bipartite graphs and on general graphs (as in the present study) are not equivalent.  

\begin{figure}[ht]	
			\centering
			\includegraphics[width=.65\textwidth]{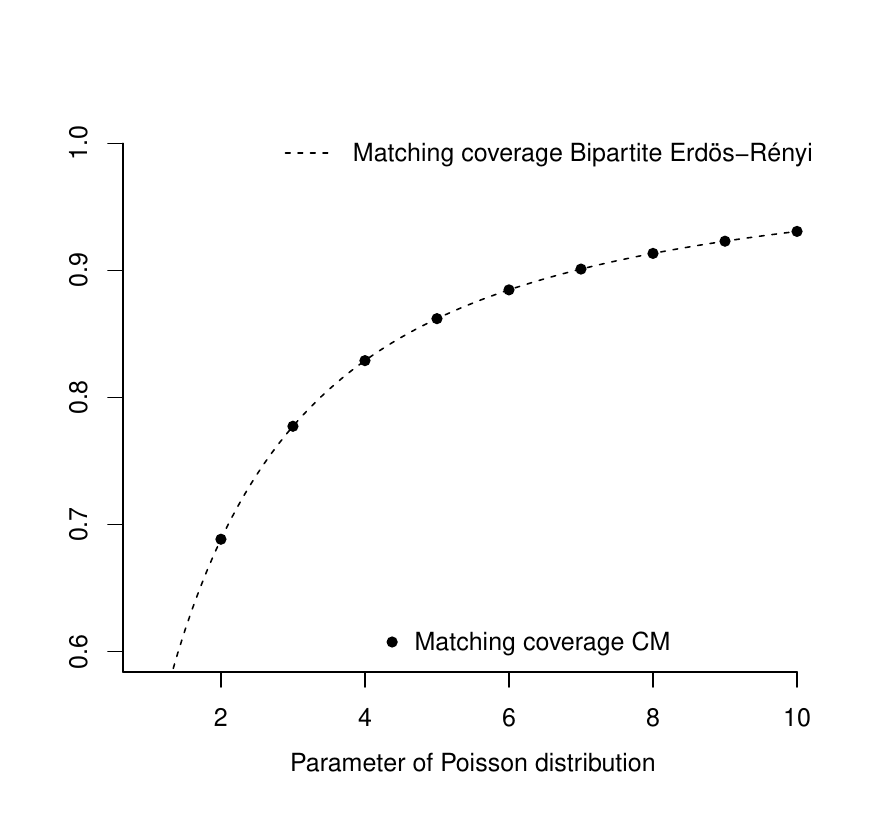}
					\label{fig:GreedyVSUnimin}
	\caption{Bullets: Estimate of the matching coverages for {\sc greedy} for Poisson degree distributions $\mathcal P(\rho)$. Dashed lines: the mapping 
	$\rho\mapsto 1-{\log(2-e^{-\rho})\over \rho}\cdot$}
\end{figure}

Last, we illustrate the influence of the degree distribution on the performance of the matching algorithm, by comparing the matching coverages for the matching criterion {\sc greedy}, for various degree distributions of same mean, namely: 
\begin{itemize}
\item The {\sc uni-min} matching algorithm on $d$-Regular graphs and Uniform degrees on $\llbracket 1,2d-1 \rrbracket$ (left figure of Figure 3);
\item The {\sc greedy} matching algorithm on Regular graphs of degree $d$, Uniform degrees on $\llbracket 1,2d-1 \rrbracket$, and Poisson degrees of parameter $d$ (right figure of Figure 3).
\end{itemize}
We observe that, for these two matching criteria, a model with deterministic degrees outperforms uniform degrees with the same mean, and Poisson degrees in the second case. Moreover, both algorithms seem to perform better as the variance of the degree distribution gets smaller. 

\begin{figure}[h!]			
			\centering
				\begin{subfigure}{0.4\textwidth}
				\includegraphics[width=.99\textwidth]{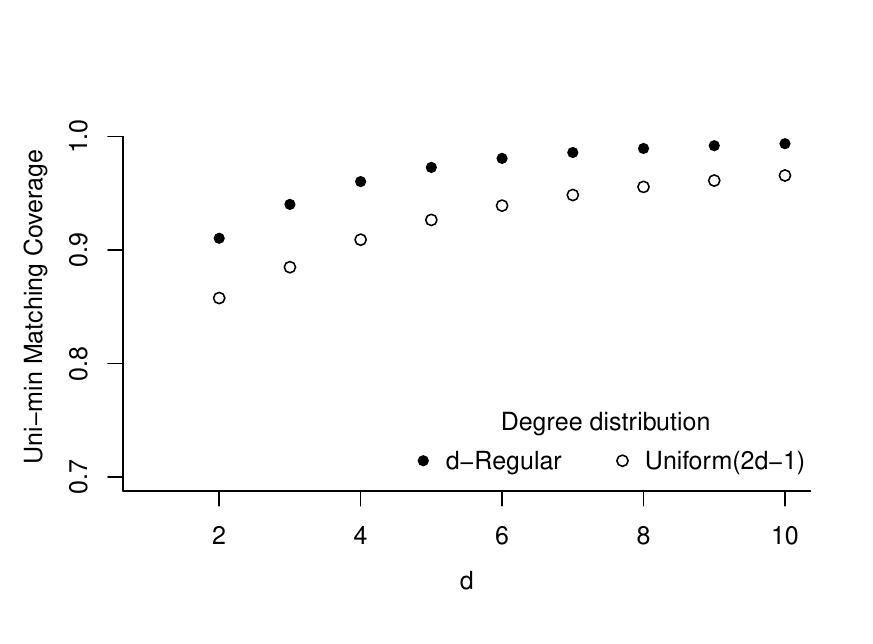}
				\end{subfigure}
			\begin{subfigure}{0.4\textwidth}
				\includegraphics[width=.99\textwidth]{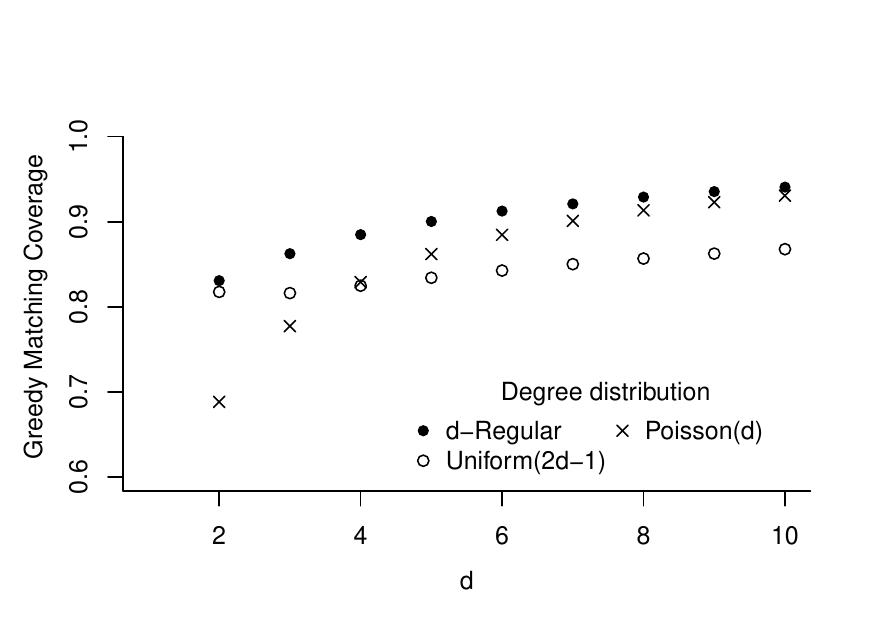}
			\end{subfigure}
	\label{fig:GreedyVSUnimin}
	\caption{Left: Matching coverages for {\sc uni-min} on $d$-regular graphs (bullets) and Uniform degrees on $\llbracket 1,2d-1 \rrbracket$ (circles). 
	Right: Matching coverages for {\sc greedy} on $d$-regular graphs (bullets), Uniform degrees on $\llbracket 1,2d-1 \rrbracket$ (circles) and Poisson degrees of parameter $d$ (crosses).}
\end{figure}

\subsection{Outline}
The rest of the paper is organized as follows: In Section \ref{sec:localalgo}, we introduce a wider class of matching algorithms which we call local, and that includes both {\sc greedy} and {\sc uni-min}, and extend Algorithms \ref{algo:greedy}-\ref{algo:uniminCM} to this more general class. This formalization also allows us to introduce the objects and processes which will be useful for the analysis. In particular, the construction of Section  
\ref{subsec:localCM} on the CM, can be represented by a measure-valued Discrete Time Markov Chain (DTMC), that is introduced in detail in Section \ref{sec:Markov}. 
Our main convergence result, Theorem \ref{thm:main} shows the weak convergence of a sequence of (interpolated version of) this DTMC, in the large-graph limit, to the solution of a deterministic system of ODE's. This result is shown in Section \ref{sec:proof}, using a coupling with a simpler, and more easily tractable dynamics, presented in Section \ref{sec:constructionhat}. Theorems \ref{thm:matchcovgreedy}  and Theorem \ref{thm:matchcovunimin} are then obtain as corollaries of Theorem \ref{thm:main}, in Sections \ref{sec:greedy} and \ref{sec:unimin}, respectively.

\section{Local matching algorithms}
\label{sec:localalgo}
In this section, we define formally a class of matching algorithms, which we call {\em local}, in a sense that will be specified below, both on a an existing graph, and on the Configuration model. This class of algorithms includes the algorithms \textsc{greedy} and \textsc{uni-min}, defined in Section \ref{sec:mainres}. 
For this, we first need to formally define the notion of {\em matching criterion}. 

\subsection{Matching criteria}

\noindent The two matching algorithms introduced in Section \ref{subsec:mainresalgo} can be roughly described as follows: At any step, 
in the remaining graph of unmatched nodes, choose 
a node uniformly at random, and then match it to a node 
that is also chosen uniformly at random, or uniformly at random among the nodes having the smallest degree. The matching criterion is then {\em local} in the sense that the choices of the two matched nodes at any step, depend at most on the degrees of the nodes, and no other information on their neighborhoods. 
We will now formalize this broad class of algorithms. For doing so, we first need to introduce the notions of {\em choice functions} and {\em matching criteria}, as follows.  

\begin{definition}
\label{def:choice}
Let $q\in\N_+$.  
A {\em choice function} on $(\N_+)^q$ is a (possibly random) mapping 
\[{\Phi}_{q}:\begin{cases}
	\N^{q} &\longrightarrow\llbracket 1,q \rrbracket ;\\
	\mathbb x=\left(x(1),\cdots,x(q)\right) &\longmapsto \Phi_q(\mathbb x).
	\end{cases}\]
	\end{definition}

%

\begin{definition}
\label{def:localcriterion}
A \emph{local matching criterion} ${\mathbf \Phi}$ on $\cup_{q\in\N_+}(\N_+)^q$, 
is a family of couples of choice functions 
\[{\mathbf \Phi}=\left\{({\Phi}_{q},\Phi'_q)\,:\,q\in\N_+\right\}.\]
Then, for any $q\in\N_+$ and any $\x=(x(1),\cdots,x(q))\in(\N_+)^q$, we denote by 
\[\mathbf \Phi(\x) = (\Phi_q(\x),\Phi'_q(\x))\in \llbracket 1,q \rrbracket^2.\]
\end{definition}


In the constructions hereafter, a local matching criterion will typically be used to choose at each iteration, first, a node in a graph following a first choice function $\Phi_{|\maA|}$ on some suitable set of nodes $\maA$ (identified with $\llbracket 1,|\maA|\rrbracket$), characterized by their degrees; and second, its {\em match}, namely, a second node chosen following a second choice function $\Phi'_{|\maB|}$ on some set of nodes $\maB$, that may possibly depend of the first chosen node. 
Let us introduce the class of criteria that will be addressed in the constructions below.

\begin{Ex}\label{ex:greedy}\rm 
Accordingly to Algorithm \ref{algo:greedy}, we say that the local matching criterion ${\mathbf{\Phi}}$ is {\sc greedy}, if for any $q$, for any 
$\mathbb x=\left(x(1),\cdots,x(q)\right)\in(\N_+)^q$, $\Phi_q(\x)$ and $\Phi'_q(\x)$ are sampled from the uniform distribution on 
$\llbracket 1,q \rrbracket$, independent of everything else. 
\end{Ex}

\begin{Ex}\label{ex:unimin}\rm
For any $q$, for any  
$\mathbb x=\left(x(1),\cdots,x(q)\right)\in(\N_+)^q$, 
set $x_{\scriptsize{\textsc{min}}} = \min\{x(i)\,:\,i\in\llbracket 1,q \rrbracket\} $ the minimal coordinate of $\mathbb x$, and 
	\[\textsc{argmin}(\mathbb x) =\{ i\in\llbracket 1,q \rrbracket \, : \,x(i) = x_{\scriptsize{\textsc{min}}}\}.\]
	Accordingly to Algorithm \ref{algo:unimin}, we say that the local matching criterion 
	${\mathbf{\Phi}}$ is {\sc uni-min}, if for any $q$ and any $\mathbb x\in(\N_+)^q$, ${\Phi}_q(\x)$ 
	is a uniform sample on $\llbracket 1,q \rrbracket$, and ${\Phi}'_q(\x)$ is a uniform sample on $\textsc{argmin}(\x)$, both independent of everything else. 
	We say that the criterion ${\mathbf{\Phi}}$ is {\sc min-min}, if for any $q\in\N_+$ and any $\x\in(\N_+)^q$, 
	$\Phi_q(\x)$ and $\Phi'_q(\x)$ are two uniform samples on  $\textsc{argmin}(\x)$, 
	independent of everything else. 
\end{Ex}
\begin{Ex}\label{ex:unimax}
\rm Similarly to Example \ref{ex:unimin}, for any $q\in\N_+$ and any $\x\in(\N_+)^q$, define $\textsc{argmax}(\x)$ similarly to $\textsc{argmin}(\mathbb x)$ 
	for {\em maximal} coordinates of $\x$. Then, we say that the local matching criterion ${\mathbf{\Phi}}$ is {\sc uni-max} if  for all $q$ and $\x$ as above, 
	${\Phi}_q(\x)$ is a uniform sample on $\llbracket 1,q \rrbracket$ and ${\Phi}'_q(\x)$ is a uniform sample on $\textsc{argmax}(\mathbb x)$, both independent 
	of everything else. We also say that ${\mathbf{\Phi}}$ is {\sc max-max} if for all $q$ and $\x$,  
	 $\Phi_q(\x)$ and $\Phi'_q(\x)$ are two 
	uniform samples on  $\textsc{argmax}(\x)$, independent of everything else. 
\end{Ex}

\subsection{Local matching algorithms on finite graphs}
\label{subsec:localfinite}
Throughout this section, fix $\mr{G}(\mr{V},\mr{E})$, a (simple, non oriented) graph of size $|\mr{V}|=n$, with vertex set $\mr{V}=\{v_1,\cdots,v_n\}$ and 
edge set $\mr{E}$. When necessary, we identify the set $\mr{V}$ of nodes of $\mr{G}$ with $\llbracket 1,n \rrbracket$. 
For any $i\in\llbracket 1,n \rrbracket$, we denote by 
$d(i)$, the degree of node $v_i$ in $\mr{G}$. We also fix a local matching criterion ${\bphi}$ on $\mr{V}$, denoted by 
\[{\mathbf \Phi}=\left\{\left({\Phi}_{n},{\Phi}'_{n}\right)\,:\,n\in\N_+\right\}.\] 
We now construct a matching algorithm associated to ${\bphi}$. 
At any iteration $j$, we are given two disjoint subgraphs of $\mr{G}$:
\begin{itemize}
\item $\mr{G}_j=(\mr{U}_j, \mr{E}_j)$ is the {\em undiscovered part} of the graph for our procedure, at iteration $j$. The nodes of $\mr{U}_j$ are said to be {\em available}. Those are the nodes whose fate is still to be determined. We also denote, for any $v\in \mr{U}_j$, by $d_j(v)$, the degree of node $v$ in the graph $\mr{G}_j$. 
\item $\mr{G}'_j=(\mr{M}_j,\mr{E}'_j)$ is the {\em matching} at iteration $j$. It is a subgraph of $\mr{G}$ in which all nodes have degree one. 
\end{itemize}
We also define $\mr{I}_j$ as the set of {\em isolated} nodes at iteration $j$, that is, nodes that do neither belong to $\mr U_j$, nor to 
$\mr M_j$, in a way that $\mr V$ can be partionned into 
\[\mr{V}=\mr{U}_j \cup \mr{M}_j\cup \mr{I}_j. \]
In the construction to come, isolated nodes will not be matched at all, because all of their neighbors have already been matched. 

At first, the whole graph is available to be matched, that is, we set 
\[\mr{U}_0=\mr{V},\,\mr{M}_0=\emptyset\quad \mbox{and} \quad\mr{I}_0=\emptyset.\] 
By construction, we also have $d_0(v)=d(v)$ for all $v\in\mr V$. 
Let us also set $\mr{E}_0=\mr{E}$ and $\mr{E}'_0=\emptyset$, in a way that 
$\mr{G}_0=\mr{G}$ and $\mr{G}'_0=(\emptyset, \emptyset)$. 
The matching algorithm on $\mr{G}$ then proceeds as follows. 
\begin{figure}[h!]			
					\centering
					\begin{tikzpicture}[scale = 0.3]
\coordinate (p1) at (0,5);
\coordinate (p2) at (3.53,3.53);
\coordinate (p3) at (5,0);
\coordinate (p4) at (3.53,-3.53);
\coordinate (m1) at (0,-5);
\coordinate (m2) at (-3.53,-3.53);
\coordinate (m3) at (-5,0);
\coordinate (m4) at (-3.53,3.53);

\draw[color=gray!60] (p1) -- (m4);
\draw[color=gray!60] (p2) -- (m1);
\draw[color=gray!60] (p2) -- (m2);
\draw[color=gray!60] (p2) -- (m3);
\draw[color=gray!60] (p2) -- (m4);
\draw[color=gray!60] (p3) -- (m2);
\draw[color=gray!60] (p3) -- (m3);
\draw[color=gray!60] (p4) -- (m3);
\draw[color=gray!60] (p4) -- (m4);
\draw (p1)[color=black,fill] circle (0.2);
\draw[color=black,fill] (p2) circle (0.2);
\draw[color=black,fill] (p3) circle (0.2);
\draw[color=black,fill] (p4) circle (0.2);
\draw[color=black,fill] (m1) circle (0.2);
\draw[color=black,fill] (m2) circle (0.2);
\draw[color=black,fill] (m3) circle (0.2);
\draw[color=black,fill] (m4) circle (0.2);
					\end{tikzpicture}
						\label{fig:init}
						\caption{The graph $\mr{G}_j$ at iteration $j$.}
						
\end{figure}
\begin{enumerate}
\item[{\bf Step $\mathring{0}$.}] If $\mr{U}_j=\emptyset$, go to {Step $\mathring{4}$.} Else, go to {Step $\mathring{1}$.} 
\item[{\bf Step $\mathring{1}$.}] Let (or draw) $\mathring I={\Phi}_{|\mr U_j|}((d_j(v)\,:\,v\in \mr U_j))$. 
%
%
%
%

\item[{\bf Step $\mathring{2}$.}] Let (or draw) 
$\mathring I'={\Phi}'_{|\mr{E}_j(\mathring I)|}((d_j(v)\,:\,v\in \mr{E}_j(\mathring I)))$, and say 
that that node $\mathring I'$ is the {\em match} of $\mathring I$. 
The edge $\{\mathring{I},\mathring{I}'\}$ is added to the matching, and so we set 
\[\left\{\begin{array}{ll}
\mr{U}^-_{j+1}&=\mr{U}_j\setminus\{\mathring{I},\mathring{I}'\},\\
\mr{M}_{j+1}&=\mr{M}_j\cup\{\mathring{I},\mathring{I}'\},\\
\mr{E}'_{j+1}&=\mr{E}'_j\cup\{\{\mathring{I},\mathring{I}'\}\}. 
\end{array}\right.\]
Figure 5 illustrates two criteria: the edge added to the matching is represented in red for {\sc greedy} and {\sc uni-min} respectively, 
with a common draw for the first node $\mathring{I}$. 
\begin{figure}[h!]			
			\centering
				\begin{subfigure}{0.4\textwidth}
				{
					\centering
					\begin{tikzpicture}[scale = 0.3]
\coordinate (p1) at (0,5);
\coordinate (p2) at (3.53,3.53);
\coordinate (p3) at (5,0);
\coordinate (p4) at (3.53,-3.53);
\coordinate (m1) at (0,-5);
\coordinate (m2) at (-3.53,-3.53);
\coordinate (m3) at (-5,0);
\coordinate (m4) at (-3.53,3.53);

\draw[color=gray!60] (p1) -- (m4);
\draw[color=black] (p2) -- (m1);
\draw[color=black] (p2) -- (m2);
\draw[color=red] (p2) -- (m3);
\draw[color=black] (p2) -- (m4);
\draw[color=gray!60] (p3) -- (m2);
\draw[color=gray!60] (p3) -- (m3);
\draw[color=gray!60] (p4) -- (m3);
\draw[color=gray!60] (p4) -- (m4);
\draw (p1)[color=black,fill] circle (0.2);
\draw[color=red,fill] (p2) circle (0.2);
\draw (p2)+(1,0) node {\textcolor{red}{$\mathring{I}$}};
\draw[color=black,fill] (p3) circle (0.2);
\draw[color=black,fill] (p4) circle (0.2);
\draw[color=black,fill] (m1) circle (0.2);
\draw[color=black,fill] (m2) circle (0.2);
\draw[color=red,fill] (m3) circle (0.2);
\draw (m3)+(-1,0) node {\textcolor{red}{${\mathring{I}'}$}};
\draw[color=black,fill] (m4) circle (0.2);

					\end{tikzpicture}
						\label{fig:S2}
						\caption{$\textsc{greedy}$ criterion}
						}
				\end{subfigure}
			\begin{subfigure}{0.4\textwidth}
			{
					\centering
					\begin{tikzpicture}[scale = 0.3]
\coordinate (p1) at (0,5);
\coordinate (p2) at (3.53,3.53);
\coordinate (p3) at (5,0);
\coordinate (p4) at (3.53,-3.53);
\coordinate (m1) at (0,-5);
\coordinate (m2) at (-3.53,-3.53);
\coordinate (m3) at (-5,0);
\coordinate (m4) at (-3.53,3.53);

\draw[color=gray!60] (p1) -- (m4);
\draw[color=red] (p2) -- (m1);
\draw[color=black] (p2) -- (m2);
\draw[color=black] (p2) -- (m3);
\draw[color=black] (p2) -- (m4);
\draw[color=gray!60] (p3) -- (m2);
\draw[color=gray!60] (p3) -- (m3);
\draw[color=gray!60] (p4) -- (m3);
\draw[color=gray!60] (p4) -- (m4);
\draw (p1)[color=black,fill] circle (0.2);
\draw[color=red,fill] (p2) circle (0.2);
\draw (p2)+(1,0) node {\textcolor{red}{$\mathring{I}$}};
\draw[color=black,fill] (p3) circle (0.2);
\draw[color=black,fill] (p4) circle (0.2);
\draw[color=red,fill] (m1) circle (0.2);
\draw[color=black,fill] (m2) circle (0.2);
\draw[color=black,fill] (m3) circle (0.2);
\draw (m1)+(-1,0) node {\textcolor{red}{${\mathring{I}'}$}};
\draw[color=black,fill] (m4) circle (0.2);

				\end{tikzpicture}
					\label{fig:S3}
					\caption{$\textsc{uni-min}$ criterion}
					}
			\end{subfigure}
	\label{fig:poli}
	\caption{Choice of an edge}
\end{figure}

\item[{\bf Step $\mathring{3}$.}] The matched nodes $\mathring{I}$ and ${\mathring{I}'}$, as well as their neighboring edges, have now been explored, and are thus removed from the unexplored graph (in grey). The nodes of null degree in the remaining graph are moved from the class of `Unexplored' to the class of `Isolated' nodes. Specifically, set 
\[\left\{\begin{array}{ll}
\mr{G}_{j+1}&=\mbox{Induced subgraph of }\mr{U}^-_{j+1}\mbox{ in }\mr{G}_j,\\
\mr{I}_{j+1}&=\mr{I}_j\cup\{v\in\mr{G}_{j+1}\,:\,d_{j+1}(v)=0\},\\
\mr{U}_{j+1}&=\mr{U}^-_{j+1}\setminus\{v\in\mr{G}_{j+1}\,:\,d_{j+1}(v)=0\}.
\end{array}\right.\]
\begin{figure}[h!]			
			\begin{center}
				\begin{subfigure}{0.4\textwidth}
				{
					\centering
					\begin{tikzpicture}[scale = 0.3]
\coordinate (p1) at (0,5);
\coordinate (p2) at (3.53,3.53);
\coordinate (p3) at (5,0);
\coordinate (p4) at (3.53,-3.53);
\coordinate (m1) at (0,-5);
\coordinate (m2) at (-3.53,-3.53);
\coordinate (m3) at (-5,0);
\coordinate (m4) at (-3.53,3.53);

\draw[color=gray!60] (p1) -- (m4);
\draw[color=gray!60] (p3) -- (m2);
\draw[color=gray!60] (p4) -- (m4);
\draw (p1)[color=black,fill] circle (0.2);
\draw[color=black,fill] (p3) circle (0.2);
\draw[color=black,fill] (p4) circle (0.2);
\draw[color=black,fill] (m1) circle (0.2);
\draw[color=black,fill] (m2) circle (0.2);
\draw[color=black,fill] (m4) circle (0.2);

				\end{tikzpicture}
					\label{fig:S2f}
					\caption{$\textsc{greedy}$: End of iteration}
					}
							\end{subfigure}
							\begin{subfigure}{0.4\textwidth}
				{
				\centering
					\begin{tikzpicture}[scale = 0.3]
\coordinate (p1) at (0,5);
\coordinate (p2) at (3.53,3.53);
\coordinate (p3) at (5,0);
\coordinate (p4) at (3.53,-3.53);
\coordinate (m1) at (0,-5);
\coordinate (m2) at (-3.53,-3.53);
\coordinate (m3) at (-5,0);
\coordinate (m4) at (-3.53,3.53);

\draw[color=gray!60] (p1) -- (m4);
\draw[color=gray!60] (p3) -- (m2);
\draw[color=gray!60] (p3) -- (m3);
\draw[color=gray!60] (p4) -- (m3);
\draw[color=gray!60] (p4) -- (m4);
\draw (p1)[color=black,fill] circle (0.2);
\draw[color=black,fill] (p3) circle (0.2);
\draw[color=black,fill] (p4) circle (0.2);
\draw[color=black,fill] (m1) node{ };
\draw[color=black,fill] (m2) circle (0.2);
\draw[color=black,fill] (m3) circle (0.2);
\draw[color=black,fill] (m4) circle (0.2);

				\end{tikzpicture}
					\label{fig:S3f}
					\caption{$\textsc{uni-min}$: End of iteration}
					}
			\end{subfigure}
		\end{center}	
\label{fig:iter}
\caption{Explored part of the graph after an iteration}
\end{figure}

\item[{\bf Step $\mathring{4}$.}] Set $j:=j+1$. If $j=n$, terminate the procedure. Else, go to {Step $\mathring{0}$}. 
\end{enumerate}
\medskip

At the terminating point $n$, all the nodes have necessarily been investigated, and all are either matched or isolated. We get 
$$|\mr{I}_{n}|+|\mr{M}_{n}|=|\mr{V}|=n.$$
Extending the definition \eqref{eq:defmatchingcov} to any local matching criterion, the {\em matching coverage} 
$\mathring{\mathbf M}^n_\bphi(\mr{G})$ is then the proportion of initial nodes that end up in the matching at the termination time $n$. 
It can thus be expressed as a simple function of the resulting sets, 
\begin{equation}
\label{eq:ratiolocal}
\mathring{\mathbf M}^n_\bphi(\mr{G})=\frac{|\mr{M}_{n}|}{n}=1-\frac{|\mr{I}_{n}|}{n}\, \,\in [0,1],
\end{equation}

\subsection{Local matching on the configuration model}
\label{subsec:localCM}
In this section, we use the classical procedure of sequential uniform pairing, to produce a realization of a multi-graph by the configuration model.  
We also transpose the local matching algorithm introduced in Section \ref{subsec:localfinite} to the resulting multi-graph, by a simultaneous construction. 
For this, we let $\xi$ be a probability measure on $\N$, and $n$ be a positive integer. Let $\mathbb d:=(d(1),...,d(n))$ be an 
$n$-sample of the probability distribution $\xi$, called {\em degree vector} of the multi-graph. We assume that $\sum_{i=1}^n d(i)$ is even 
(if not, we just substract 1 to an arbitrary positive component of $\mathbb d$). 
We let $$\mu := \sum_{i\leq n} \delta_{d(i)}\in\M^n,$$ be the corresponding \textit{degree measure}. We let $\mathcal{V}=\{v_1,\cdots,v_n\}$ be the set of nodes of the multi-graph to be constructed. 
For all $i \in\llbracket 1,n \rrbracket$, $d(i)$ will be interpreted as the {\em degree} of node $v_i$ in the latter multi-graph. See Figure 7. 
For every $i$, the $d(i)$ half-edges of $v_i$ are to be completed into edges by uniform pairing, and we initially set $a_0(v_i)=d(i)$. 

Fix a local matching criterion 
$\mathbf \Phi$ on $\maV$, 
\[{\mathbf \Phi}=\left\{\left({\Phi}_{n},{\Phi}'_{n}\right)\,:\,n\in\N_+\right\}.\] 
We can now set the trackers for the number of \textit{available} stubs (or half edges) of the nodes. Stubs are to be paired sequentially and {uniformly}, into edges. 
\begin{figure}
\centering
\label{confmodinit}
\caption{A running example: Initial state for $\mathbb{d} = (3,2,1,4,2,2)$}
\begin{tikzpicture}[node distance = {40mm},thick, main/.style = {draw,circle}]
\foreach \ang [count=\n from 1] in {90,30,...,-210}
\node[main] (v\n) at (\ang:2.5cm) {$v_\n$};
\draw (v1.-135) -- +(-135:0.3);
\draw (v4.120) -- +(120:0.3);
\draw (v1.-90) -- +(-90:0.3);
\draw (v3.135) -- +(135:0.3);
\draw (v1.-45) -- +(-45:0.3);
\draw (v2.-180) -- +(180:0.3);
\draw (v2.-135) -- +(-135:0.3);
\draw (v4.100) -- +(100:0.3);
\draw (v4.80) -- +(80:0.3);
\draw (v4.60) -- +(60:0.3);
\draw (v5.45) -- +(45:0.3);
\draw (v5.0) -- +(0:0.3);
\draw (v6.0) -- +(0:0.3);
\draw (v6.-45) -- +(-45:0.3);
\end{tikzpicture}
\end{figure}
At first, all stubs are available and the graph $\mathcal{G}_0$ has no edges. The matching is initially empty, and all the edges have to be discovered. 
We define the following initial sets,   
\[\begin{cases}\mathcal{U}_0 &=\{v\in \maV\,:\,a_0(v)>0\}:=\left\{v_{0_1},\cdots,v_{0_{p_0}}\right\};\\
\mathcal{I}_0 &=\{v\in \maV\,:\,a_0(v)=0\};\\
\maE_0 &=\emptyset;\\
\maM_0 &=\emptyset;\\
\maE'_0&=\emptyset;\\
\maB_0&=\emptyset,
\end{cases}\]
where we observe that $\mathcal{I}_0$ is not necessarily the empty set, in the case where the degree distribution has a mass at zero. 
In our running example, Figure 7, 
we have $p_0=6$ and 
\[a_0(v_1)=3;\,a_0(v_2)=2;\,a_0(v_3)=1;\,a_0(v_4)=4;\,a_0(v_5)=2\mbox{ and }a_0(v_6)=2.\]
We also let $\mathcal{G}'_0=(\emptyset,\emptyset)$ be the empty graph.

\bigskip

\noindent In a similar fashion to Section \ref{subsec:localfinite}, we shall proceed by induction. At iteration $j$, we are given:
		
\begin{itemize}
	\item A multi-graph $\mathcal{G}_j=(\mathcal{V},\mathcal{E}_j)=(\maM_j\cup \,\mathcal{U}_j\cup \mathcal{I}_j,\mathcal{E}_j)$, representing the partially constructed connections between elements of $\mathcal{V}$, 
where we denote by:
		\begin{itemize}
		\item $\maM_j$ the set of {\em matched} nodes at $j$, which are nodes that are fully attached to the multi-graph at $j$ 
		(no available stubs), and belong to the matching at $j$; 
		\item $\mathcal{U}_j$ the set of {\em unexplored} nodes at $j$, that is, nodes that do not belong to the matching at $j$, but can still be attached to it since they have available stubs, which will become edges that can possibly be added to the matching. The nodes of the set $\maU_j$ are indexed as 		
		\[\maU_j=\left\{v_{j_1},\cdots,v_{j_{p_j}}\right\},\]
		for some positive $p_j\le n$. For any $l\in\llbracket 1,p_j \rrbracket$, we denote by $a_j(v_{j_l})>0$, the {\em availability} of node 
		$v_{j_l}\in \maU_j$, that is, the number of its available stubs. We then identify $v_{j_l}$ with a `bunch' of $a_j(v_{j_l})$ stubs. 
		\item $\mathcal{I}_j$, the set of {\em isolated} nodes at $j$, that is, nodes that are already fully attached to the multi-graph at $j$, have no more 
		available stubs, but do not belong to the matching at $j$. These nodes have been isolated because their last completed adjacent edge pointed to 
		an already matched node. 
		\end{itemize}
	\item A maximum {\em matching} $\mathcal{G}'_j=(\maM_j,\mathcal{E}'_j)$ 
       on the induced subgraph of $\maM_j$ in $\mathcal{G}_j$. In particular, 
      $\mathcal{E}'_j$ is a set of pairs of $\maM_j$ of the form $\{v_i,v_{i'}\}$ for $v_i,v_{i'}\in \maM_j$, such that any element of $\maM_j$ appears in exactly one pair of $\mathcal{E}'_j$. By our very construction, all nodes of $\maM_j$ (if any) will have degree at least $1$ in 
$\mathcal{G}_j$. 
      \item $\maB_j$, the set of {\em blocked} nodes at $j$: those are isolated nodes in the same sense as above, 
		but whose last completed adjacent edge was a self-loop. See case 2a) below. 
\end{itemize}
\noindent Our joint construction of a multi-graph, and of a matching on it, goes as follows: 
\begin{enumerate}
 \item[{\bf Step 0.}] We initially have the following cases:
 \begin{itemize}
 \item[{\bf 0a)}] If $\maU_j=\emptyset$ (no more unexplored nodes at $j$), set 
 \[\begin{cases}
  \maU_{j+1} &=\maU_j=\emptyset;\\
  \maI_{j+1} &=\maI_j;\\
   a_{j+1}(v) &=a_j(v)=0,\,v\in\maI_{j+1};\\
  \maE_{j+1}&=\maE_{j};\\
  \maM_{j+1}&=\maM_j;\\
  \mathcal{E}'_{j+1}&=\mathcal{E}'_j;\\
   \maB_{j+1}&=\maB_j,
  \end{cases}\]
 and go to {Step 5}. 
 \item[{\bf 0b)}] Else, go to Step 1. 
 \end{itemize}
  \item[{\bf Step 1.}] Let (or draw) $$I=\Phi_{|\maU_j|}\left(a_j(v_{j_{1}}),\cdots,a_j(v_{j_{p_j}})\right),$$
  and let $K\equiv a_j(I)$. Let $q\in \llbracket 1,p_j\rrbracket$ be such that $I=v_{j_{q}}$. 
    We then apply the uniform pairing procedure, to complete the $K$ stubs of $I$ into edges. More specifically, we draw, 
  {without repetition}, $K$ half-edges $H_1,\cdots,H_K$, uniformly at random among $p_j$ bunches of half-edges of respective sizes 
  $a_j\left(v_{j_l}\right)$, $l \in \llbracket 1,p_j \rrbracket$, to be paired with the half-edges of $I$, to complete the emanating edges of node $I$. 
  Note that this operation may lead to parallel edges or self-loops, whenever several elements of the same bunch of half-edges are chosen. 
  Denote by $I_1,...,I_{\tilde K}$ the indexes of the (possibly equal) bunches to which the $K$ half-edges $H_1,\cdots,H_K$ belong. 
  Note that these indexes may include 
  repetitions in case of multiple edges and/or $I$ itself in case of self-loops, and that  $\tilde K\le K$, with an equality if $I$ has no self-loop. Then we update the set of edges by setting  
$$\mathcal{E}_{j+1^-}=\mathcal{E}_j\,\cup\, \bigl\{\{I,I'_1\},...,\{I,I'_{\tilde K}\}\bigl\},$$
where $\{I,I\}$ is understood as a self-loop at $I$, and the repetition of the same edge $\{I,J\}$ is understood as a multiple edge between $I$ and $J$.   
    Then, we let 
  $$\mathcal{N}_j(I) := \{v_{j_{i_1}},...,v_{j_{i_{\ell}}}\} \subset \left\{v_{j_{1}},...,v_{j_{p_j}}\right\},$$ for some $\ell\equiv\ell_j\le K$, be the set of the $\ell$ 
  {neighbors} of $I$ other than $I$ itself, if any, i.e., the distinct nodes different from $I$, appearing in the list $(I'_1,...,I'_{\tilde K})$, if any. 
For all $l \in \llbracket 1,p_j \rrbracket\setminus\{q\}$, we let $b(v_{j_l})\equiv b_j(v_{j_l})$ be the number of {edges} shared by $v_{j_l}$ with $I$, that is, the (possibly null) number of elements in the bunch associated to $v_{j_l}$, that were chosen in the uniform pairing procedure. In other words, $b(v_{j_l})$ is the number of repetitions of the node $v_{j_l}$ in the list $(I'_1,...,I'_{\tilde K})$. Go to Step 2. 

\medskip

\begin{figure}[h!]			
								\centering
					\begin{tikzpicture}[scale =0.8,node distance = {2.5cm},thick, main/.style = {draw,circle}]
						\def\match{1}
						\def\matchee{0}
						\def\mytestcolor{black}
						\foreach \ang [count=\n from 1] in {90,30,...,-210} {
						\ifnum\n=\match\relax\def\mytestcolor{red}\fi
						\ifnum\n=\matchee\relax\def\mytestcolor{red}\else\fi
						\node[main,color = \mytestcolor] (v\n) at (\ang:2.5cm) {$v_\n$};
						}
						\draw (v1.-135) to[bend left] (v6.0);
						\draw (v1.-90) to[bend right] (v3.135);
						\draw (v1.-45) to[bend right] (v2.180);
						\draw (v4.120) -- +(120:0.3);
						\draw (v2.-135) -- +(-135:0.3);
						\draw (v4.100) -- +(100:0.3);
						\draw (v4.80) -- +(80:0.3);
						\draw (v4.60) -- +(60:0.3);
						\draw (v5.45) -- +(45:0.3);
						\draw (v5.0) -- +(0:0.3);
						\draw (v6.-45) -- +(-45:0.3);	

					\end{tikzpicture}
					\label{fig:e1}
\caption{Chosing $I = v_1$ and discovering its neighbors.}
\end{figure}

In our running example (Figure 8), say that $\Phi$ is the random uniform choice over $\maU_0$, and we draw $I=v_1$. Then we get 
$K=\ell = 3$, $\maN_0(I)=\{v_2,v_3,v_6\}$ and 
\[b(v_1)=0;\,b(v_2)=1;\,b(v_3)=1;\,b(v_4)=0;\,b(v_5)=0\mbox{ and }b(v_6)=1.\]

\item[{\bf Step 2.}] 
We have the following cases: 
  \begin{itemize}
  \item[{\bf 2a)}] If $\ell=0$, i.e., at iteration $j$, $I$ has no other new neighbor than itself, and in particular we get that 
  $b(v_{j_l})=0$ for all $l \in \llbracket 1,p_j \rrbracket\setminus\{q\}$. Then, we say that node $I$ is blocked: 
  it is now fully attached to the graph (with at least one self-loop), but is not matched. We then set $b'(v_{j_l})\equiv b'_j(v_{j_l})=0$ for all 
  $l \in \llbracket 1,p_j \rrbracket$, and then fix 
  \begin{equation}
      \label{eq:lossmasskbis}
      a_{j+1}(v_{j_l})=a_j(v_{j_l}),\quad l \in \llbracket 1,p_j \rrbracket\setminus\{q\} 
      \end{equation}
      and 
  \[\begin{cases}
  \maU_{j+1} &=\maU_j\setminus\{I\};\\
  \maI_{j+1} &=\maI_j;\\
  a_{j+1}(v) &=a_j(v)=0,\,v\in\maI_{j+1};\\
  \maE_{j+1}&=\maE_{j+1^-};\\
  \maM_{j+1}&=\maM_j;\\
  \mathcal{E}'_{j+1}&=\mathcal{E}'_j;\\
  \maB_{j+1}&=\maB_j.
  \end{cases}\]
  Then go to Step 5. 
  \medskip 
  \item[{\bf 2b)}] If $\ell>0$, $I$ has other new neighbors than itself, and we pick the match $I'$ of $I$ within the set 
  $\{v_{j_{i_1}},...,v_{j_{i_{\ell}}}\}$. For this, let (or draw)  
\[I'=\Phi'_{|\mathcal{N}_j(I)|}\left(a_j(v_{j_{i_1}}),\cdots,a_j(v_{j_{i_{\ell}}})\right),\]
and let $K'=a_j(I')$ be the availability of node $I'$. Let also $r\in \llbracket 1,p_j\rrbracket$ be such that $I'=v_{j_{r}}$. 
Henceforth, both nodes $I$ and $I'$ together with the edge $\{I,I'\}$ are added to the matching $\mathcal{G}'_{j}$, that is, we set
\[\begin{cases}
  \maU_{j+1^-} &=\maU_j\setminus\{I,I'\};\\
  \maI_{j+1^-} &=\maI_j;\\
  \maM_{j+1}&=\maM_j\cup \{I,I'\};\\
  \mathcal{E}'_{j+1}&=\mathcal{E}'_j\cup \bigl\{\{I,I'\}\bigl\};\\
   \maB_{j+1}&=\maB_j,
  \end{cases}\]
and go to Step 3. 

\medskip

\indent In the running example (Figure 9), we present two matching criteria for a common draw $I=v_1$: On the left-hand subfigure, 
$\mathbf{\Phi}$ is {\sc greedy} and we obtain $I'=v_2$ and $K'=2$. On the right-hand, $\mathbf{\Phi}$ is {\sc uni-min} and we get $I'=v_3$ and $K'=1$. 
\begin{figure}[ h!]			
			\begin{center}
				\begin{subfigure}{0.4\textwidth}
					\centering
					\begin{tikzpicture}[scale =0.8,node distance = {2.5cm},thick, main/.style = {draw,circle}]
						\def\match{1}
						\def\matchee{2}
						\def\mytestcolor{black}
						\foreach \ang [count=\n from 1] in {90,30,...,-210} {
						\ifnum\n=\match\relax\def\mytestcolor{red}\fi
						\ifnum\n=\matchee\relax\def\mytestcolor{red}\else\fi
						\node[main,color = \mytestcolor] (v\n) at (\ang:2.5cm) {$v_\n$};
						}
						\draw (v1.-135) to[bend left] (v6.0);
						\draw (v1.-90) to[bend right] (v3.135);
						\draw[ultra thick,red] (v1.-45) to[bend right] (v2.180);
						\draw (v4.120) -- +(120:0.3);
						\draw (v2.-135) -- +(-135:0.3);
						\draw (v4.100) -- +(100:0.3);
						\draw (v4.80) -- +(80:0.3);
						\draw (v4.60) -- +(60:0.3);
						\draw (v5.45) -- +(45:0.3);
						\draw (v5.0) -- +(0:0.3);
						\draw (v6.-45) -- +(-45:0.3);	
					\end{tikzpicture}
					\label{gre:e2}
					\caption{$\textsc{greedy}$ chooses the match uniformly}
				\end{subfigure}
				\begin{subfigure}{0.4\textwidth}
					\centering
					\begin{tikzpicture}[scale =0.8,node distance = {2.5cm},thick, main/.style = {draw,circle}]
						\def\match{1}
						\def\matchee{3}
						\def\mytestcolor{black}
						\foreach \ang [count=\n from 1] in {90,30,...,-210} {
						\ifnum\n=\match\relax\def\mytestcolor{red}\fi
						\ifnum\n=\matchee\relax\def\mytestcolor{red}\else\fi
						\node[main,color = \mytestcolor] (v\n) at (\ang:2.5cm) {$v_\n$};
						}
						\draw (v1.-135) to[bend left] (v6.0);
						\draw[ultra thick,red] (v1.-90) to[bend right] (v3.135);
						\draw (v1.-45) to[bend right] (v2.180);
						\draw (v4.120) -- +(120:0.3);
						\draw (v2.-135) -- +(-135:0.3);
						\draw (v4.100) -- +(100:0.3);
						\draw (v4.80) -- +(80:0.3);
						\draw (v4.60) -- +(60:0.3);
						\draw (v5.45) -- +(45:0.3);
						\draw (v5.0) -- +(0:0.3);
						\draw (v6.-45) -- +(-45:0.3);	
					\end{tikzpicture}
					\label{minres:e2}
					\caption{$\textsc{uni-min}$ only has one choice}
				\end{subfigure}
		\end{center}	
\label{fig:e2}
\caption{Matching $I$ with $I'$.}
\end{figure}

\end{itemize}

\item[{\bf Step 3.}]
We determine the edges adjacent to node $I'$. We have the following cases: 
		\begin{itemize} 
      \item[{\bf 3a)}] If $K'=b(I')$, then $I'$ has no more open half-edges to complete. In this case we do not do anything, and just set 
      \[\mathcal{E}_{j+1}=\mathcal{E}_{j+1^-}.\]
      This is the case 
      in the right-hand subfigure of Figure 9. 
      \item[{\bf 3b)}] If $K' > b(I')$, $I'$ still has $K'-b(I')$ incomplete half-edges. To complete them into edges, we reiterate the same procedure of uniform pairing 
      as in Step 1, to 
         determine the indexes $(I_1,\cdots,I_{\tilde K'})$ of the endpoints of all edges emanating from $I'$, other than the already drawn $\{I,I'\}$ edge. 
         The neighborhood of node $I'$ is now complete, and we set 
         \[\mathcal{E}_{j+1}=\mathcal{E}_{j+1^-}\,\cup\, \bigl\{\{I',I_1\},...,\{I',I_{\tilde K'}\}\bigl\},\]
         using the same notational convention as at Step 1. 
      We let $$\left\{v_{j_{s_1}},...,v_{j_{s_{m}}}\right\} \subset \left\{v_{j_1},...,v_{j_{p_j}}\right\}\setminus\{v_{j_q}\},$$ 
      for $m \le  a_j\left(I'\right) - b(I')$,  be the set of the $m$ neighbors of $I'$ other than $I$ and $I'$ itself, i.e., the distinct nodes different from $I$ and $I'$, 
      appearing in the list $(I_1,...,I_{\tilde K'})$, if any. 
      In the left-hand subfigure of Figure 10, we then have $\tilde K'=1$ and $I_1=v_{j_{s_1}}=v_5$. 
      		\end{itemize} 
		\begin{figure}[ h!]			
			\begin{center}
				\begin{subfigure}{0.4\textwidth}
					\centering
					\begin{tikzpicture}[scale =0.8,node distance = {2.5cm},thick, main/.style = {draw,circle}]
						\def\match{1}
						\def\matchee{2}
						\def\mytestcolor{black}
						\foreach \ang [count=\n from 1] in {90,30,...,-210} {
						\ifnum\n=\match\relax\def\mytestcolor{red}\fi
						\ifnum\n=\matchee\relax\def\mytestcolor{red}\else\fi
						\node[main,color = \mytestcolor] (v\n) at (\ang:2.5cm) {$v_\n$};
						}
						\draw (v1.-135) to[bend left] (v6.0);
						\draw (v1.-90) to[bend right] (v3.135);
						\draw[ultra thick,red] (v1.-45) to[bend right] (v2.180);
						\draw (v2.-135) to[bend left] (v5.45);
						\draw (v4.120) -- +(120:0.3);
						\draw (v4.100) -- +(100:0.3);
						\draw (v4.80) -- +(80:0.3);
						\draw (v4.60) -- +(60:0.3);
						\draw (v5.0) -- +(0:0.3);
						\draw (v6.-45) -- +(-45:0.3);	
					\end{tikzpicture}
					\label{gre:e3}
					\caption{$\textsc{greedy}$: $I'=v_2$ has another neighbor}
				\end{subfigure}
				\begin{subfigure}{0.4\textwidth}
					\centering
					\begin{tikzpicture}[scale =0.8,node distance = {2.5cm},thick, main/.style = {draw,circle}]
						\def\match{1}
						\def\matchee{3}
						\def\mytestcolor{black}
						\foreach \ang [count=\n from 1] in {90,30,...,-210} {
						\ifnum\n=\match\relax\def\mytestcolor{red}\fi
						\ifnum\n=\matchee\relax\def\mytestcolor{red}\else\fi
						\node[main,color = \mytestcolor] (v\n) at (\ang:2.5cm) {$v_\n$};
						}
						\draw (v1.-135) to[bend left] (v6.0);
						\draw[ultra thick,red] (v1.-90) to[bend right] (v3.135);
						\draw (v1.-45) to[bend right] (v2.180);
						\draw (v4.120) -- +(120:0.3);
						\draw (v2.-135) -- +(-135:0.3);
						\draw (v4.100) -- +(100:0.3);
						\draw (v4.80) -- +(80:0.3);
						\draw (v4.60) -- +(60:0.3);
						\draw (v5.45) -- +(45:0.3);
						\draw (v5.0) -- +(0:0.3);
						\draw (v6.-45) -- +(-45:0.3);	
					\end{tikzpicture}
					\label{minres:e3}
					\caption{$\textsc{uni-min}$: no other operation}
				\end{subfigure}
		\end{center}	
\label{fig:e3}
\caption{Completing the neighborhood of $I'$.}
\end{figure}

\noindent Go to Step 4. 
\bigskip
\bigskip 

\item[{\bf Step 4.}] 
For all $l \in \llbracket 1,p_j \rrbracket\setminus\{q,r\}$, let  $b'(v_{j_l})\equiv b'_j(v_{j_l})$ be the number of {edges} shared by $v_{j_l}$ with $I'$. 
This number is 0 in case 3a), and in case 3b), it is given by the (possibly null) 
number of elements in the bunch associated to $v_{j_l}$ that were chosen in the uniform pairing procedure. 
We then update the availabilities of all nodes of $\maU_j$, by setting 
      \begin{equation}
      \label{eq:lossmassk}
      a_{j+1}(v_{j_l})=a_j(v_{j_l})-b(v_{j_l})-b'(v_{j_l}),\quad l \in \llbracket 1,p_j \rrbracket\setminus\{q,r\}.
      \end{equation}
      Finally, the set $\maU_{j+1}$ of unexplored nodes at step $k+1$ is obtained from $\maU_j$ by deleting $I$ and $I'$ 
      (which are now elements of $\mathcal M_{j+1}$), together with all elements of $\maU_j$ whose availability is null at step $k+1$, if any, since these 
      nodes become isolated. In other words, set
      \[\begin{cases}
      \maU_{j+1}&=\maU_{j+1^-}\setminus\left\{v_{j_l}\,:\,a_{j+1}(v_{j_l})=0,\,l\in\llbracket 1,p_j\rrbracket \setminus\{q,r\}\right\};\\
      \mathcal I_{j+1}&=\mathcal I_{j+1^-}\cup \left\{v_{j_l}\,:\,a_{j+1}(v_{j_l})=0,\,l\in\llbracket 1,p_j\rrbracket \setminus\{q,r\}\right\};\\
      a_{j+1}(v) &=a_j(v)=0,\,v\in\maI_{j+1}.
      \end{cases}\]
      We then re-index the set $\maU_{j+1}$ as \[\maU_{j+1}:=\left\{v_{(j+1)_1},\cdots,v_{(j+1)_{p_{j+1}}}\right\},\]
      for some $p_{j+1}\le p_j$. Note that, by the very procedure of uniform pairing, the total sum of availabilities 
      $\sum_{l=1}^{p_{j+1}}a_{j+1}(v_{(j+1)_l})$ at iteration $j+1$ is even, because it has the same parity as the sum of availabilities 
      $\sum_{l=1}^{p_{j}}a_{j}(v_{j_l})$ at iteration $j$.  Now go to Step 5. 
     
 \medskip

In the left-hand subfigure of Figure 10, we then have 
      \[b'(v_1)=1;\,b'(v_2)=0;\,b'(v_3)=0;\,b'(v_4)=0;\,b'(v_5)=1\mbox{ and }b'(v_6)=0,\]
while in the right-hand subfigure of Figure 10, we obtain $\maN_0(I')=\{I\}$ and 
      \[b'(v_1)=1\mbox{ and }\,b'(v_2)=b'(v_3)=b'(v_4)=b'(v_5)=b'(v_6)=0.\]
      In our running example, in the case of {\sc greedy} we end up at step $1$ with 
      \[\begin{cases}
      \maU_{1}&=\{v_4,v_5,v_6\},\mbox{ with }a_1(v_4)=4,\,a_1(v_5)=1,\mbox{ and }a_1(v_6)=1,\\
      \mathcal M_{1}&=\{v_1,v_2\},\\
      \mathcal I_1&=\{v_3\}.
      \end{cases}\]
      Regarding {\sc uni-min}, we obtain
      \[\begin{cases}
      \maU_{1}&=\{v_2,v_4,v_5,v_6\},\mbox{ with }a_1(v_2)=1,\,a_1(v_4)=4,\,a_1(v_5)=2,\mbox{ and }a_1(v_6)=1,\\
      \mathcal M_{1}&=\{v_1,v_2\},\\
      \mathcal I_1&=\emptyset.
      \end{cases}\]

\item[{\bf Step 5.}] Set $k:=k+1$. If $k=n$, terminate the procedure. Else, go to {Step ${0}$}.
\end{enumerate}

\medskip

As in Section \ref{subsec:localfinite}, the procedure terminates at step $n$, where we necessarily  have $\mathcal{U}_n=\emptyset$. 
 At that time, we end up with a multi-graph $\mathcal{G}:=\mathcal{G}_n:=(\mathcal{V},\mathcal{E}_n)$, since all stubs have been completed. 
 Moreover all the nodes are either matched or isolated. Extending \eqref{eq:defmatchingcovCM} to any local matching criterion, the {matching coverage} is re-expressed similarly to \eqref{eq:ratiolocal}: 
\begin{equation}
\label{eq:ratioCM}
\mathbf M^n_\bphi(\xi)={|{\maM}_n|\over n}=1-{|\mathcal{I}_n|\over n}-{|\mathcal{B}_n|\over n}\, \,\in [0,1].
\end{equation}

\section{Measure-valued representation and Markov dynamics}
\label{sec:Markov}

We first address a comparison between the matching algorithms respectively defined in Sections \ref{subsec:localfinite} and \ref{subsec:localCM}. 
We let, for all $j\in\llbracket 0,n \rrbracket$, 
\begin{itemize}
\item $\mathring{\mu}_j$
be the empirical degree distribution of all unexplored nodes at time $j$ in the remaining graph $\mr{G}_j$ associated to the construction of Section \ref{subsec:localfinite}, that is, 
\begin{equation}
\label{eq:defmuzero}
\mathring{\mu}_j= \sum_{v \in \mr{U}_j\cup\mr{I}_j}\delta_{d_j(v)} \in \M^{|\mr{U}_j|+|\mr{I}_j|};
\end{equation}
\item $\mu_j$ be the analogous empirical distribution representing the availabilities of all unexplored nodes at $j$, in the construction of Section \ref{subsec:localCM}. Specifically, 
\begin{equation}
\label{eq:defmu}
\mu_j=\sum_{v \in \mathcal{U}_j\cup\mathcal{I}_j}\delta_{a_j(v)} \in \M^{|\mathcal{U}_j|+|\mathcal{I}_j|}. 
\end{equation}
\end{itemize}
Let us also observe the following immediate representation of the matching coverage, 
\begin{proposition}
\label{prop:Matchingcover}
In the construction of Section \ref{subsec:localCM}, the associated matching coverage can be rewritten as 
\be
\label{eq:ratioCMmu}
{\mathbf {M}}_\bphi(\mu_0) = 1-{\mu_n(0)\over n}-{{\maB}_n\over n},
\ee
where $\maB_n$ is the final number of blocked nodes in the construction. 
\end{proposition} 
\begin{proof}
At each iteration, the Dirac masses $\delta_j$ and (in case 2b), whenever a matching does occur) $\delta_{K'}$ are removed from the degree measure, while the masses associated to their neighboring nodes lose mass according to \eqref{eq:lossmasskbis} and \eqref{eq:lossmassk}. 
Eventually, at step $n$, when all the edges in the resulting multi-graph have been created, the only remaining Dirac masses 
in $\mu_n$  are those associated to isolated nodes with no availability. In other words, the final set of unmatched nodes 
consists of all nodes having availability $0$ at iteration $n$, and the nodes that have been blocked in case 2a) of Step 2. 
Then \eqref{eq:ratioCMmu} follows from \eqref{eq:ratioCM}. 
\end{proof}

The following result establishes a connection between the two constructions. It is a formalization of the intuitive fact, that the matching criteria behave similarly whenever they are ran on a pre-constructed graph, and on a graph sampled from the CM of same degree distribution. 
\begin{theorem}
\label{thm:coupling}
Let $\mr{G}(\mr{V},\mr{E})$ be a (simple) graph of size $n$ and degree vector  $\mathbb{d}$, and let $\suiten{\mr{\mu}_j}$ be the measure-valued random sequence defined by (\ref{eq:defmuzero}), with initial value 
\begin{equation*} 
\mr{\mu}_0 =\sum_{i=1}^n \delta_{d(i)}.
\end{equation*}
Let $\mathcal{G}(\mathcal V,\mathcal E)$ be the resulting multi-graph of the construction of Section \ref{subsec:localCM}, corresponding to $n$ nodes and an initial degree vector $\mathbb d$. Suppose that $\mathcal V=\mr{V}$, and fix a common local matching criterion $\mathbf \Phi$ on $\maV=\mr{V}$. 
Then, for any $j\in\llbracket 0,n \rrbracket$ and any measure $\nu \in \mathcal M_F^n$ we get that 
\[\pr{\mu_j = \nu \mid \mathcal{G} = \mr{G}} = \pr{\mr{\mu}_j = \nu}.\]
\end{theorem}   

\begin{proof}
Suppose that $\mathcal{G} = \mr{G}$, that is, the final result of the second construction produces the graph 
$\mr{G}$ on which the first construction is deployed. 
We then index the nodes of $\mathcal{V}=\mr{V}$ consistently in the two constructions. By assumption, we initially have that 
\begin{equation*}
\begin{cases}
\mr{U}_0 &=\mathcal{U}_0=\maV;\\
a_0(v) &=d_0(v)=d(v),\quad\mbox{ for all }v\mbox{ in }\mr{U}_0=\mathcal{U}_0;\\
\mr{M}_0 &=\emptyset = \mathcal M_0;\\
\mr{E}(v)\cap \mr{M}_0 &=\emptyset=\mathcal E_0(v),\quad\mbox{ for all }v\mbox{ in }\mr{U}_0=\mathcal{U}_0;\\
\mr{I}_0 &= \emptyset =\mathcal I_0;\\
\maB_0 &=\emptyset. 
\end{cases}
\end{equation*}
The result is then obtained by induction on $j$. Suppose that, at some time $j\in \llbracket 0,n-1 \rrbracket$ we have 
\begin{equation}
\label{eq:HR}
\begin{cases}
\mr{U}_j &=\mathcal{U}_j\\
a_j(v) &=d_j(v),\quad\mbox{ for all }v\mbox{ in }\mr{U}_j=\mathcal{U}_j,\\
\mr{M}_j &= \mathcal M_j;\\
\mr{E}(v)\cap \mr{M}_j &=\mathcal E_j(v),\quad\mbox{ for all }v\mbox{ in }\mr{U}_j=\mathcal{U}_j;\\
\mr{I}_j &=\mathcal I_j;\\
\maB_j &=\emptyset, 
\end{cases}
\end{equation}
meaning that at iteration $j$: (i) The sets of unexplored nodes are the same in the two constructions; (ii) For any such unexplored node, the degree in the remaining graph of the first construction equals its availability in the second construction; (iii) The matched nodes are the same in the two constructions; 
(iv) For any unexplored node, the neighboring nodes amongst the already matched nodes in the first construction coincide with the neighbors in the already constructed graph in the second construction; (v) The sets of isolated nodes are the same in the two constructions and (vi) There are no blocked nodes in the second construction. 

\medskip

First, in view of \eqref{eq:HR} note that, 
if $\,\mr{U}_j=\maU_j=\emptyset$ (case 0a)) then we keep everything unchanged at iteration $j+1$ 
in both algorithms, so \eqref{eq:HR} remains valid. Then, in case 0b),   
at Step $\mathring{1}$ and $1$ respectively we can set a common realization of 
$${\Phi}_{\mr{U}_j}((d_j(v)\,:\,v\in \mr{U}_j))\quad \mbox{and}\quad{\Phi}_{\mathcal{U}_j}((a_j(v)\,:\,v\in \mathcal{U}_j)),$$
leading to the same values for $\mathring{I}$ and $I$. Now, as $\mathcal G=\mr{G}$, node $I$ has the same neighborhood in $\mr{G}$ and $\mathcal G$. 
But from \eqref{eq:HR}, the neighbors of $I$ amongst the matched nodes in the first construction are exactly the already connected neighbors of $I$ in the graph $\mathcal G_j$ in the second. Therefore, the uniform pairing procedure at Step 2 necessarily leads to the same set of additional neighbors for $I$ as the set of remaining neighbors of $I$ in $\mr{G}_j$, namely 
$$\mr{E}_j(\mathring I)=\maN_j(I)=\{v_{j_{i_1}},...,v_{j_{i_\ell}}\}.$$ 
Then, as $\mr{G}$ is a graph, at Step 2, case 2a) cannot occur. So $I$ necessarily has neighbors in $\maG_j$, and so  
$$\maB_{j+1}=\emptyset.$$ 
We can again set a common realization of 
$${\Phi}'_{\mr{E}_j(\mathring I)}((d_j(v)\,:\,v\in \mr{E}_j(\mathring I)))\quad \mbox{and}\quad{\Phi}'_{\mathcal{N}_j(I)}((a_j(v)\,:\,v\in \mathcal{N}_j(I))),$$
leading to $\mathring{I}' = I'$. We then deduce that  
\[\begin{cases}
\mr{U}_{j+1}&=\mr{U}_j \setminus \{\mathring{I},\mathring{I}'\}=\mathcal{U}_j\setminus \{ I,I'\}=\mathcal{U}_{j+1};\\
\mr{M}_{j+1}&=\mr{M}_j\cup\{\mathring I,\mathring I'\}=\mathcal M_j\cup\{I,I'\}=\mathcal M_{j+1}.
\end{cases}\]
From the same argument as above, the uniform pairing procedure leads again to the same set of additional neighbors for $I'$ as the set of remaining neighbors of $\mathring I'$ in $\mr{G}_j$, namely
$$\mr{E}_j(\mathring I')=\{v_{j_{s_1}},...,v_{j_{s_m}}\}\cup\{v_{j_q}\}.$$ 
Therefore, for any $v\in \mr{U}_{j+1}=\mathcal U_{j+1}$ we are in the following cases:
\begin{itemize}
\item If $v\in\maN_j(I)\cap\maN_j(I)$, $v$ was chosen as a common neighbor of $I$ and $I'$ at step $j$, implying, as $\mr{G}=\mathcal G$, that it is 
a common neighbor of $\mathring I$ and $\mathring I'$ in $\mathcal G$. Then from \eqref{eq:HR} we get 
$$\mathcal E_{j+1}(v)=\mathcal E_j(v)\cup\{I,I'\}=(\mr{E}(v)\cap \mr{M}_j)\cup\{\mathring I,\mathring I'\}=\mr{E}(v)\cap \mr{M}_{j+1}.$$
Also, as $\mr{G}$ is a (simple) graph, 
$v$ loses exactly 2 neighbors at iteration $j+1$ of the first algorithm, and we also have that $b(v)=b'(v)=1$ in the second one. 
Thus \eqref{eq:lossmassk} implies 
\[a_{j+1}(v)=a_j(v)-2=d_j(v)-2=d_{j+1}(v).\]
\item Likewise, if $v\in\maN_j(I)\cap\maN_j(I)^c$, then $v$ shares an edge with $\mathring I$ and not with $\mathring I'$ in $\mr{G}$, and so
$$\mathcal E_{j+1}(v)=\mathcal E_j(v)\cup\{I\}=(\mr{E}(v)\cap \mr{M}_j)\cup\{\mathring I\}=\mr{E}(v)\cap \mr{M}_{j+1}.$$
Moreover, $v$ loses exactly one neighbor at iteration $j+1$ of the first algorithm, and so 
$$a_{j+1}(v)=a_j(v)-1=d_j(v)-1=d_{j+1}(v).$$
\item Similarly, if $v\in\maN_j(I)^c\cap\maN_j(I)$, we get 
$$\mathcal E_{j+1}(v)=\mathcal E_j(v)\cup\{I'\}=(\mr{E}(v)\cap \mr{M}_j)\cup\{\mathring I'\}=\mr{E}(v)\cap \mr{M}_{j+1};$$
\[a_{j+1}(v)=a_j(v)-1=d_j(v)-1=d_{j+1}(v).\]
\item Finally, if $v\in\maN_j(I)^c\cap\maN_j(I)^c$, we readily obtain that  
$$\mathcal E_{j+1}(v)=\mathcal E_j(v)=(\mr{E}(v)\cap \mr{M}_j)=\mr{E}(v)\cap \mr{M}_{j+1};$$
\[a_{j+1}(v)=a_j(v)=d_j(v)=d_{j+1}(v).\]
\end{itemize}
In particular, the nodes that become isolated (because their degree / their availability becomes null) at iteration $j+1$ are the same in the two constructions, in other words
\[\mr{I}_{j+1}\setminus \mr{I}_j = \mathcal I_{j+1} \setminus \mathcal I_{j},\]
which implies in turn that $\mr{I}_{j+1}=\mathcal I_{j+1}$. As a conclusion, assertion (\ref{eq:HR}) holds at iteration $j+1$. 

Thus we can conclude that there exists a coupling of the two systems, such that (\ref{eq:HR}) holds for all $j\in\llbracket 0,n \rrbracket$. Finally, (\ref{eq:defmuzero}) and (\ref{eq:defmu}) imply in particular that $\mathring{\mu}_j = \mu_j$ for all $j\in\llbracket 0,n \rrbracket$, concluding the proof. 

\end{proof}

The latter result establishes that, whenever the configuration model produces a given graph, the local matching algorithm simply behaves like an exploration algorithm on the prescribed graph. Therefore, to capture the asymptotic properties of the latter, it is sufficient to study the dynamics of the former. 
In the following Proposition, we write any point measure $\mu\in\M^n$ as 
\begin{equation}
\label{eq:indexingmu}
\mu :=\sum_{l=1}^{\mu\left(\N_+\right)} \delta_{a_{l}(\mu)}+\mu(0)\delta_0,
\end{equation}
where a sum over an empty set is fixed to zero, and where the $a_{l}(\mu),\,l\in \llbracket 1,\mu\left(\N_+\right)\rrbracket$, are the non-zero atoms of $\mu$, if any, indexed arbitrarily. 

\begin{proposition}
\label{prop:Mchain}
For any local matching criterion $\mathbf{\Phi}$, the 
sequence $\procn{\mu_j}$ defined by \eqref{eq:defmu}, is an homogeneous DTMC having transition kernel
\be
\label{eq:defL}
\mathfrak{L}F(\mu) := \E_{\mu} \( F(\mu_0 + \vartheta(\mu_0))  - F(\mu_0)  \),
\ee
for $\mu\in \M^n$, $F$ a bounded continuous function: $\M^n \rightarrow \R$, and for $\vartheta(\mu)$ defined as 
\be
\label{eq:defvartheta}
    \vartheta(\mu) =
    -\left\{\delta_{K_{{\mathbf \Phi}}(\mu)} + \left(\delta_{K'_{{\mathbf \Phi}}(\mu)}+\sum_{l=1}^{\mu\left(\N_+\right)} \(\delta_{a_{l}(\mu)} - \delta_{a_{l}(\mu) - b_{\bphi,l}(\mu)-b'_{\bphi,l}(\mu)} \)\right)\mathbb 1_{\mathscr E(\mu)}\right\}\mathbb 1_{\{\cro{\mu,\chi}>0\}}, 
\ee
where $\mathscr E(\mu)$ is an event independent of everything else and $K_{{\mathbf \Phi}}$, $K'_{{\mathbf \Phi}}$, $b_{\bphi,l}$ and $b'_{\bphi,l},\,l\in\N_+$,  
are mappings that depend only on draws independent of everything else, to be specified 
in the proof below. 
\end{proposition} 

\begin{proof}
Recall the construction of Section \ref{subsec:localCM}, and fix $\mu\in \M^n$, written as in \eqref{eq:indexingmu}. Fix also $j\in \llbracket 0,n-1 \rrbracket$. First observe that if $\mu_{j}=\mu$ 
such that $\mu(\N_+)=0$, then there are no more unexplored nodes at $j$ (case 0a)), and we get 
$\mu_{j+1}=\mu_j$ and thereby $\vartheta(\mu)=\mathbf 0$. 

Otherwise , we are in case 0b) and we have that $p_j=|\maU_j|=\mu(\N_+)$. We can then re-index the nodes of $\maU_j$ in a way that $a_{l}(\mu)=a_j(v_{j_l})$ for all $l\in\llbracket 1,p_j \rrbracket=\llbracket 1,\mu(\N_+) \rrbracket$. 
Then, at Step 1 we have that 
\begin{equation*}
I=v_{j_q}=\Phi_{|\maU_j|}\left(a_j(v_{j_1}),\cdots,a_j(v_{j_{p_j}})\right)=\Phi_{\mu\left(\N_+\right)}\left(a_\mu(1),\cdots,a_\mu(\mu\left(\N_+\right))\right).
\end{equation*}
Thus we can set 
\begin{equation}
\label{eq:defK}
K=a_j(I)=a_j(v_{j_q})=a_\mu(q)=:K_{\mathbf \Phi}(\mu),
\end{equation}
where $K_{\mathbf \Phi}$ is a mapping : $\M^n\to \N$, depending only on $\mathbf \Phi$ on a draws that are independent of everything else. 

Now, observe that the number $\ell$ of neighbors of $I$ in the uniform pairing procedure can also be written as $\ell=\ell(\mu)$, where $\ell(.)$ is a random mapping that depends only on $K_{\mathbf \Phi}(\mu)$, and on draws that are independent of everything else. 
Thus, given that $\mu_j=\mu$, the event 
\begin{equation}
\label{eq:defeventE}\mathscr E(\mu):=\{\mbox{node $I$ has new neighbors different from $I$ itself at iteration $j$}\}
\end{equation}
is indeed independent of $k$, and indeed depends only on $\mu$ and on independent draws. 
We have the following cases: 
\begin{itemize}
\item On the event $\{\mu_j=\mu\}\cap \mathscr E(\mu)^c$, we are in case 2a), i.e., $\ell(\mu)=0$. Then node $I$ is blocked, we let $K'_{\mathbf \Phi}(\mu)\equiv 0$, and in view of \eqref{eq:lossmasskbis} we obtain that 
\[\mu_{j+1} =\mu-\delta_K=\mu-\delta_{K_{\mathbf\Phi}(\mu)}. 
\]
\item On the event $\{\mu_j=\mu\}\cap \mathscr E(\mu)$, we are in case 2b). For all 
$l=1,...,\ell(\mu)$, the index $j_{i_l}$ of the $l$-th neighbor of $I$ obtained by uniform pairing, is just a function of $\mu$ and of independent draws. 
It can thus be written as $j_{i_l}=:j_{i_l(\mu)}$, where $i_l(.)$ is also a mapping that depends only on independent draws. Thus, at Step 2 we get  that 
\begin{equation*}
I'=v_{j_r}=\Phi'_{|\maN_j(I)|}\left(a_j(v_{j_{i_1}}),\cdots,a_j(v_{j_{i_{\ell}}})\right)=\Phi'_{\ell(\mu)}\left(a_\mu(i_1(\mu)),\cdots,a_\mu(i_{\ell(\mu)}(\mu))\right),
\end{equation*}
and we can set 
\begin{equation}
\label{eq:defK'}
K'=a_j(I')=a_j(v_{j_r})=a_\mu(r)=:K'_{\mathbf \Phi}(\mu),
\end{equation}
where $K'_{\mathbf \Phi}$ is another mapping depending only on $\mathbf \Phi$ and on independent draws. 
Then the two nodes $I$ and $I'$ are matched, 
and the corresponding Dirac masses at $j$ and $K'$ are substracted from $\mu$. Also, for any 
$l\in\llbracket 1,p_j \rrbracket\setminus\{q,r\}$, 
the number of edges $b(v_{j_l})$ (resp., $b'(v_{j_l})$) shared by $v_{j_l}$ with $I$ (resp., $I'$) after the procedure of uniform pairing of Step 1 (resp. Step 3b), is a quantity $b_{\bphi,l}(\mu)$ (resp., $b'_{\bphi,l}(\mu)$) that depend only on $\mu$ and of independent draws. 
Therefore, from \eqref{eq:lossmassk} we get that 
 \[\mu_{j+1} =\mu_j-\delta_j-\delta_{K'}
=\mu-\delta_{K_{\mathbf\Phi}(\mu)}-\delta_{K'_{\mathbf\Phi}(\mu)}-\sum_{l=1}^{\mu\left(\N_+\right)} \(\delta_{a_{l}(\mu)} - \delta_{a_{l}(\mu) - b_{\bphi,l}(\mu)-b'_{\bphi,l}(\mu)}\),
\]
with the convention $b_{\bphi,q}(\mu)=b_{\bphi,r}(\mu)=b'_{\bphi,q}(\mu)=b'_{\bphi,r}(\mu)=0$. 
\end{itemize} 
This completes the proof. 
\end{proof}

\medskip

{For any $\mu\in\M^n$, let us denote by 
$\maK_{\mathbf \Phi}(\mu)$ the distribution on $\N_+$ of the r.v. $K_{\mathbf \Phi}(\mu)$ defined by \eqref{eq:defK}, that is, 
\begin{equation}
\label{eq:distribK}
\maK_{\mathbf \Phi}(\mu)(k) = \pr{K_{\mathbf \Phi}(\mu)=k},\quad k \in \N_+. 
\end{equation}
Also, for all $k\in\N_+$ we let ${\maK}'_{\bphi}(\mu,k)$ be the distribution on $\N$ of the  r.v. $K'_{\mathbf \Phi}(\mu)$ defined by \eqref{eq:defK'}, 
conditional on $\{K_{\mathbf \Phi}(\mu)=k\}$, namely, 
\begin{equation}
\label{eq:distribK'}
{\maK}'_{\bphi}(\mu,k)(k')=\P\left(K'_{\mathbf \Phi}(\mu)=k'\,|\,j_{\mathbf \Phi}(\mu)=k\right),\,k'\in\N,
\end{equation}
observing that we have by definition \[{\maK}'_{\bphi}(\mu,k)(0)=\P\left({\mathscr E}(\mu)^c\,|\,j_{\mathbf \Phi}(\mu)=k\right),\,k\in\N_+.\]}

%
%
%
%

\bigskip

{To conclude this section, observe that the DTMC $\suiten{\mu_j}$} looses mass over time, 
since the availabilities of the nodes decrease as they loose half-edges. 
We formalize this intuitive fact as follows, 
\begin{proposition}
\label{Momest1}
Let {$f: \N \to \R$} be a nonnegative and increasing function. Then we have that 
\[
\cro{\mu_j,f} \geq \cro{\mu_{j+1},f},\quad j\in \llbracket 0,n-1 \rrbracket. 
\]
\end{proposition}

\begin{proof}
{For all such $j$ and $f$}, in view of Proposition \ref{prop:Mchain} we get that 
\begin{equation*}
\cro{\mu_{j+1},f} - \cro{\mu_j,f} = \cro{\vartheta(\mu_j),f},\quad j\in \llbracket 0,n-1 \rrbracket. 
\end{equation*}
But if follows from \eqref{eq:defvartheta} that for all $\mu$, 
\begin{multline*}
\cro{\vartheta(\mu),f}\\
 = - \left\{ f(K_{{\mathbf \Phi}}(\mu)) + \left(f(K'_{{\mathbf \Phi}}(\mu)) + \sum_{l=1}^{\mu\left(\N_+\right)} \(f\left(a_{l}(\mu)\right) - f\left(a_{l}(\mu) - b_{\bphi,l}(\mu)-b'_{\bphi,l}(\mu)\right) \)\right) \mathbb 1_{\mathscr E(\mu)}\right\}\mathbb 1_{\{\cro{\mu,\chi}>0\}},
\end{multline*}
a quantity that is non-positive since $f$ is non-negative and non-decreasing. 
\end{proof}

\bigskip 
\bigskip

\section{Approximated dynamics}
\label{sec:constructionhat}
\subsection{An alternative construction}
To identify the large-graph limiting behavior of the DTMC $\procn{\mu_j}$ defined in Section \ref{sec:Markov}, it is convenient to 
couple the construction of Section \ref{subsec:localCM} to an alternative one, that is defined as follows. Fix a counting measure $\mu \in \M_p$ such that {$p:=\cro{\mu,\mathbb 1_{\N_+}}>0$}, 
denoted as in \eqref{eq:indexingmu}. 
 To each $i \in \N_+$ are associated $\mu(i)$ non-empty buckets of $i$ items. ({\em Buckets} and {\em items} correspond to {\em nodes} and {\em half-edges} in the original construction.)
So there are $p$ buckets and a total
of $M=\cro{\mu,\chi}$ items. We label the buckets arbitrarily, as $B_1,...,B_p$. 
For any $l=1,...,p$, we set the cardinality of bucket $B_l$ as $a_l:=a_l(\mu)$. We also label the items from $1$ to $M$ as follows: 
items 1 to $a_1$ are the elements of bucket $B_1$ labeled arbitrarily, items $a_1+1,...,a_1+a_2$ are the elements of bucket $B_2$, and so on. 

We perform the following random experiment, {which mimics the dynamics of Section \ref{subsec:localCM} at any given step, but when the sampling is performed with (rather than {\em without}) replacement:}
\begin{itemize}
\item[(i)] We draw $$J=\Phi_p(a_1,...,a_p)\in\llbracket 1,p \rrbracket,$$ 
and set $\wh I:=B_J$, the corresponding bucket. We denote by $\wh{K}=a_J$, its cardinality. 
{Observe that, by identifying buckets with nodes and items with half-edges, $\wh K$ follows the very distribution $\maK_{\bphi}(\mu)$ 
defined by \eqref{eq:distribK}.} 
\item[(ii)] Then we draw uniformly at random, and \emph{with replacement}, $\wh{K}$ (possibly equal) items among $M$,  
denoted by $\wh{H}_1,...,{\wh H}_{\wh K}$. For all $l \in \llbracket 1,\wh K\rrbracket$ we denote by 
$B_{j_l}$, the bucket $\wh H_l$ belongs to. Note, that various indexes $j_l$ may be equal to one another and/or equal to $j$. Then we distinguish two cases: 
	\begin{itemize}
	\item[(iia)] On the event 
	$$\wh{\mathscr E}(\mu)^c:=\left\{\mbox{All indexes $j_l,\,l\in \llbracket 1,\wh K \rrbracket$, are equal to $J$}\right\},$$ 
	we set $\wh{K}':=0$, and terminate the procedure. 
	\item[(iib)] Else, on $\wh{\mathscr E}(\mu)$, we let $\ell$ be the number of distinct buckets, and distinct from $\wh I$, that were chosen at step (i), 
	and denote 
	by $B_{s_1},...,B_{s_\ell}$, these distinct buckets. We draw $$J'=\Phi'_{\ell}(a_{s_1},...,a_{s_{\ell}})\in \llbracket 1,\ell \rrbracket,$$  
        and set $\wh I':=B_{j'}$, the corresponding bucket. Note that we have $J'=j_s$ for some $s\in\llbracket 1,\wh K \rrbracket$. 
        We denote by $\wh{K}'=a_{J'}$ the cardinality of $\wh I'$, and go to step (iii).
	\end{itemize}
	{For all $k\in\N_+$, 
        we let $\wh{\maK}'_{\bphi}(\mu,k)$ be the 
        distribution of $\wh K'$ on $\N$ conditional on $\{\wh K=k\}$, namely, 
\begin{equation}
\label{eq:distribhatK'}
\P_\mu\left(\wh K'=k'\,|\,\wh K=k\right)=\wh{\maK}'_{\bphi}(\mu,k)(k'),\,k'\in\N,
\end{equation}
and by definition we have that 
\[\wh{\maK}'_{\bphi}(\mu,k)(0)=\P_\mu\left(\wh{\mathscr E}(\mu)^c\,|\,\wh K=k\right),\,\mbox{for all }k.\]
However, as the choices are made with replacement, for any $k$ the distribution $\wh{\maK}'_{\bphi}(\mu,k)$ does not coincide 
in general with ${\maK}'_{\bphi}(\mu,k)$ defined by \eqref{eq:distribK'}.}
\medskip
\item[(iii)] We are in the following cases: 
	\begin{itemize} 
	\item[(iiia)] If $\wh K'=1$, we terminate the procedure;
	\item[(iiib)] If $\wh K'>1$, we draw uniformly at random, and with replacement, $\wh{K}'-1$ items among $M$,  
         denoted by ${i}'_1,...,{i}'_{\wh K'-1}$. Then, for all $l \in \llbracket 1,\wh K'-1 \rrbracket$ we denote by $B_{j'_l}$, 
         the bucket $i'_l$ belongs to, and terminate the procedure. 
         \end{itemize}
\end{itemize}
Let us denote, for all $y\in\N_+$, by
\[\begin{cases}
\wh{N}_\mu(y) &=\mbox{Card} \left\{l\in \llbracket 1,\wh K\rrbracket\setminus\{s\}\,:\,a_{j_l}=y\right\};\\
\wh{N}'_\mu(y) &=\mbox{Card} \left\{l\in \llbracket 1,\wh K'-1\rrbracket\,:\,a_{j'_l}=y\right\}\mathbb 1_{\wh K'>1},
\end{cases}\]
the number of (possibly equal) items drawn at step (ii) other than $i_s$ (respectively, at step (iii)) and belonging to a bucket of size $y$, 
%
and define 
{\begin{equation}
\wh{\vartheta}(\mu) =-\left(\delta_{\wh{K}} + \left(\delta_{\wh{K}'}
+\sum_{y=1}^{\infty}
\left(\wh{N}_\mu(y)+\wh{N}'_\mu(y)\right)\left(\delta_{y}-\delta_{y-1}\right)\right)\mathbb 1_{\wh{\maE}(\mu)}
\right)\mathbb 1_{\cro{\mu,\chi}>0}. 
\label{eq:defwhvartheta}
\end{equation}}

\noindent For all $f\in\maC_b$, define the operator $\wh{\mathfrak L}_f$ on $\M^n$ by 
\begin{equation*}
\wh{\mathfrak L}_f(\mu)=\E_\mu\left[\cro{\wh{\vartheta}(\mu),f}\right],\quad \mu\in \M^n.
\end{equation*}
Observe the following result, 
\begin{lemma}
\label{lemma:defLbar}
For all $\mu\in\M^n$ such that $\cro{\mu,\chi}>0$, for all $f\in\maC_b$ we have that 
\begin{equation}
\label{eq:defLhat}
\wh{\mathfrak L}_f(\mu)=- \cro{\maK_{\bphi}(\mu),f+\cro{\wh{\maK}'_{\bphi}(\mu,.),f}}
-{\cro{\mu,\chi\nabla f}\over \cro{\mu,\chi}}
\cro{\maK_{\bphi}(\mu),\chi-\mathbb 1+\cro{{\wh{\maK}'_\bphi(\mu,.)},\chi-\mathbb 1}},
\end{equation}
where the measures $\maK_\bphi(\mu)$ and ${\wh{\maK}'_\bphi(\mu,k)}$, $k\in\N_+$, are defined respectively by \eqref{eq:distribK} and \eqref{eq:distribhatK'}.
\end{lemma}
\begin{proof}
From \eqref{eq:defwhvartheta}, we immediately get that 
\begin{equation}
\wh{\mathfrak L}_f(\mu)=-\E_\mu\left[f\left(\wh{K}\right)\right] - \E_\mu\left[f\left(\wh{K}'\right)\right]
-\sum_{y=1}^{\infty}
\E_\mu\left[\wh{N}_\mu(y)+\wh{N}'_\mu(y)\right]\nabla f(y).\label{eq:stem}
\end{equation}

\medskip

\noindent Now observe, first, that  
\[
\E_{\mu}\left[f\left(\wh K\right)\right] = \cro{\maK_{\bphi}(\mu),f},\]
and second,
\[\E_{\mu}\left[f\left(\wh K'\right)\right] = \displaystyle\sum_{k\in\N_+}\displaystyle\sum_{k'\in\N_+} f(k')\wh{\maK}'_{\bphi}(\mu,k)(k'){\maK}_{\bphi}(\mu)(k)=\cro{\maK_{\bphi}(\mu),\cro{\wh{\maK}'_{\bphi}(\mu,.),f}}.\]
Also, given that $\wh K_{\bphi}(\mu)=k\in\N_+$, at step (ii), for any $l \in \llbracket 1,k \rrbracket$, 
the size $a_{j_l}$ of the bucket to which item $i_s$ belongs follows, independently of everything else, the size-biased distribution associated to $\mu$, 
namely for all $y\in\N_+$, 
\[\mathbb P_\mu(a_{j_l}=y)={y\mu(y)\over \cro{\mu,\chi}}\cdot\]
Consequently, for every $y\in\N_+$, $\wh{N}_\mu(y)$ follows the Binomial distribution with parameters $k-1$ and ${y\mu(y)\over \cro{\mu,\chi}}$, understood as $0$ a.s. if $k=1$.  
Likewise, given that $\wh j_{\bphi}(\mu)=k\in\N_+$ and $\wh K'_{\bphi}(\mu)=k'\in\N_+$, for every $y\in\N_+$, $\wh{N}_\mu(y)$ follows 
the Binomial distribution with parameters $k'-1$ (understood as 0 a.s. if $k'=1$) and ${y\mu(y)\over \cro{\mu,\chi}}$. 
All in all, \eqref{eq:stem} implies that 
\begin{multline*}
\E_\mu\left[\cro{\wh{\vartheta}(\mu),f}\right] =
- \cro{\maK_{\bphi}(\mu),f}-\cro{\maK_{\bphi}(\mu),\cro{\wh{\maK}'_{\bphi}(\mu,.),f}}\\
-\sum_{y=1}^{\infty}\nabla f(y)\sum_{k\in\N_+}
\left\{(k-1){y\mu(y)\over \cro{\mu,\chi}}+\sum_{k'\in\N_+}(k'-1){y\mu(y)\over \cro{\mu,\chi}}\wh\maK'_{\bphi}(\mu,k)(k')\right\}\maK_{\bphi}(\mu)(k)\\
\shoveleft{= - \cro{\maK_{\bphi}(\mu),f}-\cro{\maK_{\bphi}(\mu),\cro{\wh{\maK}'_{\bphi}(\mu,.),f}}
-{\cro{\mu,\chi\nabla f}\over \cro{\mu,\chi}}
\Bigl\{\cro{\maK_{\bphi}(\mu),\chi-\mathbb 1}+\cro{\maK_{\bphi}(\mu),\cro{\wh\maK'_{\bphi}(\mu,.),\chi-\mathbb 1}}\Bigl\}}\\
=- \cro{\maK_{\bphi}(\mu),f+\cro{\wh{\maK}'_{\bphi}(\mu,.),f}}
-{\cro{\mu,\chi\nabla f}\over \cro{\mu,\chi}}
\Bigl\{\cro{\maK_{\bphi}(\mu),\chi-\mathbb 1+\cro{\wh\maK'_{\bphi}(\mu,.),\chi-\mathbb 1}}\Bigl\},
\end{multline*}
as desired. 
\end{proof}

\noindent Now let us define the following event, 
\begin{equation}
\label{eq:defeventT}
\wh{\mathscr T}(\mu) =\left\{\mbox{The buckets }\wh I,\,\,B_{j_1},...,\,B_{j_{\wh K}},\,B_{j'_1},\,...,\,B_{j'_{\wh K'-1}}\mbox{ are all distinct}\right\},
\end{equation}
with an obvious meaning if $\wh K'=1$. Observe that we clearly have  $\wh{\mathscr T}(\mu) \subset \wh{\mathscr E}(\mu)$. 
We have the following result. 
\begin{lemma} 
\label{lemma:simple}
For all $\mu\in\M_p$ having positive first moment and finite third moment, the event defined by \eqref{eq:defeventT} satisfies 
\begin{equation*} 
\P_\mu\left( \wh{\mathscr T}(\mu)^c\right) \le 
 {\cro{\mu,\chi^2} \over \cro{\mu,\chi}^2}\cro{\maK_\bphi(\mu),\chi^2+\cro{{\wh{\maK}'_\bphi(\mu,.)},\chi^2}+(\mathbb 1+2\chi)\cro{{\wh{\maK}'_\bphi(\mu,.)},\chi}}+{\cro{\maK_\bphi(\mu),\chi^2}\over \cro{\mu,\chi}}\cdot 
\end{equation*}
\end{lemma}
\begin{proof}
\noindent Plainly, 
\begin{multline*} 
\P_\mu\left(\wh{\mathscr T}(\mu)^c\right)
\le \P_\mu\left(B_{j_s} = B_{j_l} \mbox{ for some }s\ne l \in \llbracket 1,\wh K\rrbracket \right)
+ \P_\mu \left(B_{j'_s} = B_{j'_l} \mbox{ for some }s\ne l \in \llbracket 1,\wh K'-1\rrbracket \right)\\
\shoveright{+ \P_\mu\left(B_{j'_s} \in \{ \wh{I}, B_{j_1},...,B_{j_{\wh K}} \}\mbox{ for some }s\in \llbracket 1,\wh K'-1 \rrbracket \right)
+\P_\mu\left(\wh B_{j_s}=\wh{I} \mbox{ for some }s\in \llbracket 1,\wh K \rrbracket \right)}\\
\shoveleft{\le \sum_{k=2}^{\infty} \sum_{s\ne l\in \llbracket 1,k \rrbracket}\!\!\! \P_\mu\left(B_{j_s} = B_{j_l} \,|\,\wh K=k\right)\P_\mu\left(\wh K=k\right)
+ \sum_{k=1}^{\infty} \sum_{s\in \llbracket 1,k \rrbracket}\!\!\! \P_\mu\left(B_{j_s} = \wh I \,|\,\wh K=k\right)\P_\mu\left(\wh K=k\right)}\\ 
\shoveright{+ \sum_{k=1}^{\infty}\sum_{k'=3}^{\infty}\sum_{s\ne l\in \llbracket 1,k'-1 \rrbracket}\!\!\!\P_\mu\left(B_{j'_s} = B_{j'_s}\, | \, \wh K=k, \, \wh K'=k'\right)
\P_\mu\left(\wh K=k , \, \wh K'=k'\right)}\\
 +\sum_{k=2}^{\infty} \sum_{k'=2}^{\infty} \sum_{s=1}^{k'-1} \P_\mu\left( B_{j'_s} \in \{ \wh{I},B_{j_1},...,B_{j_{\wh K}} \}\, | \, \wh K=k, \, \wh K'=k'\right)\P_\mu\left(\wh K=k , \, \wh K'=k'\right).
\end{multline*}
\noindent By the symmetry of uniform draws, we deduce that
\begin{multline}
\P_\mu\left(\wh{\mathscr T}(\mu)^c\right)\le \sum_{k=2}^{\infty}{k\choose 2} \P_\mu\left(B_{j_1} = B_{j_2} \,|\,\wh K=k\right)\maK_\bphi(\mu)(k)
+\sum_{k=1}^{\infty} k\P_\mu\left(B_{j_1} = \wh I \,|\,\wh K=k\right)\maK_\bphi(\mu)(k)\\
\shoveright{+  \sum_{k=1}^{\infty}\sum_{k'=3}^{\infty}{k'-1\choose 2} \P_\mu\left(B_{j'_1} = B_{j'_2}\,|\,\wh K=k,\,\wh K'=k'\right){\wh{\maK}'_\bphi(\mu,k)(k')}\maK_\bphi(\mu)(k)}\\
+ \sum_{k=2}^{\infty}\sum_{k'=2}^{\infty}(k'-1)\P_\mu\left(B_{j'_1} \in \{ \wh{I},B_{j_1},...,B_{j_{\wh K}} \}\,|\,\wh K=k,\,\wh K'=k'\right)
{\wh{\maK}'_\bphi(\mu,k)(k')}\maK_\bphi(\mu)(k)\cdot
\label{eq:approx0}
\end{multline}

\noindent But, for all $k\ge 2$, recalling that the size of $B_{j_1}$ is size-biased we obtain that 
\begin{align}
 \P_\mu\left(B_{j_1} = B_{j_2} \,|\,\wh K=k\right)
&= \sum_{y=1}^{\infty} \P_\mu\left(B_{j_1} = B_{j_2}\,|\,\wh K=k,a_{j_1}(\mu)=y\right)\P_\mu\left(a_{j_1}(\mu)=y\,|\,\wh K=k\right)
\notag\\
&=\sum_{y=1}^{\infty} {y\over \cro{\mu,\chi}}{y\mu(y)\over \cro{\mu,\chi}}={\cro{\mu,\chi^2} \over \cro{\mu,\chi}^2}\cdot\label{eq:approx1}
\end{align}
\noindent Likewise, for all $k\ge 1$ we have 
\begin{equation}
\P_\mu\left(B_{j_1} = \wh I \,|\,\wh K=k\right) ={k\over \cro{\mu,\chi}}\cdot\label{eq:approx2}
\end{equation}
and for all $k\ge 1$ and $k'\ge 3$ we obtain that
\begin{equation}
\P_\mu\left(B_{j'_1} = B_{j'_2}\,|\,\wh K=k,\,\wh K'=k'\right)={\cro{\mu,\chi^2} \over \cro{\mu,\chi}^2}\cdot\label{eq:approx3}
\end{equation}

\noindent Regarding the final term, denote
\[\wh{D} = \sum_{i=1}^{\ell}  a_{s_i}(\mu)-a_{J'}(\mu),\]
the sum of the sizes of all distinct buckets chosen at step (i) other than $\wh I'$, if any. 
\noindent Then, for all $k\ge 2,k' \ge 2$ and $d \ge 1$ we get that 
\begin{multline*}
\P_\mu\left(B_{j'_1} \in \{ \wh{I},B_{j_1},...,B_{j_{\wh K}} \}\,|\,\wh K=k,\,\wh K'=k'\right)\\
\begin{aligned}
&=\sum_{d=0}^\infty \P_\mu(B_{j'_1} \in \{ \wh{I},B_{j_1},...,B_{j_{\wh K}} \}\,|\,\wh K=k,\wh K'=k', \wh D=d)
\P_\mu\left(\wh D=d\,|\,\wh K=k,\wh K'=k'\right)\\
&\le \sum_{d=0}^\infty{k+k'+d \over \cro{\mu,\chi}}\P_\mu\left(\wh D=d\,|\,\wh K=k,\wh K'=k'\right)
={k+k'+\mathbb{E}_\mu\left[\wh D  \,|\,\wh K=k,\wh K'=k'\right] \over\cro{\mu,\chi}}\cdot
\end{aligned}
\end{multline*}
But given that $\{\wh K=k\}$ we have $\ell\le k$ a.s., and thus by symmetry we obtain that 
\[\mathbb{E}_\mu\left[\wh D  \,|\,\wh K=k,\wh K'=k'\right] \le (k-1) \mathbb{E}_\mu\left[a_{s_1}(\mu)\,|\,\wh K=k,\wh K'=k'\right].\]
As $a_{s_1}(\mu)$ follows the size-biased distribution associated to $\mu$, we conclude that 
$$\P_\mu\left(B_{j'_1} \in \{ \wh{I},B_{j_1},...,B_{j_{\wh K}} \}\,|\,\wh K=k,\,\wh K'=k'\right)
\le {k+k'+(k-1){\cro{\mu,\chi^2}\over \cro{\mu,\chi}} \over \cro{\mu,\chi}}\le  {\cro{\mu,\chi^2}\over \cro{\mu,\chi}^2}\left(2k+k'\right), 
$$ 
because ${\cro{\mu,\chi^2}\over \cro{\mu,\chi}}\ge 1$. This, combined with \eqref{eq:approx1}, \eqref{eq:approx2} and \eqref{eq:approx3} in \eqref{eq:approx0}, yields that 
\begin{multline*}
\P_\mu\left(\wh{\mathscr T}(\mu)^c\right)
\le {\cro{\mu,\chi^2} \over \cro{\mu,\chi}^2} \cro{\maK_\bphi(\mu),\chi^2}
+{\cro{\maK_\bphi(\mu),\chi^2}\over \cro{\mu,\chi}}
+ {\cro{\mu,\chi^2} \over \cro{\mu,\chi}^2} \sum_{k=1}^{\infty}  \cro{{\wh{\maK}'_\bphi(\mu,k)},\chi^2}\maK_\bphi(\mu)(k)\\
\shoveright{+  {\cro{\mu,\chi^2} \over \cro{\mu,\chi}^2}\sum_{k=2}^{\infty}\sum_{k'=2}^{\infty}(k'-1)\left(2k+k'\right)
{\wh{\maK}'_\bphi(\mu,k)(k')}\maK_\bphi(\mu)(k)}\\
\le {\cro{\mu,\chi^2} \over \cro{\mu,\chi}^2}\cro{\maK_\bphi(\mu),\chi^2+\cro{{\wh{\maK}'_\bphi(\mu,.)},\chi^2}+(\mathbb 1+2\chi)\cro{{\wh{\maK}'_\bphi(\mu,.)},\chi}}+{\cro{\maK_\bphi(\mu),\chi^2}\over \cro{\mu,\chi}},
\end{multline*}
concluding the proof. 
\end{proof}

\subsection{Generator approximation}
\label{subsec:genapprox}
We now show that the operator $\wh{\mathfrak{L}}$ defined in \eqref{eq:defLhat} is indeed an approximation of the original generator. 
\begin{proposition}
\label{prop:generatorapprox}
Let $\mu$ a counting measure of $\M^n$ {such that $\cro{\mu,\chi}>0$,} and let $f\in \maC_b$. Then, for some constant $C>0$, 
\begin{multline*}
\left|\wh{\mathfrak L}_f(\mu) - \mathfrak{L}\Pi_f(\mu) \right|\\
\leq C \lVert f \rVert
\left( {\cro{\mu,\chi^2} \over \cro{\mu,\chi}^2}\cro{\maK_\bphi(\mu),\chi^2+\cro{{\wh{\maK}'_\bphi(\mu,.)},\chi^2}+(\mathbb 1+2\chi)\cro{{\wh{\maK}'_\bphi(\mu,.)},\chi}}+{\cro{\maK_\bphi(\mu),\chi^2}\over \cro{\mu,\chi}}\right)^{1/2}\\
\times\left(2+\cro{{\maK}_\bphi(\mu),2\chi^2+\cro{{\wh{\maK}'_\bphi(\mu,.)},\chi^2}+\cro{{{\maK}'_\bphi(\mu,.)},\chi^2}}\right)^{1/2},
\end{multline*}
where the operators $\mathfrak{L}$ and $\wh{\mathfrak L}_f$ are respectively defined by \eqref{eq:defL} and \eqref{eq:defLhat}. 
\end{proposition}
\begin{proof}
Fix a point measure $\mu$, and at any iteration $j$ of the construction of Section \ref{subsec:localCM}, denote the event 
\begin{multline*}
\maT(\mu) =\Bigl\{\mbox{$I$ and $I'$ do not have any self-loops/multiple edges,}\\
\mbox{ and do not share any common neighbor, given that the degree measure is $\mu$}\Bigl\},\end{multline*} 
and observe that we also have $\maT(\mu)\subset \mathscr E(\mu)$, for the event defined by \eqref{eq:defeventE}. 
Throughout this proof, for notational simplicity we skip the dependence in $\mu$ of
$\vartheta$, $\wh{\vartheta}$, $\maT$ and $\wh{\maT}$. 

\medskip

Recall \eqref{eq:defvartheta} and \eqref{eq:defwhvartheta}. 
First observe that a uniform sampling with replacement and conditioned on not drawing twice the same element, has the same distribution as 
a uniform sampling without replacement. Therefore, as we use the same local matching criterion in both constructions, the distribution of 
$K'_{\bphi}(\mu)$ conditional on ${\mathscr T}(\mu)$ coincides with that of $\wh K'_{\bphi}$ conditional on $\wh{\mathscr T}(\mu)$. 
Second, on ${\mathscr T}(\mu)$, in \eqref{eq:defvartheta} the quantities 
$b_{\bphi,l}(\mu)$ and $b_{\bphi,l}(\mu)$, $l\in\llbracket 1,\mu(\N_+) \rrbracket$ are all 0 or 1, and moreover the indexes $l$ for which 
$b_{\bphi,l}(\mu)\ne 0$ and those for which $b'_{\bphi,l}(\mu)\ne 0$ form two disjoint subsets of $\llbracket 1,\mu(\N_+) \rrbracket$. So we get that for any point measure $\nu$, 
\begin{equation}
\label{eq:withwithout}
\P_\mu\left(\wh{\vartheta}=\nu \,|\,\wh{\maT}\right) = \P_\mu\left({\vartheta}=\nu \,|\,\maT\right).
\end{equation} 
Also, it is immediate to observe that self-loops and multiple edges occur with a larger probability if 
draws of half-edges are performed with replacement, with respect to draws without replacement. 
Thus we get that $\P_\mu\left(\maT\right) \ge \P_\mu\left(\wh{\maT}\right),$ 
 and set 
\[q_{\mu}:={\P_{\mu}\left(\maT\right) - \P_{\mu}\left(\wh{\maT}\right) \over 1 - \P_{\mu}\left(\wh{\maT}\right)} \in
(0,1).\]
Let the $\M_p$-valued r.v. $\breve{\vartheta}$ be drawn from the distribution $\P_{\mu}\left(\vartheta= . \,|\,\maT\right)$ with probability
$q_\mu$, and independently, from the distribution $\P_{\mu}\left(\vartheta= . \,|\,\maT^c\right)$ with probability
$1-q_\mu$, and let us set  
\[\widetilde{\vartheta}=\wh{\vartheta}\ind_{\wh{\maT}} + \breve{\vartheta} \ind_{\wh{\maT}^c}.\]
So defined, $\widetilde{\vartheta}$ is a $\M_p$-valued r.v. that coincides with $\wh{\vartheta}$ on $\wh{\maT}$, and that has the same distribution as $\vartheta$. To see this, observe that for all $v\in \M_p$, 
\begin{align*}
&\P_{\mu}\left(\widetilde{\vartheta}= v\right) = \P_{\mu}\left(\widetilde{\vartheta}= v \,|\, \wh{\maT}\right)
\P_{\mu}\left(\wh{\maT}\right) + \P_{\mu}\left(\widetilde{\vartheta} \,|\, \wh{\maT}^c\right)\P_{\mu}\left(\wh{\maT}^c\right)\\
&= \P_{\mu}\left(\wh{\vartheta} = v \,|\, \wh{\maT}\right)\P_{\mu}\left(\wh{\maT}\right)+ \biggl(\P_{\mu}\left(\vartheta = v \,|\, \maT\right)q_{\mu}
+\P_{\mu}\left(\vartheta = v \,|\, \maT^c\right) (1-q_{\mu})\biggl)\P_{\mu}\left(\wh{\maT}^c\right)\\
&= \P_{\mu}\left(\vartheta = v \,|\, \maT\right)\P_{\mu}\left(\wh{\maT}\right) + \P_{\mu}\left(\vartheta = v \,|\, \maT\right)\left(\P_{\mu}\left(\maT\right) - \P_{\mu}\left(\wh{\maT}\right)\right) + \P_{\mu}\left(\vartheta = v \,|\, \maT^c\right)\P_{\mu}\left(\maT^c\right)\\
&= \P_{\mu}\left(\vartheta = v\right),
\end{align*}
where we used (\ref{eq:withwithout}) in the third equality.
Therefore, as $\wh{\maT} \subset \{\wh{\vartheta} = \widetilde{\vartheta}\}$, we obtain that
\begin{align}
\left| \wh{\mathfrak L}_f(\mu) - \mathfrak{L}\Pi_f(\mu) \right|
&= \left| \E_{\mu}\left[\cro{\wh{\vartheta},f}\right]-\E_{\mu}\left[\cro{\vartheta,f}\right] \right|\nonumber\\
&= \left| \E_{\mu}\left[\cro{\wh{\vartheta},f}\right]-\E_{\mu}\left[\cro{\widetilde\vartheta,f}\right] \right|\nonumber\\
&= \left|\E_{\mu}\left[\cro{\wh{\vartheta}-\widetilde{\vartheta},f}
\ind_{\wh{\maT}^c}\right]\right|\nonumber\\
&\le \E_{\mu}\left[\cro{\wh{\vartheta}-\widetilde{\vartheta},f}^2\right]^{1/2}\P_{\mu}\left(\wh{\maT}^c
\right)^{1/2}\nonumber\\
&\le 
\left(2\E_{\mu}\left[\cro{\wh\vartheta,f}^2\right]+2\E_{\mu}\left[\cro{\widetilde{\vartheta},f}^2\right]\right)^{1/2}
\P_{\mu}\left(\wh{\maT}^c\right)^{1/2}\nonumber\\
&= 
\left(2\E_{\mu}\left[\cro{\wh\vartheta,f}^2\right]+2\E_{\mu}\left[\cro{{\vartheta},f}^2\right]\right)^{1/2}
\P_{\mu}\left(\wh{\maT}^c\right)^{1/2}.\label{eq:badday}
\end{align}
But, first, observe that the sum in \eqref{eq:defwhvartheta} has at most $\wh K-1+\wh K'-1$ non-zero terms, hence 
\begin{align*}
\E_{\mu}\left[\cro{\wh{\vartheta},f}^2\right]
\le 
\parallel f \parallel^2\E_{\mu}\left[\left(2+2(\wh K+\wh K')\right)^2\right]
&\le 12\parallel f \parallel^2\left(1+\E_{\mu}\left[\wh K^2\right]+\E_\mu\left[(\wh K')^2\right]\right)\\
&= 12 \parallel f \parallel^2\left(1+\cro{{\maK}_\bphi(\mu),\chi^2+\cro{{\wh{\maK}'_\bphi(\mu,.)},\chi^2}}\right)
\end{align*}
and by the exact same argument, 
\begin{equation*}
\E_{\mu}\left[\cro{{\vartheta},f}^2\right]
\le 12\parallel f \parallel^2\left(1+\cro{{\maK}_\bphi(\mu),\chi^2+\cro{{{\maK}'_\bphi(\mu,.)},\chi^2}}\right)
\end{equation*}
which concludes the proof using \eqref{eq:badday} and Lemma \ref{lemma:simple}. 
\end{proof}
Let us now define the following sets of measures: for any $n\in\N_+$, $\beta>0$ and $M>0$, 
\begin{equation}
\label{eq:defMbetaM}
\M^n_{\beta,M} = \left\{\mu \in \M^n\,:\,\cro{\mu,\chi^3}\le nM\mbox{ and }\cro{\mu,\mathbb 1_{\N_+}} \ge n\beta\right\}.
\end{equation}
\begin{equation}
\label{eq:defMbetaMbar}
\bar\M_{\beta,M} 
=\left\{\bar\mu \in \bar\M\,:\,\cro{\bar\mu,\chi^3} \le M\mbox{ and }\cro{\bar\mu,\mathbb 1_{\N_+}}\ge \beta\right\},
\end{equation}
and observe that for any $n$ and $\mu\in\M^n_{\beta,M}$, the measure ${1\over n}\mu$ is an element of $\bar\M_{\beta,M}$. 
To prepare the way for our generator approximations, we need to impose 
some moment assumptions on the degree distributions of the matched nodes, defined respectively by \eqref{eq:distribK} and \eqref{eq:distribhatK'}. 
\begin{definition}
\label{def:polPM}
The local matching criterion $\bphi$ is said to \emph{preserve moments} (up to order $2$) if, for all $\beta>0$ and $M>0$, there exists 
a positive constant $C'(\beta,M)$ such that, for every large enough $n$, 
	\[\sup_{\mu\in\M^n_{\beta,M}}\max\left\{\cro{\maK_{\bphi}({\mu}),\chi^{2}}\,,\,\sup_{k\in\N_+}\cro{\wh\maK'_{\bphi}({\mu},k),\chi^{2}}
	\right\}\le C'(\beta,M).\]
\end{definition}


As the choices of $K'$ and $\wh K'$ correspond to the same random draw, respectively {\em without} and {\em with} replacement, it is natural that the distributions of these r.v.'s, respectively given by \eqref{eq:distribK'} and \eqref{eq:distribhatK'}, become somewhat close, as the size $n$ of the graph (resp. the number of $n$ of buckets) goes large. However, the accuracy of this approximation in function of $n$ clearly depends on the matching criterion. Hence the following definition, which means that the $L^2$ distance between the two distributions is asymptotically of order 1. 
\begin{definition}
\label{def:polWB}
The local matching criterion $\bphi$ is said to be \emph{well behaved} if, for all $\beta>0$ and $M>0$, there exists 
a positive constant $C'(\beta,M)$ such that, for every large enough $n$, 
\[\sup_{\mu\in\M^n_{\beta,M}}\sup_{k\in\N_+}\left|\cro{\maK'_{\bphi}(\mu,k),\chi^2} - \cro{\wh\maK'_{\bphi}(\mu,k),\chi^2}\right| 
\le C''(\beta,M).\]
 \end{definition}
\noindent Observe the following immediate bound,
\begin{proposition}
\label{prop:boundKK'}
If the local matching criterion $\bphi$ is moment preserving and well-behaved, then for all $\beta>0$ and $M>0$, for every large enough $n$ we get that 
\[
\sup_{\mu\in\M^n_{\beta,M}}\left\{\esp{(K_\bphi(\mu))^2} \vee \esp{(K'_\bphi(\mu))^2}\right\} \le C'(\beta,M)+C''(\beta,M),
\]
for the constants $C'(\beta,M)$ and $C''(\beta,M)$ of Definitions \ref{def:polPM} and \ref{def:polWB}. 
\end{proposition}
\begin{proof}
Fix $n$ large enough in Definitions \ref{def:polPM} and \ref{def:polWB}, and $\mu\in\M^n_{\beta,M}$. Then by the very Definition \ref{def:polPM} we readily get that 
\[\esp{(K_\bphi(\mu))^2}=\esp{\cro{\maK_\bphi(\mu),\chi^2}}\le C'(\beta,M).\]
Further, we also have that 
\begin{align*}
\esp{(K'_\bphi(\mu))^2}&=\esp{\cro{\maK_\bphi(\mu),\cro{\maK'_\bphi(\mu,.),\chi^2}}}\\
				   &= \esp{\cro{\maK_\bphi(\mu),\cro{\wh{\maK}'_\bphi(\mu,.),\chi^2}}} + \esp{\cro{\maK_\bphi(\mu),\cro{\maK'_\bphi(\mu,.),\chi^2}-\cro{\wh{\maK}'_\bphi(\mu,.),\chi^2}}}\\
				   &\le C'(\beta,M)+C''(\beta,M),
\end{align*}
concluding the proof. 
\end{proof}

\noindent Recall the operators defined by \eqref{eq:defL} and \eqref{eq:defLhat}. We obtain the following large graph approximation,
\begin{corollary}
\label{cor:approxgen}
Let $\bphi$ be a well-behaved and moment preserving local criterion. Then,  
for all $f\in\maC_b$ and $\beta,M>0$, there exists a constant $C'''(\beta,M,f)>0$ such that for every large enough $n$, 
{\begin{equation*}
\sup_{\mu\in\M^n_{\beta,M}}\left|\wh{\mathfrak{L}}_f\left(\mu\right) - {\mathfrak{L}}\Pi_f(\mu)\right| \le {C'''(\beta,M,f)\over\sqrt{n}}\cdot
\end{equation*}}
\end{corollary}

\begin{proof}
Fix $\beta$, $M$, $f$ and $n$ satisfying Proposition \ref{prop:boundKK'}. Let $\mu\in\M^n_{\beta,M}$. 
Then Proposition \ref{prop:generatorapprox} together with the well-behavedness of $\bphi$ yield that 
\begin{multline*}
\left|\wh{\mathfrak L}_f(\mu) - {\mathfrak{L}} \Pi_f(\mu) \right|\\
\leq C \lVert f \rVert
\left( {\cro{\mu,\chi^2} \over \cro{\mu,\chi}^2}\cro{\maK_\bphi(\mu),\chi^2+\cro{{\wh{\maK}'_\bphi(\mu,.)},\chi^2}+(\mathbb 1+2\chi)\cro{{\wh{\maK}'_\bphi(\mu,.)},\chi}}+{\cro{\maK_\bphi(\mu),\chi^2}\over \cro{\mu,\chi}}\right)^{1/2}\\
\times\left(2+\cro{{\maK}_\bphi(\mu),2\chi^2+\cro{{\wh{\maK}'_\bphi(\mu,.)},\chi^2}+C''(\beta,M)\mathbb 1}\right)^{1/2}\\
\shoveleft{= {C \lVert f \rVert \over \sqrt{n}}
\left( {\cro{{1\over n}\mu,\chi^2} \over \cro{{1\over n}\mu,\chi}^2}\cro{\maK_\bphi(\mu),\chi^2+\cro{{\wh{\maK}'_\bphi(\mu,.)},\chi^2}+(\mathbb 1+2\chi)\cro{{\wh{\maK}'_\bphi(\mu,.)},\chi}}+{\cro{\maK_\bphi(\mu),\chi^2}\over \cro{{1\over n}\mu,\chi}}\right)^{1/2}}\\
\times\left(2+\cro{{\maK}_\bphi(\mu),2\chi^2+2\cro{{\wh{\maK}'_\bphi(\mu,.)},\chi^2}+C''(\beta,M)\mathbb 1}\right)^{1/2}.
\end{multline*}
{But as $\bphi$ is moment preserving, we obtain that 
\begin{multline*}
\left|\wh{\mathfrak L}_f(\mu) - {\mathfrak{L}} \Pi_f(\mu) \right|\\
\shoveleft{\le {C \lVert f \rVert \over \sqrt{n}}
\left( {M \over \beta^2}\cro{\maK_\bphi(\mu),\chi^2+C'(\beta,M)(2\mathbb 1+2\chi)}+{M\over \beta}\right)^{1/2}}\\
\shoveright{\times\left(2+\cro{{\maK}_\bphi(\mu),2\chi^2+2C'(\beta,M)\mathbb 1+C''(\beta,M)\mathbb 1}\right)^{1/2}}\\
\shoveleft{\le {C \lVert f \rVert \over \sqrt{n}}
\left({M \over \beta^2}\left(C'(\beta,M)+4C'(\beta,M)^2\right)+{M\over \beta}\right)^{1/2}}\\
\shoveright{\times\left(2+2C'(\beta,M)+2C'(\beta,M)^2+C''(\beta,M)C'(\beta,M)
\right)^{1/2}}\\
=:{C'''(\beta,M,f)\over \sqrt{n}},
\end{multline*}
which concludes the proof. 
}
\end{proof}

\section{Hydrodynamic limits}
\label{sec:Main}
In this Section, we introduce our main result. We extend the measure-valued DTMC introduced in Section \ref{sec:Markov} for the construction of Section \ref{subsec:localCM}, to a continuous-time measure-valued process, which we scale in turn, so as to obtain a large graph approximation by a deterministic, continuous measure-valued function. 

Throughout this section, we fix the size $n$ of the multi-graph produced by the CM, and append a superscript $^n$ to all the corresponding variables. 
We first need to make specific assumptions on the initial conditions of the process at hand. 
\begin{hypo}
\label{hypo:Ho}
For some $\nu \in \M_F$ such that
\begin{align*}
\cro{\nu,\mathbb 1_{\N_+}}&>0,\\
\cro{\nu,\chi^{3.5+\varepsilon}}&<\infty,\, \mbox{ for some }\varepsilon>0. 
\end{align*}
 the sequence of initial measures $\suite{{1\over n}\mu^n_0}$  
satisfies 
\[
\cro{{1\over n}\mu^n_0,f}\xrightarrow{(n,\P)} \cro{\nu,f},\,{\mbox{ for all $f\in\maC_b$.}} 
 \]				
\end{hypo}

\begin{remark}\rm
Having a finite moment of order $3.5 +\varepsilon$ for the limiting initial distribution is a technical assumption that will be useful to prove uniqueness of the 
solution of a given system of ODE's, see Section \ref{sec:greedy} below. 
However, the interest of a bound at least for the third moment of $\nu$, is not limited to this technical aspect. The average number of neighbors at distance $2$ of a typical node in the configuration model CM$(\nu)$ is shown to be close to $\cro{\nu,\chi^2} - \cro{\nu,\chi}$ (see \cite{newman2018networks,10.5555/3086985,angel2017limit}), and each of those second neighbors has a degree that follows again the size-biased degree distribution associated to $\nu$. Satisfying Assumption \ref{hypo:Ho} guarantees a control of the number of such neighbors. Thus one can define dynamics on the considered graph, that depend at most on the second neighbors of a given node, such as the local matching algorithms that are introduced here. This assumption also allows to control the number of loops and multiple edges, since the number of such edges then converges to Poisson r.v.'s of order 1 (see e.g. Proposition 7.9 in \cite{10.5555/3086985}, or \cite{angel2017limit}), implying that their influence on the {proportion of matched nodes} vanishes as $n$ goes to infinity. 
\end{remark}
\medskip

To show the convergence of a scaled version of the DTMC $\suiten{\mu^n_j}$, we first extend this random sequence onto a RCLL process of the Skorokhod space 
$\mathbb{D}([0,1],\maM_F)$, as follows: For any $0\leq t\leq 1$, denote 
\begin{equation*}
\label{mubar}
\bar{\mu}_t^n = \frac{1}{n} \mu^n_{\lfloor nt \rfloor}. 
\end{equation*} 
One immediate downside of that extension is that for any $n$, the piecewise constant process $$\bar{\mu}^n:=\procun{\bar\mu^n_t}$$ is clearly not Markov 
on $\mathbb{D}([0,1],\maM_F)$. 
On the other hand, from \eqref{eq:ratioCMmu}, the matching coverage of the construction of Section \ref{subsec:localCM} conveniently becomes 
\be
\label{eq:ratioCMmubar}
{\mathbf {M}}^n_\Phi(\mu^n_0) =  1-{\bar{\mu}^n_1(0)}-{\maB_n\over n}\cdot
\ee

\begin{definition}
\label{def:polCont}
A local matching criterion $\bphi$ 
is said to be \emph{continuous} on $\bar{\M}$ if,
for any $\bar\mu\in\bar\M$, for any sequence $\suite{\bar\mu^n}$ of $\bar\M$ such that $\bar\mu^n\in \bar\M^n$ for all $n$, and 
such that $\bar\mu^n \xRightarrow{n} \bar\mu$ for some measure $\bar\mu\in  \bar\M$ satisfying $\cro{\bar\mu,\mathbb 1_{\N_+}}>0$, 
there exist measures $\bar{\maK}_{\bphi}(\bar{\mu})$ 
and $\bar{\maK}'_{\bphi}(\bar{\mu},k),\,k\in\N_+$, of $\M_F(\N_+)$, and such that, in the weak topology, 
	{\begin{align}
	\maK_{\bphi}(n\bar\mu^n) &\xRightarrow{n} \bar{\maK}_{\bphi}(\bar{\mu});\label{eq:defKbar}\\
	\wh\maK'_{\bphi}(n\bar\mu^n,k) &\xRightarrow{n} \bar{\maK}'_{\bphi}(\bar{\mu},k),\quad \mbox{for all }k\in\N_+.\label{eq:defK'bar}
	\end{align}}
 \end{definition}

We are now in a position to state our main convergence results: Under suitable assumptions, the sequence of measure-valued stochastic processes 
$\suite{\procun{\bar\mu^n_t}}$ associated to the local algorithm under consideration converges to a deterministic measure-valued function. And by extension, the corresponding matching coverages converge to a deterministic value. 

\begin{theorem}
\label{thm:main}
Let $\bphi$ be a local matching criterion that is moment preserving, well-behaved and continuous. 
Define the linear operator $\bar{\mathfrak{L}}$ as follows: {For all $f\in\maC_b$, for all $\bar\mu\in \bar{\M}$}, 
{
\begin{multline}
\label{eq:defLbar}
\bar{\mathfrak{L}}_f(\bar\mu) \\= 
- \left\{\cro{\bar\maK_{\bphi}(\bar\mu),f+\cro{\bar{\maK}'_{\bphi}(\bar\mu,.),f}}
+{\cro{\bar\mu,\chi\nabla f}\over \cro{\bar\mu,\chi}}
\cro{\bar\maK_{\bphi}(\bar\mu),\chi-\mathbb 1+\cro{\bar\maK'_{\bphi}(\bar\mu,.),\chi-\mathbb 1}}\right\}\mathbb 1_{\{\cro{\bar\mu,\mathbb 1_{\N_+}}>0\}},
\end{multline}
for $\bar\maK_{\bphi}(\bar\mu)$ and $\bar{\maK}'_{\bphi}(\bar\mu,.)$ respectively defined by \eqref{eq:defKbar} and \eqref{eq:defK'bar}.} 
Suppose that the sequence of initial conditions $\suite{\bar{\mu}^n_0}$ satisfies Assumption \ref{hypo:Ho} for some measure $\nu$, and 
that the system of integral equations
\be
\label{eq:limhydro}
\cro{\eta_t,f} = \cro{\nu,f} + \int_0^t \bar{\mathfrak{L}}_f(\eta_s)\d s,\quad f \in \maC_b(\R),\,t\in[0,1]
\ee
admits at most one solution in $\mathbb D\left([0,1],\bar\M\right).$ 
 Then, the system \eqref{eq:limhydro} admits a unique solution $\procun{\bar{\mu}^*_t}$ in $\mathbb C\left([0,1],\bar\M\right),$ and the following convergence holds: \begin{equation}
\label{MainConv}
\sup_{t \leq 1 } |\cro{\bar{\mu}^n_t,f} -  \cro{\bar{\mu}^*_t,f}| \xrightarrow{(n,\P)} 0.
\end{equation}
\end{theorem} 

\begin{corollary}[Convergence of the matching coverage]
\label{cor:maincov}
Under the above assumptions, for any $\nu \in {\bar\M}$ we get
\begin{equation}
\label{MainConvCover}
{\mathbf {M}}^n_\bphi(\nu) \xrightarrow{(n,\P)} {\mathbf {M}}_\Phi(\nu) := 1 - \bar{\mu}^*_1 (0).
\end{equation}
\end{corollary}

\begin{proof}
First observe that for all $n$, the number $\maB^n_n$ of blocked nodes at the end of the procedure in the $n$-th system is upper bounded by the number, say $\mathcal S^n_n$, of self-loops at the end of the procedure, in the resulting graph $\maG^n$. But it follows from Proposition 7.9 in \cite{10.5555/3086985} that the sequence $\suite{\mathcal S^n_n}$ converges weakly to a Poisson r.v. of order $1$. Therefore, 
\[{\maB^n_n \over n}\le {\mathcal S^n_n \over n} \xrightarrow{(n,\P)} 0.\]
The result then follows by taking the limit in probability in Equation (\ref{eq:ratioCMmubar}), applying Theorem \ref{thm:main} and the Continuous Mapping 
Theorem (see e.g. \cite{Bill}).   
\end{proof}
This result is to be related to the heuristic hydrodynamic approximations obtained in (21) and (22) of \cite{aoudi:hal-03294781}, respectively in the case where 
$\bphi=\textsc{greedy}$ and $\bphi=\textsc{uni-min}$, in the related (but not identical case) where the resulting graph is bipartite, and produced by the 
bipartite CM, as in \cite{OlvChen}. Theorem \ref{thm:main} gives a mathematical justification of the large-graph convergence of the suitably normalized process under consideration to the hydrodynamic limit, for general (instead of bipartite) graphs. 

As in Section 7 of \cite{aoudi:hal-03294781}, Corollary \ref{cor:maincov} can then be used to provide a large graph approximation of the matching coverage for 
various degree distributions, by solving numerically the corresponding differential system.

\section{Proof of Theorem \ref{thm:main}}
\label{sec:proof}
Throughout the proof, we fix a local matching criterion $\bphi$ satisfying the assumptions of Theorem \ref{thm:main}. 
Our strategy of proof is as follows: we first extend the sequence of (non-Markov) processes $\suite{\bar\mu^n}$ to a sequence of CTMC's $\suite{\tilde\mu^n}$,  that approximate their dynamics. 
Then, we follow a standard compactness-uniqueness approach: 
First, we show that the law of $\suite{\tilde\mu^n}$ is tight, and thus relatively compact, in the Skorokhod topology. 
{Then we show that $\suite{\tilde\mu^n}$ is asymtotically driven by a deterministic equation controlled by its generator. Finally, we show that the large graph generator of $\suite{\tilde\mu^n}$ is approximated by the operator $\bar{\mathfrak{L}}$ defined in Theorem \ref{thm:main}. In particular, any subsequential limit of $\suite{\tilde\mu^n}$ is driven by $\bar{\mathfrak{L}}$. And since $\bar{\mathfrak{L}}$ characterizes a unique process $\bar\mu$, all the subsequential limits are in fact equal to $\bar\mu$.}

First, observe that, in view of Assumption \ref{hypo:Ho}, there exist $M>0$ and an arbitrary small $\beta>0$, that are such that 
\begin{align}
\cro{\nu,\mathbb 1_{\N_+}}&>\beta,\label{eq:condbeta}\\
\cro{\nu,\chi^3}&<M. \notag
\end{align}
Both are fixed until the end of the section.  

\subsection{Continuous time Markov chain}
Fix again $n\in\N_+$, and a matching criterion $\bphi$. 
To extend the DTMC $\mu^n$ onto a CTMC, we let $\suitei{\theta_i}$ be a sequence of independent identically distributed exponential random variables with parameter $1$, independent of  $\bar{\mu}^n$. Define the sequence $\suitei{\tau_i}$ of random times, by 
\[\tau_0 = 0\quad\mbox{ and }\quad \tau_i:= \sum_{j=1}^{i} {\theta_{j}\over n},\quad i\in \N.\]
We then define the process $\tilde{\mu}^n$ of $\mathbb{D}(\R_+,\M_F)$, by 
\begin{equation}
\label{eq:defmutilde}
\tilde{\mu}^n_t = 
\bar{\mu}^n_{l/n} =  \frac{1}{n}\mu^n_l\mathbb 1_{[0,1]}(t),\quad \text{for $\tau_l \leq t < \tau_{l+1}$},\,l \in \N.
\ee
Throughout, we append a ``tilde'' to all variables related to the process 
$\tilde{\mu}^n$. Observe that, as $\procn{\mu^n_j}$ is a DTMC, the process $\tilde{\mu}^n$ is itself clearly a CTMC. 
Its generator is specified hereafter, 
\begin{proposition}
\label{Mapprox}
For all $n\in\N_+$, $\tilde{\mu}^n$ is a Feller Markov process with generator $\tilde{\mathfrak{L}}^n$ defined by 	
        \begin{equation}
	\label{eq:fight0}
	\tilde{\mathfrak{L}}^n F(\bar{\mu}) := n\E_{\bar{\mu}} \left(F\left(\bar{\mu} + \frac{1}{n}\vartheta(n\bar{\mu})\right)  - F(\bar\mu)\right),\quad 
	F\in \maC_b(\bar\M^n),\, \bar{\mu}\in \bar{\M}^n, 
	\end{equation}
	for $\vartheta$ defined by \eqref{eq:defvartheta}. In particular, for all $f\in \maC_b(\N)$ we get that 
	\begin{equation}
	\label{eq:fight1}
	\tilde{\mathfrak{L}}^n \Pi_f(\bar{\mu}) = \E_{\bar{\mu}} \cro{\vartheta(n\bar{\mu}),f}={\mathfrak{L}}^n \Pi_f(n\bar{\mu}),\quad 
	\bar{\mu}\in \bar{\M}^n. 
	\end{equation}
\end{proposition}

\begin{proof} {Fix $n\in\N_+$.} 
Since $\tilde{\mu}^n$ is a pure jump markov process with bounded intensities, it is Feller continuous (see Theorem 12.18 of \cite{Kallenberg2002}, or \cite{ethier2009markov} p. 163). 
Thus it admits a strongly defined generator: For all $F\in \maC_b(\bar\M^n)$ and $\bar\mu\in \bar\M^n$, for all $h>0$, 
\begin{align*}
\E_{\bar\mu}(F(\tilde{\mu}^n_h) - F(\tilde{\mu}^n_0)) 
= \,& \P(\tau_1<h,\, \tau_2>h) \E_{\tilde{\mu}^n_0}(F(\tilde{\mu}^n_{\tau_1}) - F(\tilde{\mu}^n_0) \mid\tau_1<h, \, \tau_2>h) \\
&\, + \P(\tau_1<h, \tau_2<h)\E_{\tilde{\mu}^n_0}(F(\tilde{\mu}^n_h) - F(\tilde{\mu}^n_0)\mid\tau_1<h, \tau_2<h).
\end{align*}
Notice that 
$$\P(\tau_1<h, \, \tau_2>h) = \P(\theta_1/n<h,\theta_1/n + \theta_2/n >h ) = hn + o(nh)$$
and likewise 
$$\P(\tau_1<h, \tau_2<h) = o(nh),$$
implying that 
\begin{align*}
\E_{\bar\mu}\left(F(\tilde{\mu}^n_h) - F(\tilde{\mu}^n_0)\right) &= nh\E_{\bar\mu}\left(F(\tilde{\mu}^n_{\tau_1}) - F(\tilde{\mu}^n_0) \mid\tau_1<h, \tau_2>h\right) + o(nh) \\
&= nh\E_{\bar{\mu}}\left(F(\bar{\mu}^n_0 + (1/n)\vartheta(n\bar{\mu}^n_0)) - F(\bar{\mu}^n_0)\right) +o(nh),
\end{align*}
proving {that} the generator of the Feller process $\tilde \mu^n$ reads as 
\[
\lim_{h\rightarrow 0} \frac{1}{h}  \E_{\bar{\mu}}(F(\tilde{\mu}^n_h) - F(\tilde{\mu}^n_0))=\tilde{\mathfrak{L}}^n F(\bar\mu),\,F\in \maC_b(\bar\M^n),\,\bar\mu\in\bar\M_n. 
\]
\end{proof}


{We have just proven that for all $n$, $\tilde{\mu}^n$ is a Feller process having} generator $\tilde{\mathfrak{L}}^n$. 
Posing the associated martingale problem yields the following result: 
\begin{proposition}
{For any $f\in\maC_b$ and $n \geq 1$, the process defined for all $t\in[0,1]$, by}
\be
\label{Martingale}
\maM_t^{f,n} = \cro{\tilde{\mu}^n_t,f} - \cro{\tilde{\mu}^n_0,f} - \int_0^{t}\tilde{\mathfrak{L}}^n \Pi_f(\tilde{\mu}^n_s) \d s
\ee
is a square integrable martingale. Moreover, its predictable quadratic variation is given by 
\be
\label{varMartingale}
\llangle\maM^{f,n}\rrangle_t = \frac{1}{n}\int_0^t \E_{\tilde{\mu}^n_s}\cro{\vartheta({n}\tilde{\mu}^n_s),f}^2\d s,\quad t\in[0,1],
\ee
for $\vartheta$ defined by \eqref{eq:defvartheta}. 
\end{proposition}

\begin{proof}
Fix $n$ and $f$. Then it is a classical fact that the process $\maM^{f,n}$ is a local martingale 
(see e.g. Section 5.1.2 in \cite{10.1007/BFb0084190}; Chapter 4 of \cite{ethier2009markov}). Applying Theorem 7.15 in \cite{MoyStoModel}, its quadratic variation is given by the process   
\begin{align}
t\longmapsto \llangle \maM^{f,n} \rrangle_t &= \int_0^t\left(\tilde{\mathfrak{L}}^n \Pi_f^2({\tilde{\mu}^n_s}) - 2\cro{{\tilde{\mu}^n_s},f} \tilde{\mathfrak{L}}^n \Pi_f({\tilde{\mu}^n_s})\right) \d s. \label{eq:fight2}
\end{align}
But for all $\bar\mu\in\bar\M^n$, we get 
\begin{align*}
\tilde{\mathfrak{L}}^n\Pi_f^2(\bar\mu) &= n\E_{\bar\mu} \left( \Pi_f^2\left(\bar\mu+ \frac{1}{n}\vartheta(n\bar\mu)\right) - \Pi_f^2(\bar\mu) \right)\\
&= n\E_{\bar\mu} \left(\cro{\bar\mu,f}^2 + \frac{2}{n}\cro{\vartheta(n\bar\mu),f}\cro{\bar\mu,f} + \frac{1}{n^2}\cro{\vartheta(n\bar\mu),f}^2 - \cro{\bar\mu,f}^2\right)\\
&= \E_{\bar\mu} \left(2\cro{\vartheta({n}\bar\mu),f}\cro{\bar\mu,f} + \frac{1}{n}\cro{\vartheta({n}\bar\mu),f}^2\right) 
\end{align*}
{which, together with \eqref{eq:fight1} and \eqref{eq:fight2}, implies \eqref{varMartingale}.}
To conclude, recalling \eqref{eq:defvartheta} we get that a.s. for all $t\in[0,1]$, 
\begin{multline}
\label{eq:boundcrochet}
\llangle \maM^{f,n} \rrangle_t = \frac{1}{n}{\int_0^t} \E_{\tilde{\mu}^n_s}( \cro{\vartheta({n}\tilde{\mu}^n_s),f}^2) {\d s}\\
\shoveleft{= \frac{1}{n} {\int_0^t}\E_{\tilde{\mu}^n_s}\Biggl\{\mathbb 1_{\{\tilde{\mu}^n_s(\N_+)>0\}}\Biggl(f(K_{{\mathbf \Phi}}({n}\tilde{\mu}^n_s))}\\
 \left.\left.+ \left(f(K'_{{\mathbf \Phi}}({n}\tilde{\mu}^n_s)) + \sum_{l=1}^{{n}\tilde{\mu}^n_s\left(\N_+\right)} 
\(f\left(a_{l}\left({n}\tilde{\mu}^n_s\right)\right) - f\left(a_{l}\left({n}\tilde{\mu}^n_s\right) - b_{\bphi,l}\left({n}\tilde{\mu}^n_s\right)-b'_{\bphi,l}\left({n}\tilde{\mu}^n_s\right)\right) \)\right)\mathbb 1_{\mathscr E({n}\tilde{\mu}^n_s)} \right)^2\right\}{\d s}. 
\end{multline}
The above is clearly bounded by $ {1\over n}\lVert f \rVert^2(2 + 2n)^2$, 
implying that $\maM^{f,n}$ is indeed a square integrable martingale.
\end{proof}

\subsection{Tightness}
\label{tightv}
Define the sequence of stopping times
\begin{equation}
\label{eq:deftaun}
\tau^n_\beta=\inf\left\{t\ge 0\,:\,\cro{\bar\mu^n_t,\chi} < \beta\right\},\,n\in\N_+,
\end{equation}
and for any càdlàg process $X$ over $\mathbb D\left([0,1],E\right)$, for $E$ a Polish space, denote for all $n\in\N_+$ and $\beta>0$, 
\[t\longmapsto X^{\tau^n_\beta}_t=X_{t\wedge \tau^n_\beta},\,t\in[0,1],\]
the process $X$ stopped at time ${\tau^n_\beta}$.

We first prove the relative compactness of the laws of the processes $\suite{\tilde\mu^{n,\tau^n_\beta}}$ in $\mathbb{D}([0,1], (\maM,\tau_v))$, 
where $\tau_v$ denotes the {vague} topology. Then we extend this result to the weak topology. 
\begin{lemma}
\label{lemma:tightv}
Under the conditions of Theorem \ref{thm:main}, the sequence $\suite{\tilde\mu^{n,\tau^n_\beta}}$ is $\mathbb C$-tight in the space  
$\mathbb{D}([0,1], (\maM,\tau_v))$, that is, it is tight in $\mathbb{D}([0,1], (\maM,\tau_v))$ and any subsequential limit is an 
element of $\mathbb{C}([0,1], (\maM,\tau_v))$. 
\end{lemma}
\begin{proof}
Méléard and Roelly (\cite{Mlard1993SurLC}) show that for $\suite{\tilde\mu^{n,\tau^n_\beta}}$ to be tight in $\mathbb{D}([0,1], (\maM,\tau_v))$, 
it is enough that, for any $f\in \maC_K$, the sequence of processes $\suite{\cro{\tilde\mu^{n,\tau^n_\beta}_.,f}}$ be tight in $\mathbb{D}([0,1], \R)$. 
To show the latter, we fix $n$ and $f\in \maC_K$, and recall the semi-martingale decomposition (\ref{Martingale}): 
For $t\in [0,1]$, $\cro{\tilde\mu^{n,\tau^n_\beta}_t,f}$ can be rewritten as 
	\begin{equation}
	\label{eq:devsemimart}
	\cro{\tilde\mu^{n,\tau^n_\beta}_t,f} = \maM_{t} ^{f,n,\tau^n_\beta} + V_{t}^{f,n,\tau^n_\beta}
	=\maM_{t\wedge \tau^n_\beta} ^{f,n} + V_{t\wedge \tau^n_\beta}^{f,n},
	\end{equation}
	where the finite variation part is given by
	\[t\longmapsto V^{f,n}_t=\cro{\tilde{\mu}^n_0,f} + \int_0^{t}\tilde{\mathfrak{L}}^n \Pi_f(\tilde{\mu}^n_s) \d s,\quad t\in[0,1].\]
	Then, we use Roelly's tightness criterion \cite{doi:10.1080/17442508608833382}, which states that 
the following conditions are sufficient for the tightness of $\suite{\cro{\tilde\mu^{n,\tau^n_\beta}_.,f}}$ in $\mathbb{D}([0,1], \R)$: Recalling the notation 
in \eqref{eq:devsemimart}, 
\begin{itemize}
	\item For all $t$, $\suite{\cro{\tilde\mu^{n,\tau^n_\beta}_t,f}}$ is tight;
	\item For all $\epsilon>0$ and $\eta >0$, there exists $\delta>0$ such that if $\suite{(S_n,T_n)}$ is a sequence of stopping times with 
	$0\leq S_n \leq T_n \leq (S_n + \delta) \wedge 1$, there exists $n_0$ such that 
	\begin{align}
\sup_{n\geq n_0} &\pr{\left|\llangle\maM^{f,n,\tau^n_\beta}\rrangle_{T_n} - \llangle\maM^{f,n,\tau^n_\beta}\rrangle_{S_n}\right| > \epsilon}\leq \eta;\label{eq:tightnew1}\\
		\sup_{n\geq n_0} &\pr{\left|V^{f,n,\tau^n_\beta}_{T_n} - V^{f,n,\tau^n_\beta}_{S_n}\right| > \epsilon}\leq \eta.\label{eq:tightnew2}
		\end{align}
\end{itemize}
To show the first assertion, for all $t\geq 0$ it suffices to write that for all $n$,
$$\left|\cro{\tilde\mu^{n,\tau^n_\beta}_t,f}\right| \leq \lVert f \rVert \left|\cro{\tilde\mu^{n,\tau^n_\beta}_t,\mathbb 1}\right| \leq \lVert f \rVert.$$ 
Hence the sequence $\suite{\cro{\tilde\mu^{n,\tau^n_\beta}_t,f}}$ is tight since it is bounded.  
Now, regarding the second assertion, fix $\epsilon>0$ and $\eta>0$. 
First, fix $\delta > 0$, $n\ge 1$, and two stopping times $S_n$ and $T_n$ such that $S_n \leq T_n \leq S_n + \delta$.  
	On the one hand, using the Markov inequality we have 
	\begin{multline}\label{eq:boundcrochet0}
	\pr{ \left|\llangle\maM^{f,n,\tau^n_\beta}\rrangle_{T_n} - \llangle\maM^{f,n,\tau^n_\beta}\rrangle_{S_n} \right| > \epsilon}\\
	\begin{aligned}
	&\le \pr{\left\{\left|\llangle\maM^{f,n,\tau^n_\beta}\rrangle_{T_n} - \llangle\maM^{f,n,\tau^n_\beta}\rrangle_{S_n} \right| > \epsilon\right\}
	\cap \left\{\mu^n_0\in \M^n_{\beta,M}\right\}}+\pr{\mu^n_0\not\in \M^n_{\beta,M}} \\
	&= \pr{\left\{\left|\llangle\maM^{f,n}\rrangle_{T_n\wedge \tau^n_\beta} - \llangle\maM^{f,n}\rrangle_{S_n\wedge\tau^n_\beta} \right| > \epsilon
	\right\}\cap \left\{\mu^n_0\in \M^n_{\beta,M}\right\}}
	+\pr{\mu^n_0\not\in \M^n_{\beta,M}}\\
	&\leq {1\over \epsilon} \esp{\left|\llangle\maM^{f,n}\rrangle_{T_n\wedge \tau^n_\beta} - \llangle\maM^{f,n}\rrangle_{S_n\wedge\tau^n_\beta}\right|
	\mathbb 1_{\{\mu^n_0\in \M^n_{\beta,M}\}}}+\pr{\mu^n_0\not\in \M^n_{\beta,M}}.
	\end{aligned}\end{multline}
But let us first observe that  
\begin{multline}
\label{eq:convfinal2}
\pr{\mu^n_0\not\in \M^n_{\beta,M}} 
\le \pr{\cro{\bar\mu^n_0,\mathbb 1_{\N_+}} < \beta} + \pr{\cro{\bar\mu^n_0,\chi^3} > M}\\
\le \pr{\cro{\bar\mu^n_0,\mathbb 1_{\N_+}} < \cro{\nu,\mathbb 1_{\N_+}}-{ \cro{\nu,\mathbb 1_{\N_+}} - \beta \over 2} } 
+ \pr{\cro{\bar\mu^n_0,\chi^3} >  \cro{\nu,\chi^3}+{M-\cro{\nu,\chi^3}  \over 2} }
\xlongrightarrow{n} 0,
\end{multline}
in view of Assumption \ref{hypo:Ho}. Second, recalling \eqref{eq:boundcrochet} we have that 
\begin{multline}\label{eq:boundcrochetbis}
\left|\llangle\maM^{f,n}\rrangle_{T_n\wedge \tau^n_\beta} - \llangle\maM^{f,n}\rrangle_{S_n\wedge\tau^n_\beta}\right|
=\frac{1}{n}\left|\int^{T_n\wedge \tau^n_\beta}_{S_n\wedge \tau^n_\beta} \E_{\tilde{\mu}^n_s}\cro{\vartheta(n\tilde{\mu}^n_s),f}^2\d s\right| \\
\shoveleft{\le \frac{2}{n}\int^{T_n\wedge \tau^n_\beta}_{S_n\wedge \tau^n_\beta} \E_{\tilde{\mu}^n_s}\left[\mathbb 1_{\{\tilde{\mu}^n_s(\N_+)>0\}}f(K_{{\mathbf \Phi}}({n}\tilde{\mu}^n_s))^2\right]\d s}\\
 \shoveleft{+{2\over n}\int^{T_n\wedge \tau^n_\beta}_{S_n\wedge \tau^n_\beta}\E_{\tilde{\mu}^n_s}\Biggl[\mathbb 1_{\{\tilde{\mu}^n_s(\N_+)>0\}\cap\mathscr E({n}\tilde{\mu}^n_s)}
 \Biggl(f(K'_{{\mathbf \Phi}}({n}\tilde{\mu}^n_s))}\\
 \shoveright{ \left.+ \sum_{l=1}^{{n}\tilde{\mu}^n_s\left(\N_+\right)} 
\(f\left(a_{l}\left({n}\tilde{\mu}^n_s\right)\right) - f\left(a_{l}\left({n}\tilde{\mu}^n_s\right) - b_{\bphi,l}\left({n}\tilde{\mu}^n_s\right)-b'_{\bphi,l}\left({n}\tilde{\mu}^n_s\right)\right) \)\Biggl)^2 \right]\d s}\\
\shoveleft{\le \frac{2}{n}\int^{T_n\wedge \tau^n_\beta}_{S_n\wedge \tau^n_\beta} \E_{\tilde{\mu}^n_s}\left[\mathbb 1_{\{\tilde{\mu}^n_s(\N_+)>0\}}f(K_{{\mathbf \Phi}}({n}\tilde{\mu}^n_s))^2\right]\d s}\\
 \shoveleft{+{2\over n}\int^{T_n\wedge \tau^n_\beta}_{S_n\wedge \tau^n_\beta} \E_{\tilde{\mu}^n_s}\Biggl[\mathbb 1_{\{\tilde{\mu}^n_s(\N_+)>0\}\cap\mathscr E({n}\tilde{\mu}^n_s)}
 \left(K_{{\mathbf \Phi}}({n}\tilde{\mu}^n_s)+K'_{{\mathbf \Phi}}({n}\tilde{\mu}^n_s)-1\right)}\\
 \shoveright{\times \left\{f(K'_{{\mathbf \Phi}}({n}\tilde{\mu}^n_s))^2+
 \sum_{l=1}^{{n}\tilde{\mu}^n_s\left(\N_+\right)} 
\(f\left(a_{l}\left({n}\tilde{\mu}^n_s\right)\right) - f\left(a_{l}\left({n}\tilde{\mu}^n_s\right) - b_{\bphi,l}\left({n}\tilde{\mu}^n_s\right)-b'_{\bphi,l}\left({n}\tilde{\mu}^n_s\right)\right) \)^2\right\} \Biggl]\d s}\\
\shoveleft{\le \frac{2}{n}\int^{T_n\wedge \tau^n_\beta}_{S_n\wedge \tau^n_\beta} \parallel f \parallel^2\d s}\\
 \shoveright{+{2\over n}\int^{T_n\wedge \tau^n_\beta}_{S_n\wedge \tau^n_\beta} \E_{\tilde{\mu}^n_s}\Biggl[
 \left(K_{{\mathbf \Phi}}({n}\tilde{\mu}^n_s)+K'_{{\mathbf \Phi}}({n}\tilde{\mu}^n_s)-1\right)\Bigl\{\parallel f \parallel^2\left(1+2\left(K_{{\mathbf \Phi}}({n}\tilde{\mu}^n_s)+K'_{{\mathbf \Phi}}({n}\tilde{\mu}^n_s)-2\right)\right)
 \Bigl\} \Biggl]\d s}\\
\le \frac{2\parallel f \parallel^2}{n}\int^{T_n\wedge \tau^n_\beta}_{S_n\wedge \tau^n_\beta} \left(1+\E_{\tilde{\mu}^n_s}\Biggl[
 \left(K_{{\mathbf \Phi}}({n}\tilde{\mu}^n_s)+K'_{{\mathbf \Phi}}({n}\tilde{\mu}^n_s)\right)\left(2K_{{\mathbf \Phi}}({n}\tilde{\mu}^n_s)+2K'_{{\mathbf \Phi}}({n}\tilde{\mu}^n_s)\right)\Biggl]\right)\d s, 
	\end{multline}
	\noindent {where, in the first two inequalities we use again the fact that for all $s$, $b_{\bphi,l}\left({n}\tilde{\mu}^n_s\right)+b'_{\bphi,l}\left({n}\tilde{\mu}^n_s\right)$ is 
	non-zero for at most $K_{{\mathbf \Phi}}({n}\tilde{\mu}^n_s)+K'_{{\mathbf \Phi}}({n}\tilde{\mu}^n_s)-2$ indexes $l$. 
	But, as the process $s\mapsto \cro{\tilde{\mu}^n_s,\chi^3}$ is a.s. non-decreasing, on the event $\{\mu^n_0\in \M^n_{\beta,M}\}$ 
	the process $\left(\tilde{\mu}^n_s,\,s\in[0,\tau^n_\beta]\right)$ takes values in $\bar\M^n_{\beta,M}$. Therefore, as $\bphi$ is moment preserving and 
	well-behaved, applying Proposition \ref{prop:boundKK'} and Cauchy-Schwarz inequality to the right-hand side of \eqref{eq:boundcrochetbis}, 
	we obtain that 
	\begin{equation}
	\esp{\left|\llangle\maM^{f,n}\rrangle_{T_n\wedge \tau^n_\beta} - \llangle\maM^{f,n}\rrangle_{S_n\wedge\tau^n_\beta}\right|
	\mathbb 1_{\{\mu^n_0\in \M^n_{\beta,M}\}}}
\le \frac{2\delta\parallel f \parallel^2}{n}\left(1+C\right),\label{eq:boundcrochetter}
	\end{equation}
         where $C$ is a constant that depends only on $\beta$ and $M$. 
 Injecting this together with \eqref{eq:convfinal2} in \eqref{eq:boundcrochet0}, 
shows that \eqref{eq:tightnew1} holds for a small enough $\delta$ and a large enough $n_0$.} 

Reasoning similarly, fixing again $\delta > 0$, $n\ge 1$, and two stopping times $S_n$ and $T_n$ such that $S_n \leq T_n \leq S_n + \delta$, 
we get  
\begin{multline}
\left|V^{f,n}_{T_n\wedge \tau^n_\beta} - V^{f,n}_{S_n\wedge ,\tau^n_\beta}\right| 
= \left|\int^{T_n\wedge \tau^n_\beta}_{S_n\wedge \tau^n_\beta} \E_{\tilde{\mu}^n_s}\cro{\vartheta(n\tilde{\mu}^n_s),f}\d s\right| \\
\shoveleft{\le \int^{T_n\wedge \tau^n_\beta}_{S_n\wedge \tau^n_\beta} \E_{\tilde{\mu}^n_s}\left[\mathbb 1_{\{\tilde{\mu}^n_s(\N_+)>0\}}\left|f(K_{{\mathbf \Phi}}({n}\tilde{\mu}^n_s))\right|\right]\d s}\\
 \shoveleft{+\int^{T_n\wedge \tau^n_\beta}_{S_n\wedge \tau^n_\beta} \E_{\tilde{\mu}^n_s}\Biggl[\mathbb 1_{\{\tilde{\mu}^n_s(\N_+)>0\}\cap\mathscr E({n}\tilde{\mu}^n_s)}
 \Biggl( \left|f(K'_{{\mathbf \Phi}}({n}\tilde{\mu}^n_s))\right|}\\
\shoveright{ \left.+ \sum_{l=1}^{{n}\tilde{\mu}^n_s\left(\N_+\right)} 
\left|f\left(a_{l}\left({n}\tilde{\mu}^n_s\right)\right) - f\left(a_{l}\left({n}\tilde{\mu}^n_s\right) - b_{\bphi,l}\left({n}\tilde{\mu}^n_s\right)-b'_{\bphi,l}\left({n}\tilde{\mu}^n_s\right)\right) \right| \Biggl)\right]\d s}\\
\le \int^{T_n\wedge \tau^n_\beta}_{S_n\wedge \tau^n_\beta}\parallel f \parallel\left(2+2\E_{\tilde{\mu}^n_s}\left[K_{{\mathbf \Phi}}({n}\tilde{\mu}^n_s)+K'_{{\mathbf \Phi}}({n}\tilde{\mu}^n_s)-2\right] \right)\d s 
\le {4 \delta\parallel f \parallel C',}\label{eq:boundcrochet}
	\end{multline}
	for some constant $C'$ that also depends only on $\beta$ and $M$. 
	We can then conclude, as above, that \eqref{eq:tightnew2} holds for 
	a small enough $\delta$ and a large enough $n_0$.
\noindent {This completes the proof of tighness of $\suite{\tilde\mu^n}$ in $\mathbb{D}([0,1], (\maM,\tau_v))$.}
%
%

To prove that the sequence is $\mathbb C$-tight, we need to show additionally that for any 
$f\in \maC_K$ the sequence of processes $\suite{\cro{\tilde\mu^{n,\tau^n_\beta}_.,f}}$ is $\mathbb C$-tight in 
$\mathbb{D}([0,1], \R)$. For this, for all $n$ and $t\in [0,1]$ we let $\delta \cro{\tilde\mu^{n,\tau^n_\beta}_t,f}$ 
be the jump of $\cro{\tilde\mu^n_t,f}$ at time $t$. Then, we get that 
$\delta \cro{\tilde\mu^{n,\tau^n_\beta}_t,f}$ is either 0 or $\frac{1}{n} \cro{\vartheta\left(n\tilde\mu^{n,\tau^n_\beta}_t\right)}$, and 
\[
\frac{1}{n} \left|\cro{\vartheta\left(n\tilde\mu^{n,\tau^n_\beta}_t\right),f}\right| \leq \frac{1}{n} \lVert f \rVert (2 + 4 o(\sqrt(n))), 
\]
which shows that $\delta \cro{\tilde\mu^{n,\tau^n_\beta}_t,f} \xlongrightarrow{n} 0$ a.s. for all $t$, proving that {(vague)} subsequential limits 
indeed have continuous paths. 
\end{proof}

\noindent We can deduce the tightness of $\suite{\tilde\mu^{n,\tau^n_\beta}}$ in $\mathbb{D}([0,1], (\maM,\tau_w))$ for any $\beta>0$, where $\tau_w$ 
denotes the weak topology. We stress that weak tightness is indeed needed, since the matching coverage depends on non-compactly supported functions, such as $\mathbb 1_0$.

\begin{corollary}
\label{cor:tightw}
The sequence of processes $\suite{\tilde\mu^{n,\tau^n_\beta}}$ is {$\mathbb C$-tight} in $\mathbb{D}([0,1], (\maM,\tau_w))$.
\end{corollary}
\begin{proof}
Following \cite{Mlard1993SurLC}, $\suite{\tilde\mu^{n,\tau^n_\beta}}$ is weakly {$\mathbb C$-tight}, 
if it is vaguely $\mathbb C$-tight and the mass process $\suite{\cro{\tilde\mu^{n,\tau^n_\beta}_.,\mathbb 1}}$ is tight in 
$\mathbb{D}([0,1], \R)$. So we have yet to show the tightness of $\suite{\cro{\tilde\mu^{n,\tau^n_\beta}_.,\mathbb 1}}$ in $\mathbb{D}([0,1], \R)$. 
The proof of the latter is exactly the same as in the proof of Lemma \ref{lemma:tightv}, for $f\equiv \mathbb 1$.
\end{proof}

\subsection{Convergence to the unique solution}
The above result shows the existence of subsequential limits for the sequence of stopped processes $\suite{\tilde\mu^{n,\tau^n_\beta}}$. 
To show that this implies the convergence of the sequence $\suite{\tilde\mu^{n,}}$ in a suitable sense, we will need to control the martingale term 
in the semi-martingale decomposition \eqref{Martingale}. 
\begin{lemma}
\label{lemma:Doob}
For any $f\in\maC_b$, the sequence of processes $\suite{\procun{\maM^{f,n}_{t\wedge \tau^n_\beta}}}$ defined by \eqref{Martingale}, 
converges in probability to $0$ uniformly over $[0,1]$. 
\end{lemma}
\begin{proof}
For all $n$ and $t$, reasonning as in \eqref{eq:boundcrochetter} we immediately get that 
$$\esp{\left(\maM^{f,n}_{t\wedge \tau^n_\beta}\right)^2} = \esp{\left\langle\maM^{f,n}\right\rangle_{t\wedge \tau^n_\beta}} \le \frac{2t \parallel f \parallel^2}{n}\left(1+C\right)
,$$ 
and we conclude using Doob's inequality.
\end{proof} 
\noindent We deduce the following result. 
\begin{proposition}
\label{prop:convt2beta}
Under the assumptions of Theorem \ref{thm:main}, there exists a deterministic time $0<t^*_{2\beta}\le 1$ such that 
a unique solution $\bar\mu^*$ to \eqref{eq:limhydro} exists on $[0,t^*_{2\beta}]$, and such that 
\[\sup_{t \in [0, t^*_{2\beta}]} \left|\cro{\tilde{\mu}^n_t,f} -  \cro{\bar{\mu}^*_t,f}\right| \xrightarrow{(n,\P)} 0.\]
\end{proposition}
\begin{proof}
The argument are related, in another context, to those of Proposition 4.6 in \cite{Jam}. 
{From Corollary \ref{cor:tightw}, the sequence $\suite{\tilde\mu^{n,\tau^n_\beta}}$ is $\mathbb C$-tight for the weak topology. 
Therefore, from Prohorov's Lemma (see \cite{ethier2009markov}, p.104), it converges weakly (at least along a subsequence) to a subsequential limit $\bar\mu^\dag$. 
By the very Assumption \ref{hypo:Ho}, we therefore have that $\bar\mu^\dag_0=\nu$, and let us set the (possibly random) time 
\[\tau^\dag_{2\beta}=\inf\left\{0\le t\le 1\,:\,\cro{\bar\mu^\dag_t,\mathbb 1_{\N_+}}\le 2\beta\right\}\wedge 1.\]
Fix $f\in\maC_b$. For all $n\in\N_+$ and $t\in[0,1]$, we have the following identity,
\begin{multline}
\label{eq:convfinal0}
\int_0^{t\wedge \tau^n_\beta \wedge \tau^\dag_{2\beta}} \bar{\mathfrak{L}}_f\left(\tilde\mu^{n,\tau^n_\beta}_s\right)\d s\\
=\left(\int_0^{t\wedge \tau^\dag_{2\beta}} \bar{\mathfrak{L}}_f\left(\tilde\mu^{n,\tau^n_\beta}_s\right)\d s\right)
\mathbb 1_{\{\tau^n_\beta>\tau^\dag_{2\beta}\}}+\left(\int_0^{t\wedge \tau^n_\beta} \bar{\mathfrak{L}}_f\left(\tilde\mu^{n,\tau^n_\beta}_s\right)\d s\right)\mathbb 1_{\{\tau^n_\beta\le \tau^\dag_{2\beta}\}}.
\end{multline}
Now, it follows from Lemma A.5 in \cite{decreusefond2012} that the map
\[\begin{cases}
\mathbb D\left([0,1],\bar {\M}_{\beta,M}\right) &\longrightarrow \mathbb D\left([0,1],\R_+\right)\\
\bar\mu_. &\longmapsto \cro{\bar\mu_.,\mathbb 1_{\N_+}}
\end{cases}
\]
is continuous for the Skorokhod topology. Therefore, as the map 
\[\begin{cases}
\mathbb D\left([0,1],\R\right) &\longrightarrow \R\\
x_. &\longmapsto \inf_{t\in[0,1]} x_t
\end{cases}
\]
is also continuous, it follows from the Continuous Mapping Theorem that, along the above subsequence, the following convergence in distribution holds:
\[\inf_{t\in[0,1]} \cro{\tilde\mu^{n,\tau^n_\beta}_{t\wedge \tau^\dag_{2\beta}},\mathbb 1_{\N_+}} 
\Longrightarrow \inf_{t\in[0,1]}\cro{\bdmu_{t\wedge \tau^\dag_{2\beta}},\mathbb 1_{\N_+}},\]
and Fatou's Lemma thus implies that 
\begin{multline}
\label{eq:convfinal-1}
 \lim_{n\to\infty}\pr{\tau^n_\beta>\tau^\dag_{2\beta}}\ge  \lim_{n\to\infty}\pr{\inf_{t\in[0,1]} \cro{\tilde\mu^{n,\tau^n_\beta}_{t\wedge \tau^\dag_{2\beta}},\mathbb 1_{\N_+}}>\beta}
\ge  \pr{ \inf_{t\in[0,1]}\cro{\bdmu_{t\wedge \tau^\dag_{2\beta}},\mathbb 1_{\N_+}}>\beta}=1,
\end{multline}
thereby implying that the second term on the right-hand side of \eqref{eq:convfinal0} converges in probability to 0. 
We now investigate the left-hand term of \eqref{eq:convfinal0}. For this, let us observe, again in view of Lemma A.5 in \cite{decreusefond2012}, that the mappings 
\[\begin{cases}
\mathbb D\left([0,1],\bar {\M}_{\beta,M}\right) &\longrightarrow \mathbb D\left([0,1],\R\right)\\
\bar\mu_. &\longmapsto \cro{\bar\mu_.,f}\\
	        &\longmapsto \cro{\bar\mu_.,\chi}\\
	        &\longmapsto \cro{\bar\mu_.,\chi-1}\\
	        &\longmapsto \cro{\bar\mu_.,\chi\nabla f}
\end{cases}
\]
are continuous for the Skorokhod topology, as well as 
\[\begin{cases}
\mathbb D\left([0,1],\R\times \R^*\right) &\longrightarrow \mathbb D\left([0,1],\R\right)\\
(x_.,y_.) &\longmapsto {x_.\over y_.}\cdot
\end{cases}
\]
Moreover, it follows from the continuity assumption on $\bphi$ (recall Definition \ref{def:polCont}), 
that the mappings $\bar{\maK}_\bphi$ and $\bar{\maK}'_\bphi(.,k)$, $k\in\N_+$, 
respectively defined by \eqref{eq:defKbar} and \eqref{eq:defK'bar}, are continuous from $\bar {\M}_{\beta,M}$ to $\M_F(\N_+)$. 
Therefore the mappings 
\[\begin{cases}
\bar {\M}_{\beta,M} &\longrightarrow \R\\
\bar\mu &\longmapsto \cro{\bar\maK_\bphi(\bar\mu),f};\\
             &\longmapsto \cro{\bar\maK'_\bphi(\bar\mu,k),f},\,\, k\in \N_+;\\
             &\longmapsto \cro{\bar\maK_\bphi(\bar\mu),\chi-\mathbb 1};\\
             &\longmapsto \cro{\bar\maK'_\bphi(\bar\mu,k),\chi-\mathbb 1},\,\, k\in \N_+
\end{cases}
\]
are also continuous, which implies in turn, as $\bphi$ is moment preserving, and by dominated convergence, that the mappings 
\[\begin{cases}
\bar {\M}_{\beta,M} &\longrightarrow \R\\
\bar\mu &\longmapsto \cro{\bar\maK_\bphi(\bar\mu),\cro{\bar\maK'_\bphi(\bar\mu,.),f}};\\
             &\longmapsto \cro{\bar\maK_\bphi(\bar\mu),\cro{\bar\maK'_\bphi(\bar\mu,.),\chi-\mathbb 1}}
             \end{cases}
\]
are also continuous. It can then finally be deduced exactly as in Lemma A.5 in \cite{decreusefond2012} that the mappings 
\[\begin{cases}
\mathbb D\left([0,1],\bar {\M}_{\beta,M}\right) &\longrightarrow \mathbb D\left([0,1],\R\right)\\
\bar\mu_. &\longmapsto \cro{\bar\maK_\bphi(\bar\mu_.),f};\\
	     &\longmapsto \cro{\bar\maK_\bphi(\bar\mu_.),\cro{\bar\maK'_\bphi(\bar\mu_.,.),f}};\\
	     &\longmapsto \cro{\bar\maK_\bphi(\bar\mu_.),\chi-\mathbb 1};\\
             &\longmapsto \cro{\bar\maK_\bphi(\bar\mu_.),\cro{\bar\maK'_\bphi(\bar\mu_.,.),\chi-\mathbb 1}}
\end{cases}
\]
are also continuous for the Skorokhod topology. All in all, as the mapping 
\[\begin{cases}
\mathbb D\left([0,1],\R\right) &\longrightarrow \mathbb C\left([0,1],\R\right)\\
x_. &\longmapsto \displaystyle\int_0^. x_s\d s
\end{cases}
\]
is also continuous, we deduce from \eqref{eq:convfinal0}, the above remark and the Continuous Mapping Theorem that, along the same sub-sequence, the following weak convergence holds in $\mathbb C\left([0,1],\R\right)$: 
\begin{equation}
\label{eq:convfinal1}
\int_0^{.\wedge \tau^n_\beta \wedge \tau^\dag_{2\beta}} \bar{\mathfrak{L}}_f\left(\tilde\mu^{n,\tau^n_\beta}_s\right)\d s
\,\,\xRightarrow{n}\,\, \int_0^{.\wedge \tau^\dag_{2\beta}} \bar{\mathfrak{L}}_f\left(\bar\mu^{\dag}_s\right)\d s.
\end{equation}
Now, recalling \eqref{eq:fight1} and \eqref{Martingale} we get that a.s. for all $n\in\N_+$ and $t\in[0,1]$, 
\begin{multline}
\label{eq:convfinal0}
\cro{\tilde\mu^{n,\tau^n_\beta}_{t\wedge \tau^\dag_{2\beta}},f} =\cro{\tilde{\mu}^n_0,f} + \int_0^{t\wedge \tau^n_\beta \wedge\tau^\dag_{2\beta}}
\tilde{\mathfrak{L}}^n \Pi_f(\tilde{\mu}^n_s) \d s +\maM_{t\wedge \tau^n_\beta \wedge\tau^\dag_{2\beta}}^{f,n}\\
\shoveleft{= \cro{\tilde{\mu}^n_0,f} + \int_0^{t\wedge \tau^n_\beta \wedge\tau^\dag_{2\beta}}
{\mathfrak{L}}^n \Pi_f(n\tilde{\mu}^n_s) \d s +\maM_{t\wedge \tau^n_\beta \wedge\tau^\dag_{2\beta}}^{f,n}}\\
\shoveleft{= \cro{\tilde{\mu}^n_0,f} + \left(\int_0^{t\wedge \tau^n_\beta \wedge\tau^\dag_{2\beta}}
{\mathfrak{L}}^n \Pi_f(n\tilde{\mu}^n_s) \d s\right)\mathbb 1_{\{\mu^n_0\in \M^n_{\beta,M}\}}}\\\
+ \left(\int_0^{t\wedge \tau^n_\beta \wedge\tau^\dag_{2\beta}}{\mathfrak{L}}^n \Pi_f(n\tilde{\mu}^n_s) \d s\right)\mathbb 1_{\{\mu^n_0\not\in \M^n_{\beta,M}\}} +\maM_{t\wedge \tau^n_\beta \wedge\tau^\dag_{2\beta}}^{f,n}.
\end{multline}
But \eqref{eq:convfinal2} implies in particular that for all $n,t$ and $\eta>0$,  
\begin{equation}
\label{eq:convfinal3}
\pr{\left(\int_0^{t\wedge \tau^n_\beta \wedge\tau^\dag_{2\beta}}{\mathfrak{L}}^n \Pi_f(n\tilde{\mu}^n_s) \d s\right)\mathbb 1_{\{\mu^n_0\not\in \M^n_{\beta,M}\}}>\eta} 
\le \pr{\mu^n_0\not\in \M^n_{\beta,M}}\xlongrightarrow{n} 0. 
\end{equation}
Second, it is immediate, as the mapping $s\mapsto \cro{n\tilde{\mu}^n_s,\chi^3}$ is clearly a.s. non-increasing, that on the event 
$\{\mu^n_0\in \M^n_{\beta,M}\}$ the measure $n\tilde{\mu}^n_s$ is an element of 
$\M^n_{\beta,M}$ for any $s\le t\wedge \tau^n_\beta \wedge\tau^\dag_{2\beta}$. Therefore, from Corollary \ref{cor:approxgen} there exists a process 
$\epsilon^{n,\beta}$ that vanishes in probability and uniformly over compact sets, and such that a.s. for all $n\in\N_+$ and $t\in[0,1]$,
 \begin{align}
\left(\int_0^{t\wedge \tau^n_\beta \wedge\tau^\dag_{2\beta}}
{\mathfrak{L}}^n \Pi_f(n\tilde{\mu}^n_s) \d s\right)\mathbb 1_{\{\mu^n_0\in \M^n_{\beta,M}\}}
&= \left(\int_0^{t\wedge \tau^n_\beta \wedge\tau^\dag_{2\beta}}
\wh{\mathfrak{L}}_f(n\tilde{\mu}^n_s) \d s+\epsilon^{n,\beta}_t\right)\mathbb 1_{\{\mu^n_0\in \M^n_{\beta,M}\}}\notag\\
&= \left(\int_0^{t\wedge \tau^n_\beta \wedge\tau^\dag_{2\beta}}
\wh{\mathfrak{L}}_f\left(n\tilde{\mu}^{n,\tau^n_\beta}_s\right) \d s+\epsilon^{n,\beta}_t\right)\mathbb 1_{\{\mu^n_0\in \M^n_{\beta,M}\}}\notag\\
&= \left(\int_0^{t\wedge \tau^n_\beta \wedge\tau^\dag_{2\beta}}
\bar{\mathfrak{L}}_f\left(\tilde{\mu}^{n,\tau^n_\beta}_s\right) \d s+\epsilon^{n,\beta}_t\right)\mathbb 1_{\{\mu^n_0\in \M^n_{\beta,M}\}},\label{eq:convfinal4}
\end{align}
where the third equality is obtained by comparing \eqref{eq:defLhat} to \eqref{eq:defLbar}, and recalling the definitions \eqref{eq:defKbar} and \eqref{eq:defK'bar}. 
Therefore, in view of  \eqref{eq:convfinal1} and \eqref{eq:convfinal2} we obtain the following weak convergence in $\mathbb D\left([0,1],\R\right)$, along the above subsequence, 
\begin{equation*}
\left(\int_0^{.\wedge \tau^n_\beta \wedge\tau^\dag_{2\beta}}
{\mathfrak{L}}^n \Pi_f(n\tilde{\mu}^n_s) \d s\right)\mathbb 1_{\{\mu^n_0\in \M^n_{\beta,M}\}} \xRightarrow{n} 
\int_0^{.\wedge \tau^\dag_{2\beta}} \bar{\mathfrak{L}}_f\left(\bar\mu^{\dag}_s\right)\d s.
\end{equation*}
Gathering this with \eqref{eq:convfinal3}, Assumption \ref{hypo:Ho} and Lemma \ref{lemma:Doob}, 
we obtain that the process on the right-hand side of \eqref{eq:convfinal0} converges weakly in $\mathbb D\left([0,1],\R\right)$, along the same subsequence, 
to 
\[\cro{\nu,f} + \int_0^{.\wedge \tau^\dag_{2\beta}} \bar{\mathfrak{L}}_f\left(\bar\mu^{\dag}_s\right)\d s.\]
Thus, in view of Skorokhod Representation Theorem, we obtain that on some probability space, a.s. 
\[\cro{\bar\mu^{\dag}_{t\wedge \tau^\dag_{2\beta}},f}=\cro{\nu,f} + \int_0^{t\wedge \tau^\dag_{2\beta}} \bar{\mathfrak{L}}_f\left(\bar\mu^{\dag}_s\right)\d s.\]
Therefore the process $\procun{\bar\mu^\dag_t}$ is a.s. an element of $\mathbb C\left([0,1],\bar{\M}\right)$ having deterministic initial value $\nu$, and 
solving the system \eqref{eq:limhydro} on the interval $[0,\tau^\dag_{2\beta}]$. But as the latter system has at most one solution in 
$\mathbb D\left([0,1],\bar{\M}\right)$, we conclude, first, that $\tau^\dag_{2\beta}$ is deterministic and second, that there exists a unique solution 
$\bar\mu^*$ to \eqref{eq:limhydro} on $[0,\tau^\dag_{2\beta}]$, and that this solution coincides a.s. with $\proc{\bar\mu^\dag_t}$ on 
$[0,\tau^\dag_{2\beta}]$. In turn, we obtain that $\tau^\dag_{2\beta}=t^*_{2\beta}$, where
\[t^*_{2\beta}:=\inf\left\{0\le t\le 1\,:\,\cro{\bar\mu^*_t,\mathbb 1_{\N_+}}\le 2\beta\right\}\wedge 1,\]
which is strictly positive in view of the continuity of the paths of $\bar\mu^*$. We can then replicate the arguments that led to \eqref{eq:convfinal1}, to obtain that 
for all $f\in\maC_b$, the following weak convergence holds in $\mathbb C\left([0,1],\R\right)$, 
\begin{equation*}
\int_0^{.\wedge \tau^n_\beta \wedge t^*_{2\beta}} \bar{\mathfrak{L}}_f\left(\tilde\mu^{n,\tau^n_\beta}_s\right)\d s
\,\,\xRightarrow{n}\,\, \int_0^{.\wedge t^*_{2\beta}} \bar{\mathfrak{L}}_f\left(\bar\mu^*_s\right)\d s.
\end{equation*}
Applying again Skorokhod Representation Theorem together with the fact that the Skorokhod topology and the topology of uniform convergence coincide 
on $\mathbb C\left([0,1],\R\right)$ (see e.g. \cite{Bill}, p.112), we deduce, on some probability space, the a.s. convergence
\begin{equation*}
\sup_{t\in[0,1]}\left|\int_0^{t\wedge \tau^n_\beta \wedge t^*_{2\beta}} \bar{\mathfrak{L}}_f\left(\tilde\mu^{n,\tau^n_\beta}_s\right)\d s
-\int_0^{t\wedge t^*_{2\beta}} \bar{\mathfrak{L}}_f\left(\bar\mu^*_s\right)\d s\right| \xlongrightarrow{n} 0\quad \mbox{a.s.,}
\end{equation*}
implying in turn the convergence in probability 
\begin{equation}
\label{eq:convfinal6}
\sup_{t\in[0,t^*_{2\beta}]}\left|\int_0^{t\wedge \tau^n_\beta} \bar{\mathfrak{L}}_f\left(\tilde\mu^{n,\tau^n_\beta}_s\right)\d s
-\int_0^{t} \bar{\mathfrak{L}}_f\left(\bar\mu^*_s\right)\d s\right| \xlongrightarrow{(n,\mathbb P)} 0.
\end{equation}
To conclude, it follows from \eqref{Martingale}, \eqref{eq:convfinal0} and \eqref{eq:convfinal4} that for all $f\in \maC_b$, and all $\eta>0$, 
\begin{multline*}
\pr{\sup_{t\in[0,t^*_{2\beta}]}\left|\cro{\bar\mu^n_t,f}-\cro{\bar\mu^*_t,f}\right|>\eta}\\
\shoveleft{\le \pr{\left\{\sup_{t\in[0,t^*_{2\beta}]}\left|\int_0^{t\wedge \tau^n_\beta} \bar{\mathfrak{L}}_f\left(\tilde\mu^{n,\tau^n_\beta}_s\right)\d s
-\int_0^{t} \bar{\mathfrak{L}}_f\left(\bar\mu^*_s\right)\d s\right|>\eta/4\right\}\cap\left\{\tau^n_\beta>t^*_{2\beta}\right\}}}\\
+\pr{\left|\cro{\tilde{\mu}^n_0,f}-\cro{\nu,f}\right|>\eta/4} + \pr{\sup_{t\in[0,t^*_{2\beta}]}\left|\epsilon^{n,\beta}_t\right|>\eta/4} \\
\shoveright{+ \pr{\left\{\sup_{t\in[0,t^*_{2\beta}]}\left|\maM_{t\wedge \tau^n_\beta}^{f,n}\right|>\eta/4\right\}\cap\left\{\tau^n_\beta>t^*_{2\beta}\right\}}+\pr{\tau^n_\beta\le t^*_{2\beta}}}\\
\xlongrightarrow{n} 0,
\end{multline*}
applying \eqref{eq:convfinal6}, and again Assumption \ref{hypo:Ho}, Lemma \ref{lemma:Doob} and \eqref{eq:convfinal-1}. 
This concludes the proof.}
\end{proof}
\noindent We now quantity the error made in the approximation of $\procun{\bar\mu^n_t}$ by the CTMC $\procun{\tilde\mu^n_t}$ on 
the interval $[0,t^*_{2\beta}]$. We have the following result. 

\begin{proposition}
\label{prop:approxtildebar}
Under the conditions of Theorem \ref{thm:main}, for all $f\in \maC_b$ and $\epsilon>0$ we have that 
\[\pr{\sup_{t\in[0,t^*_{2\beta}]}\left|\cro{\bar{\mu}^n_t,f} - \cro{\tilde{\mu}^n_t,f}\right |>\epsilon} \xlongrightarrow{n} 0.\]
\end{proposition}
\begin{proof} 
Fix $t\in[0,1]$ and $n\in\N_+$. 
By setting $\gamma^n_t = \tau_{\lfloor nt \rfloor}$, we obtain that $\bar{\mu}_t^n = \tilde{\mu}^n_{\gamma^n_t}$, and thereby from \eqref{Martingale}, that 
\[\cro{\bar{\mu}^n_t,f} = \cro{\bar{\mu}^n_0,f} + \int_0^{\gamma^n_t}{\tilde{\mathfrak{L}}^n}\Pi_f\left(\tilde{\mu}^n_s\right) \d s + \maM^{f,n}_{\gamma^n_t},\]
implying in turn that 
\begin{equation}
\label{eq:diamond0}
\left(\cro{\bar{\mu}^{n,\tau^n_\beta}_t,f} - \cro{\tilde{\mu}^{n,\tau^n_\beta}_t,f}\right)^2 \le 2\left(\int^{\gamma^n_t\wedge \tau^n_\beta}_{t\wedge \tau^n_\beta} \E_{\tilde{\mu}^n_s}\cro{\vartheta(n\tilde{\mu}^n_s),f}\d s\right)^2+2\left(\maM^{f,n}_{\gamma^n_t}-\maM^{f,n}_{t}\right)^2,
\end{equation}
But, first, reasoning as in \eqref{eq:boundcrochet} we get that
\begin{equation}
\label{eq:diamond1}
\esp{\left(\int^{\gamma^n_t\wedge \tau^n_\beta}_{t\wedge \tau^n_\beta} \E_{\tilde{\mu}^n_s}\cro{\vartheta(n\tilde{\mu}^n_s),f}\d s\right)^2\mathbb 1_{\{\mu^n_0\in \M^n_{\beta,M}\}}}\le  {16 \esp{(\gamma^n_t - t)^2}\parallel f \parallel^2 (C')^2}.
\end{equation}
On another hand, setting $\procun{\mathscr F^n_t}$, the natural filtration associated to $\procun{\mu^n_t}$, 
\begin{align*}
\esp{\left(\maM^{f,n}_{\gamma^n_t}-\maM^{f,n}_{t}\right)^2\mathbb 1_{\{\mu^n_0\in \M^n_{\beta,M}\}}}
&= \esp{\esp{\left(\maM^{f,n}_{\gamma^n_t}-\maM^{f,n}_{t}\right)^2\mathbb 1_{\{\mu^n_0\in \M^n_{\beta,M}\}}\mid \mathscr F^n_{\gamma^n_t\wedge t}}}\notag\\
&=\esp{\mathbb 1_{\{\mu^n_0\in \M^n_{\beta,M}\}}\esp{\left(\maM^{f,n}_{\gamma^n_t\vee t}-\maM^{f,n}_{\gamma^n_t\wedge t}\right)^2\mid \mathscr F^n_{\gamma^n_t\wedge t}}}\notag\\
&=\esp{\mathbb 1_{\{\mu^n_0\in \M^n_{\beta,M}\}}\esp{\left(\maM^{f,n}_{\gamma^n_t\vee t}\right)^2-\left(\maM^{f,n}_{\gamma^n_t\wedge t}\right)^2\mid \mathscr F^n_{\gamma^n_t\wedge t}}}\notag\\
&=\esp{\mathbb 1_{\{\mu^n_0\in \M^n_{\beta,M}\}}\esp{\llangle\maM^{f,n}\rrangle_{\gamma^n_t\vee t}-\llangle\maM^{f,n}\rrangle_{\gamma^n_t\wedge t}\mid \mathscr F^n_{\gamma^n_t\wedge t}}}\notag\\
&=\esp{\left|\llangle\maM^{f,n}\rrangle_{\gamma^n_t}-\llangle\maM^{f,n}\rrangle_{t}\right|\mathbb 1_{\{\mu^n_0\in \M^n_{\beta,M}\}}}\notag\\
&\le  {{2\esp{\left|\gamma^n_t - t\right|}\parallel f \parallel^2 (1+C)\over n},}
\end{align*}
where we apply the same argument as in \eqref{eq:boundcrochetbis}. This, together with \eqref{eq:diamond1} in \eqref{eq:diamond0}, yields to 
\begin{multline}
\label{eq:diamond2}
\esp{\left(\cro{\bar{\mu}^{n,\tau^n_\beta}_t,f} - \cro{\tilde{\mu}^{n,\tau^n_\beta}_t,f}\right)^2\mathbb 1_{\{\mu^n_0\in \M^n_{\beta,M}\}}}\\
\le {32 \esp{(\gamma^n_t - t)^2}\parallel f \parallel^2 (C')^2} +  {{4\esp{\left|\gamma^n_t - t\right|}\parallel f \parallel^2 (1+C)\over n}\cdot}
\end{multline}
But 
\begin{align}
\esp{|\gamma^n_t - t|} &= \esp{ \left|\gamma^n_t - {\lfloor nt \rfloor \over n} - \frac{nt - \lfloor nt \rfloor}{n}\right|}\notag\\
&\leq\esp{\left| \sum_{i=1}^{\lfloor nt \rfloor} \frac{\theta_{i} - 1}{n}\right|}+ {1\over n}\notag\\
&\leq  \esp{\sqrt{\sum_{i=1}^{\lfloor nt \rfloor} \frac{\theta_{i} - 1}{n}\sum_{j=1}^{\lfloor nt \rfloor} \frac{\theta_{j} - 1}{n}}}+ {1\over n}
\leq  \left(\sum_{i=1}^{\lfloor nt \rfloor} \frac{1}{n^2}\right)^{1/2}+ {1\over n}
\leq \left({1\over n}\right)^{1/2}+ {1\over n},\label{eq:diamond3}
\end{align}
using Jensen's inequality and the fact that the $\theta_i$'s are IID with second moment $\esp{ \theta_i^2} = 2$ for all $i$. Likewise, we get 
\begin{align}
\esp{(\gamma^n_t - t)^2} 
&\leq2\esp{\left( \sum_{i=1}^{\lfloor nt \rfloor} \frac{\theta_{i} - 1}{n}\right)^2}+ {2\over n^2}\notag\\
&\leq  \sum_{i=1}^{\lfloor nt \rfloor} \frac{2}{n^2}+ {2\over n^2}
\leq {2\over n}+ {2\over n^2}\cdot\label{eq:diamond4}
\end{align}
Finally, applying Bienaymé-Cebicev inequality, 
\begin{multline*}
\pr{\sup_{t\in[0,t^*_{2\beta}]}\left|\cro{\bar{\mu}^n_t,f} - \cro{\tilde{\mu}^n_t,f}\right |>\epsilon}\\
\begin{aligned}
&\le \pr{\sup_{t\in[0,t^*_{2\beta}]}\left|\cro{\bar{\mu}^{n,\tau^n_\beta}_t,f} - \cro{\tilde{\mu}^{n,\tau^n_\beta}_t,f}\right |>\epsilon}+\pr{\tau^n_\beta\le t^*_{2\beta}}\\
&\le \pr{\sup_{t\in[0,t^*_{2\beta}]}\left|\cro{\bar{\mu}^{n,\tau^n_\beta}_t,f} - \cro{\tilde{\mu}^{n,\tau^n_\beta}_t,f}\right |\mathbb 1_{\{\mu^n_0\in \M^n_{\beta,M}\}}>\epsilon/2}
+ \pr{\mu^n_0\not\in \M^n_{\beta,M}}+\pr{\tau^n_\beta\le t^*_{2\beta}}\\
& \le {4\over \epsilon^2}\esp{\sup_{t\in[0,t^*_{2\beta}]}\left(\cro{\bar{\mu}^{n,\tau^n_\beta}_t,f} - \cro{\tilde{\mu}^{n,\tau^n_\beta}_t,f}\right )^2\mathbb 1_{\{\mu^n_0\in \M^n_{\beta,M}\}}}
+ \pr{\mu^n_0\not\in \M^n_{\beta,M}}+\pr{\tau^n_\beta\le t^*_{2\beta}}\xlongrightarrow{n} 0,
\end{aligned}
\end{multline*}
gathering \eqref{eq:diamond3} with \eqref{eq:diamond4} in \eqref{eq:diamond2} and applying Fatou's Lemma regarding the first term, and from \eqref{eq:convfinal2} and \eqref{eq:convfinal-1} regarding the other two. This completes the proof. 
\end{proof}
\noindent We are now in a position to prove Theorem \ref{thm:main}. 
\begin{proof}[Proof of Theorem \ref{thm:main}]
In the proof of Proposition \ref{prop:convt2beta}, we have shown the existence of a unique solution $\bar\mu^*$ to \eqref{eq:limhydro} on $[0,t^*_{2\beta}]$. As this 
solution has continuous paths, by letting $\beta$ vanish to 0, the solution can be extended until $t^*_0$, defined by 
\[t^*_{0}:=\inf\left\{0\le t\le 1\,:\,\cro{\bar\mu^*_t,\mathbb 1_{\N_+}}\le 0\right\}\wedge 1.\]
Then, it trivially follows from \eqref{eq:defLbar}, that it necessary and sufficient to extend the latter process by just setting $\bar\mu^*_t=\bar\mu^*_{t^*_0}$, for all $t\ge t^*_0$, to solve the equation \eqref{eq:limhydro} on $[0,1]$.  This shows the existence and uniqueness of the solution $\procun{\bar\mu^*_t}$ on $[0,1]$. 

\medskip

We now turn to the proof of convergence to this solution. Fix $f\in\maC_b$ and $\epsilon>0$. The parameter $\beta$ in \eqref{eq:condbeta} and a positive number 
$\eta$ can be chosen in a way that 
\begin{equation}
\label{eq:condeta}\parallel f\parallel (8\beta+2\eta) \le \eps.
\end{equation}
Then, first, if $t^*_{2\beta}\ge 1$ we immediately deduce from Proposition \ref{prop:convt2beta} and Proposition \ref{prop:approxtildebar}, that 
\begin{multline*}
\pr{\sup_{t\in[0,1]} \left| \cro{\bar{\mu}^n_t,f} - \cro{\bar\mu^*_t,f}  \right| >\epsilon}\\
\le \pr{\sup_{t\in[0,1]} \left| \cro{\bar{\mu}^n_t,f} - \cro{\tilde\mu^*_n,f}  \right| >{\epsilon\over 2}}+\pr{\sup_{t\in[0,1]} \left| \cro{\tilde{\mu}^n_t,f} - \cro{\bar\mu^*_t,f}  \right| >{\epsilon\over 2}}\xlongrightarrow{n} 0.
\end{multline*}
Let us now address the case where $t^*_{2\beta}< 1$. We have 
\begin{multline}
\label{eq:convfinalfinal0}
\pr{\sup_{t\in[0,1]} \left| \cro{\bar{\mu}^n_t,f} - \cro{\bar\mu^*_t,f}  \right| >\epsilon}\\
\shoveleft{\le \pr{\sup_{t\in[0,t^*_{2\beta}]} \left| \cro{\bar{\mu}^n_t,f} - \cro{\bar\mu^*_t,f}  \right| >\epsilon} 
+ \pr{\sup_{t\in[t^*_{2\beta},1]} \left| \cro{\bar{\mu}^n_t,f} - \cro{\bar\mu^*_t,f}  \right| >\epsilon}}\\
\shoveleft{\le \pr{\sup_{t\in[0,t^*_{2\beta}]} \left| \cro{\bar{\mu}^n_t,f} - \cro{\tilde\mu^n_t,f}  \right| >{\epsilon\over 2}} 
+\pr{\sup_{t\in[0,t^*_{2\beta}]} \left| \cro{\tilde{\mu}^n_t,f} - \cro{\bar\mu^*_t,f}  \right| >{\epsilon\over 2}}}\\ 
+ \pr{\sup_{t\in[t^*_{2\beta},1]} |f(0)|\left| \bar\mu^n_t(0)-\bar\mu^*_t(0)\right| >{\epsilon \over 2}}
+\pr{\sup_{t\in[t^*_{2\beta},1]} \left| \cro{\bar{\mu}^n_t,f\mathbb 1_{\N_+}} - \cro{\bar\mu^*_t,f\mathbb 1_{\N_+}}  \right| >{\epsilon \over 2}}.
\end{multline}
Let us first address the third term on the right-hand side of \eqref{eq:convfinalfinal0}. For this, first observe that for 
any $s\le 1$, 
\[\cro{\bar\mu^*_s,\chi\nabla \mathbb 1_{\N_+}}=\sum_{i=0}^{+\infty} (\chi\nabla \mathbb 1_{\N_+})(i)\bar\mu^*_s(i)=\sum_{i=1}^{+\infty} i(\mathbb 1_{\N_+}(i)-\mathbb 1_{\N_+}(i-1))\bar\mu^*_s(i)=\bar\mu^*_s(1).\]
Thus, applying \eqref{eq:limhydro} to $f\equiv \mathbb 1_{\N_+}$, for all $t\in[t^*_{2\beta},1]$ we obtain that 
\begin{multline*}
\cro{\bar\mu^*_{t},\mathbb 1_{\N_+}} =  \cro{\bar\mu^*_{t^*_{2\beta}},\mathbb 1_{\N_+}} - \int_{t^*_{2\beta}}^{t\wedge t_0} \cro{\bar\maK_{\bphi}(\bar\mu^*_s),\mathbb 1_{\N_+}}
\d s-\int_{t^*_{2\beta}}^{t\wedge t_0} \cro{\bar\maK_{\bphi}(\bar\mu^*_s),\cro{\bar{\maK}'_{\bphi}(\bar\mu^*_s,.),\mathbb 1_{\N_+}}}\d s\\
\shoveright{-\int_{t^*_{2\beta}}^{t\wedge t_0}{\cro{\bar\mu^*_s,\chi\nabla \mathbb 1_{\N_+}}\over \cro{\bar\mu^*_s,\chi}}
\cro{\bar\maK_{\bphi}(\bar\mu^*_s),\chi-\mathbb 1+\cro{\bar\maK'_{\bphi}(\bar\mu^*_s,.),\chi-\mathbb 1}} \d s}\\
=\cro{\bar\mu^*_{t^*_{2\beta}},\mathbb 1_{\N_+}} -2\int_{t^*_{2\beta}}^{t\wedge t_0} 1\d s - \int_{t^*_{2\beta}}^{t\wedge t_0}{\bar\mu^*_s(1)\over \cro{\bar\mu^*_s,\chi}}
\cro{\bar\maK_{\bphi}(\bar\mu^*_s),\chi-\mathbb 1+\cro{\bar\maK'_{\bphi}(\bar\mu^*_s,.),\chi-\mathbb 1}} \d s.
\end{multline*}
Therefore, as $\cro{\bar\mu^*_{t},\mathbb 1_{\N_+}} \ge 0$ for all $t\ge t^*_{2\beta}$, we deduce that 
\begin{equation}
\label{eq:convfinalfinal11}
\int_{t^*_{2\beta}}^{t\wedge t_0}{\bar\mu^*_s(1)\over \cro{\bar\mu^*_s,\chi}}
\cro{\bar\maK_{\bphi}(\bar\mu^*_s),\chi-\mathbb 1+\cro{\bar\maK'_{\bphi}(\bar\mu^*_s,.),\chi-\mathbb 1}} \d s \le \cro{\bar\mu^*_{t^*_{2\beta}},\mathbb 1_{\N_+}}, \quad t\in [t^*_{2\beta},1]. 
\end{equation}
On another hand, for all $s\le 1$, 
\[\cro{\bar\mu^*_s,\chi\nabla \mathbb 1_0}=\sum_{i=0}^{+\infty} (\chi\nabla \mathbb 1_0)(i)\bar\mu^*_s(i)=
\sum_{i=0}^{+\infty} i(\mathbb 1_0(i)-\mathbb 1_0(i-1))\bar\mu^*_s(i)=\sum_{i=1}^{+\infty} i(\mathbb 1_0(i)-\mathbb 1_1(i))\bar\mu^*_s(i)=-\bar\mu^*_s(1).\]
Therefore, it follows from applying \eqref{eq:limhydro} 
to $f\equiv \mathbb 1_0$, that for all $t\in[t^*_{2\beta},1]$, 
\begin{multline}
\label{eq:convfinalfinal12}
\bar\mu^*_{t}(0) =  \bar\mu^*_{t^*_{2\beta}}(0) - \int_{t^*_{2\beta}}^{t\wedge t_0} \bar\maK_{\bphi}(\bar\mu^*_s)(0) 
\d s-\int_{t^*_{2\beta}}^{t\wedge t_0} \cro{\bar\maK_{\bphi}(\bar\mu^*_s),\bar{\maK}'_{\bphi}(\bar\mu^*_s,.)(0)}\d s\\
\shoveright{-\int_{t^*_{2\beta}}^{t\wedge t_0}{\cro{\bar\mu^*_s,\chi\nabla \mathbb 1_0}\over \cro{\bar\mu^*_s,\chi}}
\cro{\bar\maK_{\bphi}(\bar\mu^*_s),\chi-\mathbb 1+\cro{\bar\maK'_{\bphi}(\bar\mu^*_s,.),\chi-\mathbb 1}} \d s}\\
=\bar\mu^*_{t^*_{2\beta}}(0) + \int_{t^*_{2\beta}}^{t\wedge t_0}{\bar\mu^*_s(1)\over \cro{\bar\mu^*_s,\chi}}
\cro{\bar\maK_{\bphi}(\bar\mu^*_s),\chi-\mathbb 1+\cro{\bar\maK'_{\bphi}(\bar\mu^*_s,.),\chi-\mathbb 1}} \d s.
\end{multline}
Now, for all $t\in[t^*_{2\beta},1]$ we have 
\begin{equation}
\label{eq:convfinalfinal13}\bar\mu^n_t(0)=\bar\mu^n_{t^*_{2\beta}}(0)+{1\over n}Y^{n,\beta}_t,
\end{equation}
where, if $\tau_i\le t^*_{2\beta}<\tau_{i+1}$ and $\tau_j\le t<\tau_{j+1}$ in \eqref{eq:defmutilde}, $Y^{n,\beta}_t$ is, in the $n$-th graph, the number of nodes of degree at least 1 at step $i$ and that have become of degree 0 by step $j$. 
It is then immediate that we have 
\[{1\over n}Y^{n,\beta}_t \le \cro{\bar\mu^n_{t^*_{2\beta}},\mathbb 1_{\N_+}},\quad t\in[t^*_{2\beta},1],\]
and so gathering \eqref{eq:convfinalfinal11}, \eqref{eq:convfinalfinal12} and \eqref{eq:convfinalfinal13} we obtain that 
for all $t\in[t^*_{2\beta},1]$, 
\begin{align}
|f(0)|\left|\bar\mu^n_{t}(0) - \bar\mu^*_t(0)\right| &\le |f(0)|\Biggl(\left| \bar\mu^n_{t^*_{2\beta}}(0) -  \bar\mu^*_{t^*_{2\beta}}(0)\right| + {1\over n}Y^{n,\beta}_t\notag\\
&\quad\quad\quad\quad\quad\quad\quad +\int_{t^*_{2\beta}}^{t\wedge t_0}{\bar\mu^*_s(1)\over \cro{\bar\mu^*_s,\chi}}
\cro{\bar\maK_{\bphi}(\bar\mu^*_s),\chi-\mathbb 1+\cro{\bar\maK'_{\bphi}(\bar\mu^*_s,.),\chi-\mathbb 1}} \d s\Biggl)\notag\\
&\le  |f(0)|\left(\left| \bar\mu^n_{t^*_{2\beta}}(0) -  \bar\mu^*_{t^*_{2\beta}}(0)\right| +2\cro{\bar\mu^n_{t^*_{2\beta}},\mathbb 1_{\N_+}}\right)\notag\\
&\le \parallel f\parallel \left(\sup_{t\in[0,t^*_{2\beta}]} \left| \bar\mu^n_{t}(0) -  \bar\mu^*_{t}(0)\right|+2\cro{\bar\mu^n_{t^*_{2\beta}},\mathbb 1_{\N_+}}\right). 
\label{eq:convfinalfinal14}
\end{align}
But by continuity of the mapping $t\mapsto \cro{\bar\mu^*_{t},\mathbb 1_{\N_+}}$ on $[0,1]$, we have that 
\begin{equation}
\label{eq:convfinalfinal15}
\cro{\bar\mu^*_{t^*_{2\beta}},\mathbb 1_{\N_+}}=2\beta.
\end{equation}
Thus, it follows from \eqref{eq:convfinalfinal14} that
\begin{multline}\label{eq:convfinalfinal1}
\pr{\sup_{t\in[t^*_{2\beta},1]} |f(0)|\left| \bar\mu^n_t(0)-\bar\mu^*_t(0)\right| >{\epsilon \over 2}} 
\le \pr{\parallel f\parallel \left(\sup_{t\in[0,t^*_{2\beta}]} \left| \bar\mu^n_{t}(0) -  \bar\mu^*_{t}(0)\right|+2\cro{\bar\mu^n_{t^*_{2\beta}},\mathbb 1_{\N_+}}\right)>{\epsilon \over 2}}\\
\shoveleft{\le \pr{\parallel f\parallel \left(\eta+2\beta\right)>{\epsilon \over 2}} + \pr{\sup_{t\in[0,t^*_{2\beta}]} \left| \bar\mu^n_{t}(0) -  \bar\mu^*_{t}(0)\right|>\eta}}\\
\le \pr{\parallel f\parallel \left(\eta+2\beta\right)>{\epsilon \over 2}} + \pr{\sup_{t\in[0,t^*_{2\beta}]} \left| \bar\mu^n_{t}(0) -  \tilde\mu^n_{t}(0)\right|>{\eta\over 2}}
+ \pr{\sup_{t\in[0,t^*_{2\beta}]} \left| \tilde\mu^n_{t}(0) -  \bar\mu^*_{t}(0)\right|>{\eta\over 2}}
\xlongrightarrow{n} 0,
\end{multline}
from \eqref{eq:condeta}, and applying Propositions \ref{prop:convt2beta} and \ref{prop:approxtildebar} to $f=\mathbb 1_{0}$. 
We now turn to the fourth term on the right-hand side of \eqref{eq:convfinalfinal0}. 
We have that 
\begin{multline}\label{eq:convfinalfinal21}
\pr{\sup_{t\in[t^*_{2\beta},1]} \left| \cro{\bar{\mu}^n_t,f\mathbb 1_{\N_+}} - \cro{\bar\mu^*_t,f\mathbb 1_{\N_+}}  \right| >{\epsilon\over 2}}
\le \pr{\sup_{t\in[0,t^*_{2\beta}]} \left| \cro{\bar{\mu}^n_t,\mathbb 1_{\N_+}} - \cro{\bar\mu^*_t,\mathbb 1_{\N_+}}  \right| >\eta}\\
+ \pr{\left\{\sup_{t\in[t^*_{2\beta},1]} \left| \cro{\bar{\mu}^n_t,f\mathbb 1_{\N_+}} - \cro{\bar\mu^*_t,f\mathbb 1_{\N_+}}  \right| >{\epsilon \over 2}\right\}
\cap\left\{\sup_{t\in[0,t^*_{2\beta}]} \left| \cro{\bar{\mu}^n_t,\mathbb 1_{\N_+}} - \cro{\bar\mu^*_t,\mathbb 1_{\N_+}}  \right| \le \eta\right\}}\\
\shoveleft{\le \pr{\sup_{t\in[0,t^*_{2\beta}]} \left| \cro{\bar{\mu}^n_t,\mathbb 1_{\N_+}} - \cro{\tilde\mu^n_t,\mathbb 1_{\N_+}}  \right| >{\eta\over 2}}
+\pr{\sup_{t\in[0,t^*_{2\beta}]} \left| \cro{\tilde{\mu}^n_t,\mathbb 1_{\N_+}} - \cro{\bar\mu^*_t,\mathbb 1_{\N_+}}  \right| >{\eta\over 2}}}\\
+ \pr{\left\{\sup_{t\in[t^*_{2\beta},1]} \left| \cro{\bar{\mu}^n_t,f\mathbb 1_{\N_+}} - \cro{\bar\mu^*_t,f\mathbb 1_{\N_+}}  \right| >{\epsilon \over 2}\right\}
\cap\left\{\sup_{t\in[0,t^*_{2\beta}]} \left| \cro{\bar{\mu}^n_t,\mathbb 1_{\N_+}} - \cro{\bar\mu^*_t,\mathbb 1_{\N_+}}  \right| \le \eta\right\}}.
\end{multline}
It then readily follows from \eqref{eq:convfinalfinal15}, \eqref{eq:defLbar} and \eqref{eq:limhydro}, 
that $\cro{\bar\mu^*_t,\mathbb 1_{\N_+}}\le 2\beta$ for all $t\in\left[t^*_{2\beta},1\right]$. 
On the other hand, \eqref{eq:convfinalfinal15} also implies that, on the event 
$\left\{\sup_{t\in[0,t^*_{2\beta}]} \left| \cro{\bar{\mu}^n_t,\mathbb 1_{\N_+}} - \cro{\bar\mu_t,\mathbb 1_{\N_+}}  \right| \le \eta\right\}$ 
we have $$\cro{\bar\mu^n_{t^*_{2\beta}},\mathbb 1_{\N_+}}\le 2\beta +\eta\,\mbox{ a.s.}$$ and thus, as the mapping 
$t\mapsto  \cro{\bar\mu^n_{t},\mathbb 1_{\N_+}}$ is a.s. non-decreasing, we get that 
$\cro{\bar\mu^n_{t},\mathbb 1_{\N_+}}\le 2\beta +\eta$ a.s. for all $t\in\left[t^*_{2\beta},1\right]$. Therefore, 
\begin{multline*}
\pr{\left\{\sup_{t\in[t^*_{2\beta},1]} \left| \cro{\bar{\mu}^n_t,f\mathbb 1_{\N_+}} - \cro{\bar\mu^*_t,f\mathbb 1_{\N_+}}  \right| >{\epsilon \over 2}\right\}
\cap\left\{\sup_{t\in[0,t^*_{2\beta}]} \left| \cro{\bar{\mu}^n_t,\mathbb 1_{\N_+}} - \cro{\bar\mu^*_t,\mathbb 1_{\N_+}}  \right| \le \eta\right\}}\\
\le \pr{\left\{ \parallel f\parallel \left(\sup_{t\in[t^*_{2\beta},1]}\cro{\bar{\mu}^n_t,\mathbb 1_{\N_+}} + \sup_{t\in[t^*_{2\beta},1]}\cro{\bar{\mu}^*_t, \mathbb 1_{\N_+}}\right)   >{\epsilon \over 2}\right\}
\cap\left\{\sup_{t\in[0,t^*_{2\beta}]} \left| \cro{\bar{\mu}^n_t,\mathbb 1_{\N_+}} - \cro{\bar\mu^*_t,\mathbb 1_{\N_+}}  \right| \le \eta\right\}}\\
\le \pr{\parallel f\parallel \left(4\beta+\eta\right)   >{\epsilon \over 2}}=0. 
\end{multline*}
Injecting this and applying Propositions \ref{prop:convt2beta} and \ref{prop:approxtildebar} to $f=\mathbb 1_{\N_+}$ in \eqref{eq:convfinalfinal21} implies that 
\begin{equation*}
\pr{\sup_{t\in[t^*_{2\beta},1]} \left| \cro{\bar{\mu}^n_t,f\mathbb 1_{\N_+}} - \cro{\bar\mu^*_t,f\mathbb 1_{\N_+}}  \right| >{\epsilon \over 2}} \xlongrightarrow{n} 0.
\end{equation*}
Gathering this with \eqref{eq:convfinalfinal1} in \eqref{eq:convfinalfinal0}, and applying again Propositions \ref{prop:convt2beta} and \ref{prop:approxtildebar} to $f$, we obtain that 
\[\pr{\sup_{t\in[0,1]} \left| \cro{\bar{\mu}^n_t,f} - \cro{\bar\mu^*_t,f}  \right| >\epsilon}\xlongrightarrow{n} 0,\]
concluding the proof. 
\end{proof}


{We now provide three examples of local matching criteria, for which our main result holds. It is demonstrated that the algorithms {\sc greedy}, {\sc min-min} 
and {\sc uni-max} verify all the hypotheses of Theorem \ref{thm:main}, and thus that their respective matching coverages can be predicted by solving the corresponding ODE.}

\section{Proof of Theorem \ref{thm:matchcovgreedy}}
\label{sec:greedy}
Let us come back to the case of the matching criterion \textsc{greedy}, introduced in Section \ref{subsec:mainresalgo}, and formally defined in Example \ref{ex:greedy}. 
Fix $\mu$ such that $\cro{\mu,\mathbb 1_{N_+}}>0$, and recall the definition \eqref{eq:distribK}. Then, for any $k\in\N_+$ we clearly get that 
\[\maK_{\textsc{greedy}}(\mu)(k)={\mu(k)\over \cro{\mu,\mathbb 1_{\N_+}}},\,k\in\N_+,\]
and more generally for all $f\in\maC_b$, 
\begin{equation}
\label{eq:Phigreedy1}
\cro{\maK_{\textsc{greedy}}(\mu),f}  = {\cro{\mu,f\mathbb 1_{\N_+}}\over \cro{\mu,\mathbb 1_{\N_+}}}\cdot
\end{equation}
Also, for any $k\in \N_+$, given that $K_{\textsc{greedy}}(\mu)=k$, by the uniformity of the second choice, it is immediate that the distribution of 
$K'_{\textsc{greedy}}(\mu)$ is size-biased, namely, recalling \eqref{eq:distribK'}, 
\[\maK'_{\textsc{greedy}}(\mu,k)(k')={k'\mu(k')\over \cro{\mu,\chi}},\,k'\in\N_+.\]
Therefore, for all $k\in\N_+$, for all $f\in\maC_b$ we get 
\begin{equation}
\label{eq:Phigreedy2}
\cro{\maK'_{\textsc{greedy}}(\mu,k),f}=\sum_{k'\in\N_+}f(k')\maK'_{\textsc{greedy}}(\mu,k)(k')= {\cro{\mu,\chi f}\over\cro{\mu,\chi}},
\end{equation}
and thus for all such $f$, 
\begin{equation*}
\cro{\maK_{\textsc{greedy}}(\mu),\cro{\maK'_{\textsc{greedy}}(\mu,.),f}}=\sum_{k\in\N_+}{\cro{\mu,\chi f}\over\cro{\mu,\chi}}\maK_{\textsc{greedy}}(\mu)(k)
={\cro{\mu,\chi f}\over\cro{\mu,\chi}}\cdot 
\end{equation*}
{Observe that, by uniformity, the distribution $\wh\maK'_{\textsc{greedy}}(\mu,k)$ defined by \eqref{eq:distribhatK'} coincides with 
$\maK'_{\textsc{greedy}}(\mu,k)$ for all $k$, which trivially implies that {\sc greedy} is well-behaved. 
Second, for all $f\in\maC_b$ we thus also have that 
\begin{equation}
\label{eq:hatPhigreedy2}
\cro{\maK_{\textsc{greedy}}(\mu),\cro{\wh\maK'_{\textsc{greedy}}(\mu,.),f}} 
={\cro{\mu,\chi f}\over\cro{\mu,\chi}}\cdot 
\end{equation}}
We immediately check that {\sc greedy} is continuous in the sense of Definition \ref{def:polCont}: fix a sequence 
$\suite{\bar\mu^n}$ of $\bar\M$ such that $\bar\mu^n\in \bar\M^n$ for all $n$, and such that $\bar\mu^n\xRightarrow{n}\bar\mu$ for some 
$\bar\mu \in\bar\M$ such that $\cro{\bar\mu,\mathbb 1_{\N_+}}>0$. 
In view of \eqref{eq:Phigreedy1}, for any $f\in\maC_b$ we get that 
\[\cro{\maK_{\textsc{greedy}}(n\bar\mu^n),f}={\cro{n\bar\mu^n,f\mathbb 1_{\N_+}}\over \cro{n\bar\mu^n,\mathbb 1_{\N_+}}}
={\cro{\bar\mu^n,f\mathbb 1_{\N_+}}\over \cro{\bar\mu^n,\mathbb 1_{\N_+}}}\xlongrightarrow{n} {\cro{\bar\mu,f\mathbb 1_{\N_+}}\over \cro{\bar\mu,\mathbb 1_{\N_+}}}=:\cro{\bar{\maK}_{\textsc{greedy}}(\bar\mu),f}.\]
Likewise, for all $k\in\N_+$, for all $f\in\maC_b$, \eqref{eq:Phigreedy2} entails that 
\[\cro{\wh{\maK}'_{\textsc{greedy}}(n\bar\mu^n,k),f}=\cro{{\maK}'_{\textsc{greedy}}(n\bar\mu^n,k),f}={\cro{n\bar\mu^n,\chi f}\over\cro{n\bar\mu^n,\chi}}={\cro{\bar\mu^n,\chi f}\over\cro{\bar\mu^n,\chi}}
\xlongrightarrow{n}{\cro{\bar\mu,\chi f}\over\cro{\bar\mu,\chi}}=:\cro{\bar{\maK}'_{\textsc{greedy}}(\bar\mu,k),f}.\]
Hence the {\sc greedy} criterion is continuous, and moreover the operator $\bar{\mathfrak{L}}$ can be made explicit in the present case. 
Indeed, plugging 
the above into \eqref{eq:defLbar}, we immediately get that for all $\bar\mu\in \bar{\M}$ such that 
$\cro{\bar\mu,\mathbb 1_{\N_+}}>0$,  
\begin{multline*}
\bar{\mathfrak{L}}_f(\bar\mu)
=- \cro{\bar\maK_{\textsc{greedy}}(\bar\mu),f+\cro{\bar{\maK}'_{\textsc{greedy}}(\bar\mu,.),f}}\\
\shoveright{-{\cro{\bar\mu,\chi\nabla f}\over \cro{\bar\mu,\chi}}
\Bigl\{\cro{\bar\maK_{\textsc{greedy}}(\bar\mu),\chi-\mathbb 1+\cro{\bar\maK'_{\textsc{greedy}}(\bar\mu,.),\chi-\mathbb 1}}\Bigl\}}\\
=-\Biggl\{{\cro{\bar\mu,f\mathbb 1_{\N_+}}\over \cro{\bar\mu,\mathbb 1_{\N_+}}} +{\cro{\bar\mu,\chi f}\over\cro{\bar\mu,\chi}}
+ {\cro{\bar\mu,\chi\nabla f} \over \cro{\bar\mu,\chi}}\left\{{\cro{\bar\mu,(\chi- \mathbb 1_{\N_+})}\over \cro{\bar\mu,\mathbb 1_{\N_+}}}
+ {\cro{\bar\mu,\chi^2 - \chi} \over \cro{\bar\mu,\chi}}\right\} \Biggl\}.
\end{multline*}
We deduce the following result, which readily implies Theorem \ref{thm:matchcovgreedy} in view of Corollary \ref{cor:maincov}. 
\begin{corollary}[Convergence Theorem for the \textsc{greedy} criterion]
\label{cor:MainGRE}
If $\mathbf\Phi=\textsc{greedy}$ and the sequence of processes $\suite{\bar{\mu}^n}$ satisfies Assumption \ref{hypo:Ho}, 
for every $f\in \maC_b$ we get the convergence 
\begin{equation*}
\sup_{t \leq 1 } |\cro{\bar{\mu}^n_t,f} -  \cro{\bar{\mu}^{\textsc{greedy}}_t,f}| \xrightarrow{(n,\P)} 0,
\end{equation*}
where $\procun{\bar{\mu}_t^{\textsc{greedy}}}$ is the unique solution of \eqref{eq:limhydro}, for $\bar{\mathfrak{L}}$ defined, for all $f\in\maC_b,\bar\mu\in \bar{\M}$, by 
\begin{equation}
\label{eq:defLbarGRE}
\bar{\mathfrak{L}}_f(\bar\mu)=-\Biggl\{{\cro{\bar\mu,f\mathbb 1_{\N_+}}\over \cro{\bar\mu,\mathbb 1_{\N_+}}} +{\cro{\bar\mu,\chi f}\over\cro{\bar\mu,\chi}}
+ {\cro{\bar\mu,\chi\nabla f} \over \cro{\bar\mu,\chi}}\left\{{\cro{\bar\mu,(\chi- \mathbb 1_{\N_+})}\over \cro{\bar\mu,\mathbb 1_{\N_+}}}
+ {\cro{\bar\mu,\chi^2 - \chi} \over \cro{\bar\mu,\chi}}\right\} \Biggl\}\mathbb 1_{\cro{\bar\mu,\mathbb 1_{\N_+}}>0}.
\end{equation}
\end{corollary} 

\begin{proof}
We apply Theorem \ref{thm:main}. 
We have just shown that 
$\bphi=\textsc{greedy}$ is well-behaved and continuous. 
It also preserves the moments. To see this, observe that for all $\beta,M$ and $n$, for all $\mu \in {\M}^n_{\beta,M}$, in view of 
\eqref{eq:Phigreedy1} we have that 
 \[\cro{\maK_{\textsc{greedy}}(\mu),\chi^2} = {\cro{\mu,\chi^2}\over \cro{\mu,\mathbb 1_{\N_+}}}{\le {M\over \beta}}\cdot\]
 Likewise, from \eqref{eq:Phigreedy2} we get that 
 \[\sup_{k\in\N_+}\cro{\maK'_{\textsc{greedy}}({\mu},k),\chi^2}
 ={\cro{\mu,\chi^{3}}\over\cro{\mu,\chi}}{\le {M\over \beta}},\]
 {proving that \textsc{greedy} preserves the moments up to two.} 
{In view of Theorem \ref{thm:main}, it remains to show that 
the system of integral equations \eqref{eq:limhydro}, 
for $\bar{\mathfrak{L}}$ defined by \eqref{eq:defLbarGRE}, admits the only solution $\bar\mu^*$.} 
For this, recalling \ref{hypo:Ho}, let us first observe that  
\begin{equation}
\label{eq:momentboundsol}
\cro{\bar\mu^*_t , \chi^{k}}<\infty,\mbox{ for all $t\in[0,1]$, all $k\le 3.5 + \varepsilon$ and every solution $\bar\mu^*$ of \eqref{eq:limhydro}}.
\end{equation}
Indeed, for any such $\bar\mu^*$, $t$ and $k$, it is an immediate consequence of Assumption \ref{hypo:Ho} that 
the initial measure $\bar\mu^*_0$ admits a moment of order $k$. But the mapping 
$t\mapsto \bar{\mathfrak{L}}_{\chi^{k}}(\bar\mu^*_t)$ is the derivative of the $k$-th moment of $\bar\mu^*_t$, and is thus well defined and negative. 
So the mapping $t\longmapsto \cro{\bar\mu^*_t ,\chi^{k}}$ is non-increasing, proving \eqref{eq:momentboundsol}. 

We now let $\xi$ and $\zeta$ be two solutions of \eqref{eq:limhydro}. We show that $\xi_t = \zeta_t$ for all $t\in [0,1]$. 
First, in the obvious case where $\cro{\xi_0, \mathbb 1_{\N_+}} = \cro{\zeta_0, \mathbb 1_{\N_+}} = 0 $, we readily get that 
$\bar{\mathfrak{L}}_f(\xi_0) = \bar{\mathfrak{L}}_f(\zeta_0)$, and it follows that the solutions are both constant, namely $\xi_t=\zeta_t = \nu$ for all $t\in[0,1]$. 
Else, as both mappings $t\longmapsto\cro{\xi_t, \mathbb 1_{\N_+}}$ and $t\longmapsto\cro{\zeta_t, \mathbb 1_{\N_+}}$ are non-increasing, 
we can define for all $\beta\ge 0$, 
\[T^\xi_\beta= \sup \left\{t>0\,:\, \cro{\xi_t, \mathbb 1_{\N_+}}>\beta\right\},\quad T^\zeta_\beta= \sup \left\{t>0 \,:\, \cro{\zeta_t, \mathbb 1_{\N_+}}>\beta\right\}\]
and 
\[T_\beta = T^\xi_\beta \wedge T^\zeta_\beta.\]
We first fix $\beta>0$, and show that $\xi$ and $\zeta$ coincide on $[0,T_\beta]$.
For this, let us set 
$$\kappa_t = \xi_t - \zeta_t,\quad t\in[0,T_\beta],$$ 
and observe that, from \eqref{eq:momentboundsol}, $\kappa_t$ admits a moment of order $3.5 +\varepsilon$ for any $t\in[0,1]$. 
Now let $\alpha = 5+\varepsilon$, and define 
$$\Gamma_t = \sum_{i>0} i^\alpha \kappa_t(i)^2,\quad t\in[0,T_\beta]. $$
Observe that the above is finite for all $t$ since, from Cauchy-Schwarz inequality,
\[\Gamma_t \leq \left( \sum_{i>0} i^{\frac{\alpha}{2}} |\kappa_t(i)| \right)^2<\infty.\]
We will show that $\Gamma_t=0$ for all $t\in[0,T_\beta]$. 
For this, as $\Gamma_0 = 0$ by assumption, it is enough to show that for some constant $C_\beta$, 
\be
\label{majGammaDeriv}
{\d\over \d t}(\Gamma_t) \leq C_\beta\Gamma_t,\quad 0\le t \le T_\beta,
\ee
since we will then have, for all such $t$, 
${\d\over \d t}(\exp(-C_\beta t)\Gamma_t)\leq 0$ and thus $\exp(-C_\beta t)\Gamma_t \leq \Gamma_0 =0$. 
Differentiating $\Gamma$ shows that for all $t \le T_\beta$,
\begin{equation}
\label{serieGammaDer}
{\d\over \d t}\Gamma_t  = 2 \sum_{i>0} i^\alpha \kappa_t(i) {\d\over \d t}\kappa_t(i),
\end{equation}
and we are rendered to upper-bound the above expression. For this, to simplify the notation, in the computations below we omit the dependence in 
$t \leq T_\beta$, and we denote by $f'$, the derivative w.r.t. $t$ of a function $t\mapsto f(t)$. 
Also, for any measure $\mu$ we introduce the following quantities: 
\[\begin{cases}
	m_\mu &= \cro{\mu,\mathbb 1_{\N_+}};\\
	M_\mu &= \cro{\mu,\chi};\\
	V_\mu &= \cro{\mu,\chi^2}.
\end{cases}\]
Fix $i\in\N_+$. From \eqref{eq:limhydro} we have that 
\begin{equation*}
\kappa(i)'  =  \cro{\kappa, \mathbb{1}_{i}}' = \cro{\xi, \mathbb{1}_{i}}' -  \cro{\zeta, \mathbb{1}_{i}}'  
 = \bar{\mathfrak{L}}_{\mathbb{1}_{i}}(\xi) - \bar{\mathfrak{L}}_{\mathbb{1}_{i}}(\zeta). 
\end{equation*}
Then \eqref{eq:defLbarGRE} implies that 
\begin{eqnarray*}
\bar{\mathfrak{L}}_{\mathbb{1}_{i}}(\xi)  & = & -\Biggl({\cro{\xi, \mathbb{1}_{i} \mathbb 1_{\N_+}} \over \cro{\xi,\mathbb 1_{\N_+}}} +{\cro{\xi,\chi \mathbb{1}_{i}}\over\cro{\xi,\chi}}
+ {\cro{\xi,\chi\nabla \mathbb{1}_{i}} \over \cro{\xi,\chi}}\left({\cro{\xi,\chi- \mathbb 1_{\N_+}}\over \cro{\xi,\mathbb 1_{\N_+}}}
+ {\cro{\xi,\chi^2 - \chi} \over \cro{\xi,\chi}}\right) \Biggl) \\
& = &  -\Biggl({\xi(i)\mathbb 1_{\N_+}(i) \over m_\xi} +{i \xi(i) \over M_\xi}
+ {i \xi(i) - (i+1)\xi(i+1) \over M_\xi }\left({M_\xi - m_\xi \over m_\xi}
+ {V_\xi - M_\xi \over M_\xi }\right) \Biggl) \\
& = &  - \frac{\xi(i) }{m_\xi} - \frac{i \xi(i)}{M_\xi} - i \xi(i)  \frac{M_\xi - m_\xi }{M_\xi m_\xi}  - i \xi(i)  \frac{V_\xi - M_\xi }{M_\xi M_\xi} \\
& & \qquad + (i+1)\xi(i+1)  \frac{M_\xi - m_\xi }{M_\xi m_\xi} + (i+1)\xi(i+1) \frac{V_\xi - M_\xi }{M_\xi M_\xi}\\
& = &  - \frac{\xi(i) }{m_\xi} - \frac{i \xi(i)}{m_\xi}  - i \xi(i)  \frac{V_\xi - M_\xi }{M_\xi M_\xi} \\
& & \qquad + (i+1)\xi(i+1)  \frac{M_\xi - m_\xi }{M_\xi m_\xi} + (i+1)\xi(i+1) \frac{V_\xi - M_\xi }{M_\xi M_\xi} ,
\end{eqnarray*}
with a similar exression for  $\bar{\mathfrak{L}}_{\mathbb{1}_{i}}(\zeta)$, in a way that 
\begin{eqnarray*}
\kappa(i)' & = &  \left( \frac{\zeta(i)}{m_\zeta}  - \frac{\xi(i)}{m_\xi} \right) 
+ i \left( \frac{\zeta(i)}{m_\zeta} - \frac{\xi(i)}{m_\xi} \right)
+ i \left(  \zeta(i)  \frac{V_\zeta - M_\zeta}{M_\zeta M_\zeta} -   \xi(i)  \frac{V_\xi - M_\xi }{M_\xi M_\xi}  \right) \\
& & \quad + (i+1) \left( \xi(i+1)  \frac{M_\xi - m_\xi }{M_\xi m_\xi} - \zeta(i+1)  \frac{M_\zeta - m_\zeta}{M_\zeta m_\zeta} \right) \\
& & \quad + (i+1) \left( \xi(i+1) \frac{V_\xi - M_\xi }{M_\xi M_\xi} -  \zeta(i+1) \frac{V_\zeta - M_\zeta }{M_\zeta M_\zeta}  \right).
\end{eqnarray*}
The five terms of the r.h.s. of the above expression have a similar structure. We show how to simplify the first and the last one, the three other ones can be treated similarly. We have 
\begin{equation*}
 \frac{\zeta(i)}{m_\zeta}  - \frac{\xi(i)}{m_\xi}  = \frac{\zeta(i) -\xi(i)}{m_\zeta} + \left( \frac{1}{m_\zeta} - \frac{1}{m_\xi} \right)   \xi(i) 
   = - \frac{\kappa(i)}{m_\zeta} + \frac{m_\kappa}{m_\xi m_\zeta} \xi(i) 
\end{equation*}
and 
\begin{eqnarray*}
 \xi(i+1) \frac{V_\xi - M_\xi }{M_\xi M_\xi} -  \zeta(i+1) \frac{V_\zeta - M_\zeta }{M_\zeta M_\zeta}  & = & \xi(i+1) \left( \frac{V_\xi - M_\xi }{M_\xi^2} - \frac{V_\zeta - M_\zeta }{M_\zeta^2}  \right) + (\xi(i+1) -  \zeta(i+1) )\frac{V_\zeta - M_\zeta }{M_\zeta^2} \\
 & = & \xi(i+1)  \frac{V_\xi M_\zeta^2 - V_\zeta M_\xi^2 + M_\zeta M_\xi^2 - M_\xi M_\zeta^2 }{M_\xi^2 M_\zeta^2}  + \kappa(i+1) \frac{V_\zeta - M_\zeta }{M_\zeta^2} \\
 & = &  \xi(i+1)  \frac{V_\kappa M_\zeta^2 - V_\zeta (M_\xi^2 -M_\zeta^2) + M_\xi M_\zeta M_\kappa }{M_\xi^2 M_\zeta^2}  + \kappa(i+1) \frac{V_\zeta - M_\zeta }{M_\zeta^2} \\
 & = &  \xi(i+1)  \frac{V_\kappa M_\zeta^2 + M_\kappa( M_\xi M_\zeta - V_\zeta (M_\xi +M_\zeta))}{M_\xi^2 M_\zeta^2}  + \kappa(i+1) \frac{V_\zeta - M_\zeta }{M_\zeta^2}\cdot
\end{eqnarray*}
All in all, we obtain that 
\begin{eqnarray}
\kappa(i)' & = & \left( - \frac{\kappa(i)}{m_\zeta} + \frac{m_\kappa}{m_\xi m_\zeta}  \xi(i)  \right) + i \left(  - \frac{\kappa(i)}{m_\zeta} +  \frac{m_\kappa}{m_\xi m_\zeta} \xi(i)   \right) \nonumber \\
& & \quad + i \left( - \frac{V_\zeta - M_\zeta}{M_\zeta^2} \kappa(i)  -\frac{M_\zeta^2 V_\kappa + (M_\xi M_\zeta -V_\zeta (M_\xi+M_\zeta))M_\kappa}{M_\xi^2 M_\zeta^2} \xi(i)  \right) \nonumber \\
& & \quad + (i+1) \left( \frac{M_\zeta -m_\zeta}{M_\zeta m_\zeta} \kappa(i+1) + \frac{-M_\xi M_\zeta m_\kappa + m_\xi m_\zeta M_\kappa}{M_\xi m_\xi M_\zeta m_\zeta} \xi(i+1) \right) \nonumber \\
& & \quad + (i+1) \left(  \frac{V_\zeta - M_\zeta }{M_\zeta^2} \kappa(i+1)  +   \frac{M_\zeta^2 V_\kappa + ( M_\xi M_\zeta - V_\zeta (M_\xi +M_\zeta)) M_\kappa}{M_\xi^2 M_\zeta^2} \xi(i+1)  \right),\nonumber
\end{eqnarray}
in a way that the general term of the series in the r.h.s. of \eqref{serieGammaDer} reads 
\begin{eqnarray}
i^\alpha \kappa(i) \kappa(i)'&=& i^\alpha\left( - \frac{\kappa(i)^2}{m_\zeta} + \frac{m_\kappa}{m_\xi m_\zeta}  \xi(i) \kappa(i) \right) + i^{\alpha+1} \left(  - \frac{\kappa(i)^2}{m_\zeta} +  \frac{m_\kappa}{m_\xi m_\zeta} \xi(i)  \kappa(i) \right) \nonumber \\
& & \quad + i^{\alpha+1}  \left( - \frac{V_\zeta - M_\zeta}{M_\zeta^2} \kappa(i)^2  -\frac{M_\zeta^2 V_\kappa + (M_\xi M_\zeta -V_\zeta (M_\xi+M_\zeta))M_\kappa}{M_\xi^2 M_\zeta^2} \xi(i)  \kappa(i) \right) \nonumber \\
& & \quad + i^\alpha (i+1) \kappa(i) \left( \frac{M_\zeta -m_\zeta}{M_\zeta m_\zeta} \kappa(i+1) + \frac{-M_\xi M_\zeta m_\kappa + m_\xi m_\zeta M_\kappa}{M_\xi m_\xi M_\zeta m_\zeta} \xi(i+1) \right) \nonumber \\
& & \quad + i^\alpha (i+1) \kappa(i) \left(  \frac{V_\zeta - M_\zeta }{M_\zeta^2} \kappa(i+1)  +   \frac{M_\zeta^2 V_\kappa + ( M_\xi M_\zeta - V_\zeta (M_\xi +M_\zeta)) M_\kappa}{M_\xi^2 M_\zeta^2} \xi(i+1)  \right). \nonumber
\end{eqnarray}
Therefore we obtain that 
\begin{eqnarray}
i^\alpha \kappa(i) \kappa(i)'&=& a(i) + i^\alpha\left( - \frac{\kappa(i)^2}{m_\zeta} + \frac{m_\kappa}{m_\xi m_\zeta}  \xi(i) \kappa(i) \right) \nonumber \\
& & \quad + i^{\alpha+1}   \frac{m_\kappa}{m_\xi m_\zeta} \xi(i)  \kappa(i)  \nonumber \\
& & \quad + i^{\alpha+1}  \left(  -\frac{M_\zeta^2 V_\kappa + (M_\xi M_\zeta -V_\zeta (M_\xi+M_\zeta))M_\kappa}{M_\xi^2 M_\zeta^2} \xi(i)  \kappa(i) \right) \nonumber \\
& & \quad + i^\alpha \kappa(i)  \frac{M_\zeta -m_\zeta}{M_\zeta m_\zeta} \kappa(i+1) +  i^\alpha (i+1) \kappa(i) \frac{-M_\xi M_\zeta m_\kappa + m_\xi m_\zeta M_\kappa}{M_\xi m_\xi M_\zeta m_\zeta} \xi(i+1)  \nonumber \\
& & \quad + i^\alpha  \kappa(i)   \frac{V_\zeta - M_\zeta }{M_\zeta^2} \kappa(i+1)  +   i^\alpha (i+1) \kappa(i) \frac{M_\zeta^2 V_\kappa + ( M_\xi M_\zeta - V_\zeta (M_\xi +M_\zeta)) M_\kappa}{M_\xi^2 M_\zeta^2} \xi(i+1)  \nonumber \\
& = &  a(i) + i^\alpha\left( - \frac{\kappa(i)^2}{m_\zeta} + \left( \frac{M_\zeta -m_\zeta}{M_\zeta m_\zeta}  + \frac{V_\zeta - M_\zeta }{M_\zeta^2} \right) \kappa(i) \kappa(i+1) \right) \nonumber \\
& & \quad +   \frac{m_\kappa}{m_\xi m_\zeta} i^\alpha \kappa(i)  \xi(i) \nonumber \\
& & \quad  + \left( \frac{-m_\kappa}{m_\xi m_\zeta} +   \frac{M_\zeta^2 V_\kappa + (2 M_\xi M_\zeta - V_\zeta (M_\xi +M_\zeta)) M_\kappa}{M_\xi^2 M_\zeta^2}  \right)  i^\alpha \kappa(i)  \xi(i+1)   \nonumber \\
& & \quad -\frac{M_\zeta^2 V_\kappa + (M_\xi M_\zeta -V_\zeta (M_\xi+M_\zeta))M_\kappa}{M_\xi^2 M_\zeta^2}  i^{\alpha+1} \kappa(i) \xi(i)   \nonumber \\
& & \quad + \left( \frac{-m_\kappa}{m_\xi m_\zeta} +   \frac{M_\zeta^2 V_\kappa + (2 M_\xi M_\zeta - V_\zeta (M_\xi +M_\zeta)) M_\kappa}{M_\xi^2 M_\zeta^2}  \right)  i^{\alpha+1} \kappa(i)  \xi(i+1),
\label{termegeneder}
\end{eqnarray}
with 
\begin{eqnarray}
a(i) &=& i^{\alpha+1} \left(  - \frac{\kappa(i)^2}{m_\zeta}   - \frac{V_\zeta - M_\zeta}{M_\zeta^2} \kappa(i)^2  + \frac{M_\zeta -m_\zeta}{M_\zeta m_\zeta} \kappa(i) \kappa(i+1) +  \frac{V_\zeta - M_\zeta }{M_\zeta^2} \kappa(i)\kappa(i+1)   \right) \nonumber \\
& = & \frac{i^{\alpha+1}}{M_\zeta^2 m_\zeta} \bigg( \underbrace{( M_\zeta^2 + V_\zeta m_\zeta - 2M_\zeta m_\zeta )}_{ \geq 0} \underbrace{\kappa(i)\kappa(i+1)}_{\leq \frac{\kappa(i)^2 + \kappa(i+1)^2}{2}} - \underbrace{( M_\zeta^2 + V_\zeta m_\zeta - M_\zeta m_\zeta )}_{\geq M_\zeta^2 + V_\zeta m_\zeta - 2M_\zeta m_\zeta } \kappa(i)^2 \bigg) \nonumber \\
& \leq &  \frac{M_\zeta^2 + V_\zeta m_\zeta - 2M_\zeta m_\zeta }{2 M_\zeta^2 m_\zeta} i^{\alpha+1} (\kappa(i+1)^2 - \kappa(i)^2).\nonumber
\end{eqnarray}
We address one by one the terms of \eqref{termegeneder}. For this, first note that 
all linear combinations of $m_\kappa$, $M_\kappa$, $V_\kappa$ can be 
easily upper-bounded, by observing that, as $\alpha>5$, Cauchy-Schwarz inequality implies 
\begin{equation*}
|V_\kappa|  \leq  \sum_i i^2 |\kappa(i)| = \sum_i i^{2-\frac{\alpha}{2}} i^{\frac{\alpha}{2}} |\kappa(i)| 
\leq \left(\sum_i i^{4-\alpha} \right) ^{\frac{1}{2}} \Gamma^{\frac{1}{2}} \leq  C \Gamma^{\frac{1}{2}},
\end{equation*}
for some $C$ that is independent of $t$. We now control the series of general term $a(i)$. For this, using Abel's transformation, 
for all $N>0$ we obtain  
\begin{eqnarray*}
\sum_{i=1}^N i^{\alpha+1} (\kappa(i+1)^2 - \kappa(i)^2) & = & \sum_{i=0}^N i^{\alpha+1} (\kappa(i+1)^2 - \kappa(i)^2) \\
& = &  \sum_{i=0}^N \left( (i+1)^{\alpha+1} \kappa(i+1)^2 - i^{\alpha+1} \kappa(i)^2 \right) - \left( (i+1)^{\alpha+1} - i^{\alpha+1} \right)\kappa(i+1)^2 \\
& = & (N+1)^{\alpha+1} \kappa(N+1)^2 - \sum_{i=0}^N \underbrace{\left( (i+1)^{\alpha+1} - i^{\alpha+1} \right)}_{\leq (\alpha+1) (i+1)^\alpha}\kappa(i+1)^2.
\end{eqnarray*}
But as $\kappa$ has a finite moment of order $\frac{\alpha + 1}{2} = 3 + \frac{\varepsilon}{2}$, we have that 
\[(N+1)^{\alpha+1} \kappa(N+1)^2 \xrightarrow{N} 0,\]
implying the existence of a constant $C'$ that does not depend on $t$, and such that 
\begin{equation*}
\left| \sum_{i>0} a(i) \right| \leq C' \Gamma. 
\end{equation*}
The other terms of (\ref{termegeneder}) are of two main types: 
\begin{enumerate}
\item[(i)] The terms involving the factors $i^\alpha \kappa(i)^2$ and $i^\alpha \kappa(i)\kappa(i+1)$ can be upper-bounded by 
$$|\kappa(i)\kappa(i+1)|\leq \frac{1}{2} \kappa(i)^2 + \frac{1}{2} \kappa(i+1)^2,$$ 
whose series is readily upper-bounded by somme $C'' \Gamma$; 
\medskip 
\item[(ii)] The other 
terms involve the quantities $i^\alpha \kappa(i)\xi(i)$, $i^\alpha \kappa(i)\xi(i+1)$, $i^{\alpha+1} \kappa(i)\xi(i)$ and $i^{\alpha+1} \kappa(i)\xi(i+1)$.
To upper-bound the corresponding series, we again use Cauchy-Schwarz inequality. We only detail the development of 
the term in $i^{\alpha+1} \kappa(i)\xi(i+1)$, which necessitates the stronger moment assumption on $\xi$. The other terms can be treated similarly. 
We have 
\begin{eqnarray*}
\sum_{i\geq 1}  i^{\alpha+1} |\kappa(i) \xi(i+1)| & = & \sum_{i\geq 1} i^{\frac{\alpha}{2}+1} |\xi(i+1)| i^{\frac{\alpha}{2}}  |\kappa(i)|  \\
& \leq & \sum_{i\geq 1} (i+1)^{\frac{\alpha}{2}+1} |\xi(i+1)| i^{\frac{\alpha}{2}}  |\kappa(i)| \\
& \leq & \left( \sum_{i\geq 1} (i+1)^{\alpha+2} \xi(i+1)^2 \right)^{\frac{1}{2}}  \Gamma^{\frac{1}{2}} ,\\
& \leq &C'''\Gamma^{\frac{1}{2}} ,
\end{eqnarray*}
as $\xi$ has a finite moment of order $\frac{\alpha+2}{2}  = \frac{7 + \varepsilon}{2}\cdot $ 
\end{enumerate}
Injecting all these series bounds into \eqref{termegeneder}, we conclude that there exists a constant $C_\beta>0$ that does not depend on 
$t\in [0,T_\beta]$, and such that \eqref{majGammaDeriv} holds. 

\bigskip 

Now, \eqref{majGammaDeriv} implies that $\Gamma_t=0$ for all $t\in[0,T_\beta]$, and in turn, that 
\begin{equation}
\label{eq:fff1}
\cro{\kappa_t,\mathbb 1_{\N_+}}=0,\mbox{ for all }t\in [0,T_\beta].
\end{equation} 
But we have 
\[\cro{\xi_0,\mathbb 1_{\N}}-\cro{\zeta_0,\mathbb 1_{\N}}=\cro{\nu,\mathbb 1_{\N}}-\cro{\nu,\mathbb 1_{\N}}=0\]
and, from \eqref{eq:limhydro} and \eqref{eq:defLbarGRE},
\[{\d\over \d t}\cro{\xi_t, \mathbb{1}_{\N}}={\d\over \d t}\cro{\zeta_t, \mathbb{1}_{\N}}=-2,\,t\in[0,T_\beta],\]
implying that 
\begin{equation*}
\cro{\xi_t,\mathbb{1}_\N}= \cro{\zeta_t,\mathbb{1}_\N},\mbox{ for all }t\in [0,T_\beta].
\end{equation*}
This, together with \eqref{eq:fff1}, shows that $\kappa_t(0)=0$ for all $t\in [0,T_\beta]$ and thus, using again \eqref{eq:fff1}, that 
\begin{equation*}
\kappa_t= \mathbf 0,\mbox{ for all }t\in [0,T_\beta].
\end{equation*}
In particular, we get that $T^\xi_\beta=T^\zeta_\beta=T_\beta$. 
But the mapping $\Pi_{\mathbb 1_{\N_+}}$ is continuous on $\bar\M$, implying in turn that the mapping $t\mapsto \Pi_{\mathbb 1_{\N_+}}(\mu_t)$ is continuous. 
As the above holds for any $\beta>0$, taking $\beta$ to zero we obtain in turn that $T^\xi_0=T^\zeta_0=T_0$. By the very definition \eqref{eq:defLbarGRE} of $\bar{\mathfrak L}$, 
the uniqueness of a solution to \eqref{eq:limhydro} is then trivially extended to $[0,1]$. This concludes the proof.
\end{proof}

%

\section{Proof of Theorem \ref{thm:matchcovunimin}} 
\label{sec:unimin}
We now come back to the matching criterion {\sc uni-min}, introduced in Section \ref{subsec:mainresalgo} and formalized in Example \ref{ex:unimin}. 
Fix $N\in\N_+$, 
and let us assume that at all times, the graphs have degree bounded by $N$, implying that for all $n$, $\mu_0^n$ has support in $\llbracket 0,N \rrbracket$. 
In this slightly restricted case, we can show that Theorem \ref{thm:main} applies also 
for $\bphi=\textsc{uni-min}.$ 
Indeed, in this case, for any $j\in\llbracket 0,n \rrbracket$ we readily obtain that 
$$\mu^n_j := \sum_{i=0}^N \mu_j^n(i)\delta_i,$$
and it follows that the measure $\mu^n_j$ can be identified with the vector $(\mu^n_j(0),...,\mu^n_j(N))$ on $[0,1]^{N+1} \subset \R^{N+1}$.

Fix $\mu\in\M$ such that $\cro{\mu,\mathbb 1_{\N_+}}>0$, and recall again the definition \eqref{eq:distribK}. 
Then, we again obtain that for all $f\in\maC_b$, 
\begin{equation*}
\cro{\maK_{\textsc{uni-min}}(\mu),f}  = {\cro{\mu,f\mathbb 1_{\N_+}}\over \cro{\mu,\mathbb 1_{\N_+}}}\cdot
\end{equation*}
Also, for any $k\in \N_+$, recalling \eqref{eq:distribK'}, we get that for all $k'\in\N_+$, 
\begin{align*}
\maK'_{\textsc{uni-min}}(\mu,k)(k') &=\maK'_{\textsc{uni-min}}(\mu,k)\left([k',\infty)\right)-\maK'_{\textsc{uni-min}}(\mu,k)\left([k'+1,\infty)\right).
\end{align*}
Now, recall that the outcome of a sequence of independent uniform draws without replacement, equals in law the outcome of a sequence of independent uniform draws with replacement, conditioned on not drawing twice the same element. Then, by the very definition of a matching criterion, by comparing the constructions of Section \ref{subsec:localCM} and \ref{sec:constructionhat}, we obtain that for all $\mu$ and $k$, there exists a mapping $\varphi$ that depend only 
on $\bphi$ and possibly on an independent draw, and such that $\wh K'=\varphi(\wh H_1,\cdots,\wh H_{\wh K})$. 
Therefore, for all $\mu\in\M^n_{\beta,M}$ and all $k$ and $k'$, 
conditioned on that draw of $\varphi$ and on $\{\wh K=k\}$ and $\{K=k\}$, we get that 
\begin{align}
\maK'_{\textsc{uni-min}}(\mu,k)(k')&=\sum_{\substack{\mbox{\tiny{distinct}} (h_1,\cdots,h_k):\\\varphi(h_1,\cdots,h_k)=k'}}\prm{\left(H_1,\cdots,H_{K}\right)=(h_1,\cdots,h_k)}\notag\\
&=\sum_{\substack{\mbox{\tiny{distinct}} (h_1,\cdots,h_k):\\\varphi(h_1,\cdots,h_k)=k'}}\prm{\left(\wh H_1,\cdots,\wh H_{\wh K}\right)=(h_1,\cdots,h_k) \mid (\wh H_1,\cdots,\wh H_{\wh K})\mbox{ \small{are distinct}}}\notag\\
&=\sum_{\substack{\mbox{\tiny{distinct}} (h_1,\cdots,h_k):\\\varphi(h_1,\cdots,h_k)=k'}}{\prm{\left(\wh H_1,\cdots,\wh H_{\wh K}\right)=(h_1,\cdots,h_k)}\over \prm{(\wh H_1,\cdots,\wh H_{\wh K})\mbox{ \small{are distinct}}}}\notag\\
&={1 \over \prm{(\wh H_1,\cdots,\wh H_{\wh K})\mbox{ \small{are distinct}}}}\sum_{\substack{\mbox{\tiny{distinct}} (h_1,\cdots,h_k):\\\varphi(h_1,\cdots,h_k)=k'}}\prm{\left(\wh H_1,\cdots,\wh H_{\wh K}\right)=(h_1,\cdots,h_k)}\notag\\
&={\wh \maK'_{\textsc{uni-min}}(\mu,k)(k') \over \prm{\left(\wh H_1,\cdots,\wh H_{\wh K}\right)\mbox{ \small{are distinct}}}}\cdot\label{eq:WBunimin}
\end{align}
But as the draws of the construction of Section \ref{sec:constructionhat} are made with replacement, we have that
\begin{align*}
\prm{\left(\wh H_1,\cdots,\wh H_{\wh K}\right)\mbox{ \small{are distinct}}}
&={\cro{\mu,\mathbb 1_{\N_+}}! \over (\cro{\mu,\mathbb 1_{\N_+}} - k)!\cro{\mu,\mathbb 1_{\N_+}}^k}\\
&\ge {\cro{\mu,\mathbb 1_{\N_+}}! \over (\cro{\mu,\mathbb 1_{\N_+}} - N)!\cro{\mu,\mathbb 1_{\N_+}}^N}\cdot
\end{align*}
As $\cro{\mu,\mathbb 1_{\N_+}} \ge n\beta$, the above probability clearly tends to 1 as $n$ goes large. 
Therefore from \eqref{eq:WBunimin}, we obtain that for all 
$k$, 
\begin{align*}
\left|\cro{\maK'_{\textsc{uni-min}}(\mu,k)-\wh \maK'_{\textsc{uni-min}}(\mu,k),\chi^2}\right|
&\le \sum_{k'=1}^N (k')^2 \left|\maK'_{\textsc{uni-min}}(\mu,k)(k')-\wh \maK'_{\textsc{uni-min}}(\mu,k)(k')\right|\\
&\le N^2 \left|\maK'_{\textsc{uni-min}}(\mu,k)(k')-\wh \maK'_{\textsc{uni-min}}(\mu,k)(k')\right|
\end{align*}
can be made as small as desired for a large enough $n$. This implies that the criterion $\bphi$ is well-behaved, in the sense of Definition \ref{def:polWB}.

Now, from the very construction of Section \ref{sec:constructionhat}, the degree of each drawn bucket is size-biased, independently of everything else, 
and so by the very definition 
of {\sc uni-min}, for any $k,y\in\N_+$ we get that 
\begin{equation}
\label{eq:K'unimin}
\wh\maK'_{\textsc{uni-min}}(\mu,k)\left([y,\infty)\right)= \left(\sum_{j=y}^\infty {j\mu(j)\over \cro{\mu,\chi}}\right)^k=\bar F_\mu(y-1)^k,
\end{equation}
recalling \eqref{eq:defbarF}. We deduce that for all $k\in \N_+$ and $k'\in\N_+$, 
\begin{equation}
\label{eq:Phiunimin2}\maK'_{\textsc{uni-min}}(\mu,k)(k') = \bar F_\mu(k'-1)^k-\bar F_\mu(k')^k.
\end{equation}
We then immediately check that {\sc uni-min} is continuous in the sense of Definition \ref{def:polCont}: fix a sequence 
$\suite{\bar\mu^n}$ of $\bar\M$ such that $\bar\mu^n\in \bar\M^n$ for all $n$, and such that $\bar\mu^n\xRightarrow{n}\bar\mu$ for some 
measure $\bar\mu \in\bar\M$ supported in $\llbracket 0,N \rrbracket$, and such that $\cro{\bar\mu,\mathbb 1_{\N_+}}>0$. 
As for $\bphi=$ {\sc greedy}, for all $f\in\maC_b$ we get 
\[\cro{\maK_{\textsc{uni-min}}(n\bar\mu^n),f}={\cro{n\bar\mu^n,f\mathbb 1_{\N_+}}\over \cro{n\bar\mu^n,\mathbb 1_{\N_+}}}
={\cro{\bar\mu^n,f\mathbb 1_{\N_+}}\over \cro{\bar\mu^n,\mathbb 1_{\N_+}}}\xlongrightarrow{n} {\cro{\bar\mu,f\mathbb 1_{\N_+}}\over \cro{\bar\mu,\mathbb 1_{\N_+}}}=:\cro{\bar{\maK}_{\textsc{uni-min}}(\bar\mu),f}.\]
Likewise, for all $k\in\N_+$, for all $f\in\maC_b$, \eqref{eq:Phiunimin2} entails that 
\begin{align*}
\cro{\wh{\maK}'_{\textsc{uni-min}}(n\bar\mu^n,k),f}&=\sum_{k'\in\N_+}f(k')\left\{\bar F_{n\bar\mu^n}(k'-1)^k-\bar F_{n\bar\mu^n}(k')^k\right\}\\
&=\sum_{k'=1}^N f(k')\left\{\bar F_{\bar\mu^n}(k'-1)^k-\bar F_{\bar\mu^n}(k')^k\right\}\\
 &\xlongrightarrow{n} \sum_{k'=1}^Nf(k')\left\{\bar F_{\bar\mu}(k'-1)^k-\bar F_{\bar\mu}(k')^k\right\}=:\cro{\bar{\maK}'_{\textsc{uni-min}}(\bar\mu,k),f},
 \end{align*}
 showing the continuity of the {\sc uni-min} criterion. Moreover, plugging the above 
 into \eqref{eq:defLbar}, we can again make the operator $\bar{\mathfrak{L}}$ explicit in the present case: For any measure $\bar\mu\in \bar{\M}$ that is 
 supported in $\llbracket 0,N \rrbracket$ and such that $\cro{\bar\mu,\mathbb 1_{\N_+}}>0$, for any $f\in\maC_b$, 
\begin{multline}
\label{eq:defLbarMIN}
\bar{\mathfrak{L}}_f(\bar\mu)
=- \cro{\bar\maK_{\textsc{uni-min}}(\bar\mu),f+\cro{\bar{\maK}'_{\textsc{uni-min}}(\bar\mu,.),f}}\\
\shoveright{-{\cro{\bar\mu,\chi\nabla f}\over \cro{\bar\mu,\chi}}
\Bigl\{\cro{\bar\maK_{\textsc{uni-min}}(\bar\mu),\chi-\mathbb 1+\cro{\bar\maK'_{\textsc{uni-min}}(\bar\mu,.),\chi-\mathbb 1}}\Bigl\}}\\
\shoveleft{=-\Biggl\{{\cro{\bar\mu,f\mathbb 1_{\N_+}}\over \cro{\bar\mu,\mathbb 1_{\N_+}}} + {\sum_{k\in\N_+}\bar\mu(k)\sum_{k'\in\N_+}f(k')
\left\{\bar F_{\bar\mu}(k'-1)^k-\bar F_{\bar\mu}(k')^k\right\}\over\cro{\bar\mu,\mathbb 1_{\N_+}}}}\\
\left.+ {\cro{\bar\mu,\chi\nabla f} \over \cro{\bar\mu,\chi}}\left\{{\cro{\bar\mu,(\chi- \mathbb 1_{\N_+})}\over \cro{\bar\mu,\mathbb 1_{\N_+}}}
+{\sum_{k\in\N_+}\bar\mu(k)\sum_{k'\in\N_+}\left(k'-1\right)\left\{\bar F_{\bar\mu}(k'-1)^k-\bar F_{\bar\mu}(k')^k\right\}\over\cro{\bar\mu,\mathbb 1_{\N_+}}}\right\} \right\}.
\end{multline}

\noindent We have the following result, which implies Theorem \ref{thm:matchcovunimin} in view of Corollary \ref{cor:maincov}. 
\begin{corollary}[Convergence for the \textsc{uni-min} criterion]
\label{cor:MainMIN}
If $\mathbf\Phi=\textsc{uni-min}$, if the sequence of initial degree distributions $\suite{\mu^n_0}$ are supported in $\llbracket 0,N \rrbracket$ for some 
$N\in \N_+$ and satisfy Assumption \ref{hypo:Ho} for some measure $\nu$, then for all $f\in \maC_b$ we get the convergence 
\begin{equation*}
\sup_{t \leq 1 } |\cro{\bar{\mu}^n_t,f} -  \cro{\bar{\mu}^{\textsc{uni-min}}_t,f}| \xrightarrow{(n,\P)} 0,
\end{equation*}
where $\procun{\bar{\mu}_t^{\textsc{uni-min}}}$ is the unique solution of \eqref{eq:limhydro}, for $\bar{\mathfrak{L}}$ defined, for all $f\in\maC_b$ and all measures 
$\bar\mu\in\bar\M$ supported in $\llbracket 0,N \rrbracket$,  by \eqref{eq:defLbarMIN} if $\cro{\bar\mu,\mathbf 1_{\N_+}}>0$, and by 
$\bar{\mathfrak{L}}_f(\bar\mu) = 0$ if $\cro{\bar\mu,\mathbf 1_{\N_+}}=0$. 
\end{corollary} 
\begin{proof}
{We apply again Theorem \ref{thm:main}. As the initial measures $\suite{\mu^n_0}$ are supported in 
$\llbracket 0,N \rrbracket$ and satisfy Assumption \ref{hypo:Ho}, the limiting initial measure $\nu$ is, clearly, also supported in $\llbracket 0,N \rrbracket$. 
This is then also the case for any $\bar\mu^*_t$, $t\in[0,1]$,  for $\bar\mu^*$ a solution to \eqref{eq:limhydro}. 
Second, we have just shown that the {\sc uni-min} criterion is well-behaved and continuous. It is also immediate that the sequence of processes 
$\suite{\bar\mu^n}$ take value in the set of measures of mass less than 1 and supported in $\llbracket 0,N \rrbracket$, and so the moment preservation property holds, if one restricts for any $n$, to measures $\mu\in\M^n_{\beta,M}$ that are supported in $\llbracket 0,N \rrbracket$. 

 In view of Theorem \ref{thm:main}, it remains to show that the system of integral equations \eqref{eq:limhydro}, 
this time for $\bar{\mathfrak{L}}$ defined by \eqref{eq:defLbarMIN} for measures $\bar\mu$ such that $\cro{\bar\mu,\mathbb 1_{\N_+}}>0$, admits the only solution $\bar\mu^*$. For this, we adopt a similar approach to the proof of Corollary \ref{cor:MainGRE}, and keep the notation therein. 
Again, let us first observe that, for any $1 \le p \le 3$, 
$$\sup_{t \in [0,1]} \cro{\bar\mu_t, \chi^p}  < M,$$
and let again et $\xi$ and $\zeta$ be two solutions of \eqref{eq:limhydro} under \textsc{uni-min}. As in the proof of Corollary \ref{cor:MainGRE}, 
$\xi$ and $\zeta$ trivially coincide if 
$\cro{\xi_0, \mathbb 1_{\N_+}} = \cro{\zeta_0, \mathbb 1_{\N_+}} = 0$. 
Else, as above we fix $\beta>0$, and first show that $\xi$ and $\zeta$ coincide at least up to $T_\beta$.} 
In the case of bounded degrees, the proof of this fact becomes significantly easier. Indeed, notice that the uniqueness of $\xi$ is equivalent to that of the vector-valued process
$\left(X_t\,:\,t\in[0,T_\beta]\right)$, where 
\[X_t  :=   \begin{bmatrix}
           \xi_t(0) \\
           \xi_t(1) \\
           \vdots \\
           \xi_t(N)
     \end{bmatrix} \in \left[0,1\right]^{N+1},\quad t\in[0,T_\beta], 
 \]
and that we can rewrite the measure valued ODE associated to the algorithm as a $\R^{N+1}_+$-valued ODE :
\begin{equation}
\label{eq:systminres}
    \frac{\d}{\d t} X_t =
		F\left(X_t \right),\,t\in[0,T_\beta], \qquad \textrm{with } 
F\left(X_t \right) = 
		\begin{bmatrix}
           F_0\left(X_t\right) \\
           F_1\left(X_t\right) \\
           \vdots \\
           F_N\left(X_t\right)
     \end{bmatrix} = \begin{bmatrix}
           \bar{\mathfrak{L}}_{\mathbb{1}_{0}}(\xi_t) \\
           \bar{\mathfrak{L}}_{\mathbb{1}_{1}}(\xi_t)\\
           \vdots \\
           \bar{\mathfrak{L}}_{\mathbb{1}_{N}}(\xi_t)
     \end{bmatrix}  \in \R^{N+1}.
\end{equation}
By the Cauchy-Lipschitz theorem, we just have to prove that $F$ is $K$-Lipschitz Continuous on $\left[0,1\right]^{N+1}$, for some well-chosen norm 
and some $K>0$. For all $X = (X(0),...,X(N)) \in \left[0,1\right]^{N+1}$, it follows from \eqref{eq:defLbarMIN} that we have 
\begin{equation}
\label{eq:systminres0}
F_0(X) = -\frac{M_X - m_X}{m_X}\frac{X(1)}{M_X} - \frac{X(1)}{M_X}\sum_{l=1}^N(l-1)\sum_{k=1}^N \frac{{X(k)}}{m_X}\left\{Q_X(l)^k - Q_X(l+1)^k\right\},
\end{equation}
and for all $i\in \llbracket  1, N \rrbracket$,
\begin{multline}
F_i(X) = -\frac{X(i)}{m_X} -\frac{M_X - m_X}{m_X}\frac{iX(i) - (i+1) X(i+1)}{M_X} -\sum_{k=1}^N \frac{X(k)}{m_X}\left\{Q_X(i)^k - Q_X(i+1)^k\right\}\\
 - \frac{iX(i) - (i+1) X(i+1)}{M_X}\sum_{l=1}^N(l-1)\sum_{k=1}^N \frac{X(k)}{m_X}\left\{Q_X(l)^k - Q_X(l+1)^k\right\}, \label{eq:systminresi}
\end{multline}
where we have set $X(N+1)=0$ and defined
$$
m_X = \sum_{k=1}^N X(k), \quad M_X = \sum_{k=1}^N k X(k), \quad Q_X(i) = \frac{1}{M_X}\sum_{k=i}^N k X(k).
$$
The following intermediate Lemma gives us the desired result : 
\begin{lemma}
\label{lemma:Lip}
For any $\beta \in (0,1)$, on the set 
\begin{equation*}
G_{\beta}=\big\{ X \in \left[ 0, 1 \right]^{N+1}\,:\, m_X \geq \beta \big\} 
\end{equation*}
the mappings $$X \longmapsto m_X, \quad X \longmapsto \frac{1}{m_X}, \quad X \longmapsto M_X,\quad X \longmapsto \frac{1}{M_X}, \quad X \longmapsto Q_X(i)\quad \text{and} \quad X \longmapsto Q_X(i)^k$$ 
for $1\leq i \leq N$ and $1\leq k \leq N$, are all Lipschitz continuous and bounded from $\left[0,1\right]^{N+1}$ to $\R$, 
for the $1$-norm defined by $$\|X\|_1 = \sum_{k=0}^N |X(k)|,\quad X\in \left[0,1\right]^{N+1}.$$ 
\end{lemma}
\begin{proof}[Proof of Lemma \ref{lemma:Lip}]
Let $X, Y \in G_{\beta}$. 
We first have that $$|m_X - m_Y| \leq \sum_{k=1}^N |X(k) - Y(k)| \leq  \|X-Y\|_1. $$
Moreover, we have that  
$$ \beta \leq  m_X = \sum_{k=1}^N X(k) \leq N,$$
which also implies that  
$$ \frac{1}{m_X}  \leq \frac{1}{\beta},$$
and thus 
$$
\left| \frac{1}{m_X}  - \frac{1}{m_Y}  \right| = \left|  \frac{m_Y-m_X}{m_X m_Y}  \right|\leq \frac{1}{\beta^2} \|X-Y\|_1.
$$
Next, we have that 
\begin{align*}
|M_X - M_Y| &= \left|\sum_{k=1}^N k X(k) - k Y(k)\right|\\
						&\leq \sum_{k=1}^N k|X(k) - Y(k)|
						\leq \sum_{k=1}^N N|X(k) - Y(k)|= N \|X-Y\|_1
\end{align*}
and moreover 
$$
\beta \leq  m_X \leq \sum_{k=1}^N X(k) \leq \sum_{k=1}^N kX(k) =M_X \le \sum_{k=1}^N NX(k) \leq N^2,
$$
implying in particular that 
$$
\left|\frac{1}{M_X}  - \frac{1}{M_Y} \right| = \left|\frac{M_Y - M_X}{M_XM_Y} \right| 
								\leq \frac{N \|X-Y\|_1}{\beta^2}\cdot
$$
Finally, fix $1\le i \le N$. We have 
$$
0\leq Q_X(i) = \frac{1}{M_X}\sum_{k=i}^N k X(k) \leq \frac{1}{M_X}\sum_{k=1}^N k X(k) = 1
$$
and thus $0\leq Q_X(i)^k \leq 1$ for all $1\le k \le N$. 
Gathering the above, we obtain that 
\begin{align*}
\left|Q_X(i) - Q_Y(i)\right| &= \left|   \frac{1}{M_X}\sum_{k=i}^N k X(k) -  \frac{1}{M_Y}\sum_{k=i}^N k Y(k)    \right|\\
						&= \left|   \left(\frac{1}{M_X} -  \frac{1}{M_Y} \right) \sum_{k=i}^N k X(k) +  \frac{1}{M_Y} \left(\sum_{k=i}^N k X(k)  -\sum_{k=i}^N k Y(k) \right)   \right|\\
						&\leq \left| \frac{1}{M_X} -  \frac{1}{M_Y} \right| \cdot\left| \sum_{k=i}^N k X(k) \right| +  \frac{1}{M_Y} \left| \sum_{k=i}^N k X(k)  -\sum_{k=i}^N k Y(k)   \right| \\
						&\leq \frac{N^3}{\beta^2} \|X-Y\|_1 + \frac{N}{\beta} \|X-Y\|_1
						\leq \max\left(\frac{N^3}{\beta^2} , \frac{N}{\beta}  \right) \|X-Y\|_1. 
\end{align*}
Thus for all $1\le k \le N$ we obtain that 
\begin{align*}
\left|Q_X(i)^k - Q_Y(i)^k\right| &= \left| (Q_X(i) - Q_Y(i))  \sum_{l=0}^{k-1} Q_X(i)^{k-1-l} Q_Y(i)^l  \right|\\
						& \leq  \left|Q_X(i) - Q_Y(i)\right| \sum_{l=0}^{k-1}  \left| Q_X(i)^{k-1-l} Q_Y(i)^l \right|   \\
				&\leq k \max\left(\frac{N^3}{\beta^2} , \frac{N}{\beta}  \right) \|X-Y\|_1 \le \max\left(\frac{N^4}{\beta^2} , \frac{N^2}{\beta}  \right) \|X-Y\|_1,
\end{align*}
concluding the proof. 
\end{proof}
We deduce from Lemma \ref{lemma:Lip} together with \eqref{eq:systminres0} and \eqref{eq:systminresi}, 
that for any $i\in \llbracket 0,N \rrbracket$, $F_i$ is a {finite} sum of bounded, Lipchitz continuous mappings, in turn rendering $F$ Lipschitz continuous on 
$G_{\beta}$. We conclude, using the Cauchy Lipschitz theorem, that the system \eqref{eq:systminres} admits a unique continuous $[0,1]^N$-valued 
solution $X$ on the interval $[0,T_\beta]$. This is equivalent to saying that there exists a unique solution to the system \eqref{eq:limhydro} in 
$\mathbb C\left([0,T_\beta],\bar\M\right)$. Therefore, we get that $\xi=\zeta$ on $[0,T_\beta]$ for any $\beta>0$, and we conclude as in the proof of 
Corollary \ref{cor:MainGRE}.  
\end{proof}

\begin{remark}\rm 
The formulation of the uniqueness problem using a representation on $\R^{N+1}$ is reminiscent of the differential equation method in its standard presentation 
(see e.g. Wormald \cite{WCours}). The characterization of the limiting dynamics boils down to checking a Lipschitz condition over a finite number of simpler functions. This is why the classical differential equation method can be portrayed as a restriction of our method to smaller spaces, such as spaces of finite-support measures. 
Observe in particular that, in this restricted case, it is easy to adapt to the {\sc greedy} criterion, the argument that we have developed here regarding the {\sc uni-min} criterion, thereby obtaining a simpler proof than that of Corollary \ref{cor:MainGRE}. 
\end{remark}

\section{{\sc uni-max} for bounded degrees} 
\label{sec:unimax}
It is immediate to observe that similar arguments as those of Section \ref{sec:unimin} can be applied to the criterion $\bphi=\textsc{uni-max}$, defined in Example \ref{ex:unimax}. All the arguments applied above can be transposed to the present case, by setting now instead of \eqref{eq:K'unimin}, for all $\mu$ having non-zero first moment and for all $k,k'\in\N_+$, 
\begin{equation*}
\mathcal K'_{\textsc{uni-max}}(\mu,k)(k') = \left(\sum_{j=1}^{k'} {j\mu(j)\over \cro{\mu,\chi}}\right)^k-\left(\sum_{j=1}^{k'-1} {j\mu(j)\over \cro{\mu,\chi}}\right)^k
= F_\mu(k')^k-F_\mu(k'-1)^k,
\end{equation*}
where we recall \eqref{eq:defF}. In view of \eqref{eq:defLbarMIN}, we obtain the following result, 
\begin{corollary}[Convergence for the \textsc{uni-max} criterion]
\label{cor:MainMAX}
If $\mathbf\Phi=\textsc{uni-max}$, if the sequence of initial degree distributions $\suite{\mu^n_0}$ are supported in $\llbracket 0,N \rrbracket$ for some 
$N\in \N_+$ and satisfy Assumption \ref{hypo:Ho} for some measure $\nu$, then for all $f\in \maC_b$ we get the convergence 
\begin{equation*}
\sup_{t \leq 1 } |\cro{\bar{\mu}^n_t,f} -  \cro{\bar{\mu}^*_t,f}| \xrightarrow{(n,\P)} 0,
\end{equation*}
where $\procun{\bar{\mu}_t^*}$ is the unique solution of \eqref{eq:limhydro}, for $\bar{\mathfrak{L}}$ defined, for all $f\in\maC_b$ and all measures 
$\bar\mu\in\bar\M$ supported in $\llbracket 0,N \rrbracket$,  by 
\begin{multline*}
\label{eq:defLbarMIN}
\bar{\mathfrak{L}}_f(\bar\mu)
=-\Biggl\{{\cro{\bar\mu,f\mathbb 1_{\N_+}}\over \cro{\bar\mu,\mathbb 1_{\N_+}}} + {\sum_{k\in\N_+}\bar\mu(k)\sum_{k'\in\N_+}f(k')\left\{F_{\bar\mu}(k')^k-F_{\bar\mu}(k'-1)^k\right\}\over\cro{\bar\mu,\mathbb 1_{\N_+}}}\\
\left.+ {\cro{\bar\mu,\chi\nabla f} \over \cro{\bar\mu,\chi}}\left\{{\cro{\bar\mu,(\chi- \mathbb 1_{\N_+})}\over \cro{\bar\mu,\mathbb 1_{\N_+}}}
+{\sum_{k\in\N_+}\bar\mu(k)\sum_{k'\in\N_+}\left(k'-1\right)\left\{F_{\bar\mu}(k')^k-F_{\bar\mu}(k'-1)^k\right\}\over\cro{\bar\mu,\mathbb 1_{\N_+}}}\right\} \right\}
\mathbb 1_{\cro{\bar\mu,\mathbf 1_{\N_+}}>0},
\end{multline*}
where $F_{\bar\mu}$ is the c.d.f of the size-biased distribution associated to $\bar\mu$. In particular, the sequence of matching coverages satisfies 
\[{\mathbf {M}}^n_{\textsc{uni-max}}(\nu) \xrightarrow{(n,\P)} {\mathbf {M}}_{\textsc{uni-max}}(\nu) := 1 - \bar{\mu}^*_1 (\{0\}).\]
\end{corollary} 

\subsection*{Acknowledgements}
The authors would like to warmly thank Flavio Mari--Moyal for his precious help in the numerical examples of Section \ref{subsec:mainexamples}. 

\bibliographystyle{plain} 
\bibliography{bibli}

\end{document}